\documentclass[reqno]{amsart}%
\usepackage{palatino, mathpazo}
\usepackage{amsfonts}
\usepackage{tikz,amsmath}
\usepackage{amssymb,latexsym,xcolor}
\usepackage{graphicx}
\usepackage{caption}
\usepackage{harpoon}
\usepackage[mathscr]{eucal}
\usepackage{amssymb}%
\usepackage[
linktocpage=true,colorlinks,citecolor=magenta,linkcolor=blue,urlcolor=magenta]{hyperref}
\providecommand{\U}[1]{\protect \rule{.1in}{.1in}}

\newtheorem{theorem}{Theorem}[section]

\newtheorem{corollary}[theorem]{Corollary}

\newtheorem{lemma}[theorem]{Lemma}
\newtheorem{proposition}[theorem]{Proposition}
\newtheorem{Theorem}{Theorem}

\theoremstyle{remark}
\newtheorem{remark}[theorem]{Remark}

\numberwithin{equation}{section}

\begin{document}
\title[Geometric analysis on rhombus torus]{Geometric analysis on rhombus torus: Green function with two singularities}
\author{Zhijie Chen}
\address{Department of Mathematical Sciences, Yau Mathematical Sciences Center,
Tsinghua University, Beijing, 100084, China }
\email{zjchen2016@tsinghua.edu.cn}
\author{Erjuan Fu}
\address{Beijing Institute of Mathematical Sciences and Applications, Beijing, 101408, China}
\email{fej.2010@tsinghua.org.cn, ejfu@bimsa.cn}
\author{Chang-Shou Lin}
\address{Department of Mathematics, National Taiwan University, Taipei 10617, Taiwan }
\email{cslin@math.ntu.edu.tw}
\author{Zhen Song}
\address{Department of Mathematics, Northeastern University, Shenyang 110004, China}
\email{songzhen@mail.neu.edu.cn}

\begin{abstract}
Let $G(z)$ be the Green function on the flat torus $E_{\tau}=\mathbb{C}/(\mathbb{Z}+\mathbb{Z}\tau)$ with the singularity at $0$. Lin and Wang (Ann. Math. 2010) proved that $G(z)$ has at most one pair of nontrivial critical points. 

This is the third of a series of papers to study the sum of two Green functions which can be reduced to $G_p(z):=\frac12(G(z+p)+G(z-p))$. We study how the geometry of the torus and the location of singularities $\pm p$ affect the structure of critical points of $G_p(z)$. In Part I \cite{CFL}, we proved that $G_p(z)$ has at most three pairs of nontrivial critical points for all tori. In Part II \cite{CFL-II} (Proc. Lond. Math. Soc. 2026), we studied the important case that $E_{\tau}$ is a rectangular torus.  
In this paper, first we prove that if $G_p(z)$ has three pairs of nontrivial critical points, then critical points are all non-degenerate. Secondly, we study the other important but more challenging case that $E_{\tau}$ is a rhombus torus, by developing different approaches from \cite{CFL, CFL-II}. As applications, we show that the curvature equation $\Delta u+e^{u}=4\pi(\delta_p+\delta_{-p})$ on $E_{\tau}$ has exactly either $0$, $1$ or $2$ even axisymmetric solutions and each number really occurs.
\end{abstract}


\maketitle

\section{Introduction}

Let $\tau \in \mathbb{H}=\left \{  \tau\in\mathbb C|\operatorname{Im}\tau>0\right \}$, $\Lambda_{\tau}=\mathbb{Z}+\mathbb{Z}\tau$, and denote
$$\omega_{0}=0,\quad\omega_{1}=1,\quad\omega_{2}=\tau,\quad\omega_{3}=1+\tau.$$Let $E_{\tau}:=\mathbb{C}/\Lambda_{\tau}$ be a flat torus in the
plane and $E_{\tau}[2]:=\{ \frac{\omega_{k}}{2}|k=0,1,2,3\}+\Lambda
_{\tau}$ be the set consisting of the lattice points and half periods
in $E_{\tau}$.  
The Green function $G(z,w)=G(z,w;\tau)$ of the flat torus $E_{\tau}$ is the unique solution of
\[
-\Delta_z G(z, w)=\delta_{w}-\frac{1}{\left \vert E_{\tau}\right \vert }\text{
\ on }E_{\tau},\quad
\int_{E_{\tau}}G(z,w)dxdy=0,
\]
where we use the complex variable $z=x+iy$, $\Delta_z=\frac{\partial^2}{\partial x^2}+\frac{\partial^2}{\partial y^2}=4\partial^2_{\bar z z}$ is the Laplace operator, $\delta_{w}$ is the Dirac measure at $w$ and $\vert E_{\tau}\vert$ is
the area of the torus $E_{\tau}$. By the translation invariance, we have $G(z,w)=G(z-w,0)$ and it is enough to consider the Green function
$$G(z;\tau)=G(z):=G(z,0).$$
Clearly $G(z)$ is an even function on $E_{\tau}$ with the
singularity at $0$, so $\frac{\omega_k}{2}$, $k\in \{1,2,3\}$, are always critical points of $G(z)$. 

\medskip

\noindent{\bf Definition.} {\it A critical point $a\in E_{\tau}$ of $G$ (resp. $G_p$; see below) is called trivial if $a=-a$ in $E_{\tau}$, i.e., $a\in E_{\tau}[2]$.  A critical point $a\in E_{\tau}$ is called nontrivial if $a\neq-a$ in $E_{\tau}$, i.e., $a\notin E_{\tau}[2]$.}
\medskip

Nontrivial critical points must appear in pairs if exist.
 Lin and Wang \cite{LW} proved the following remarkable result about critical points of $G(z)$.
 
\begin{Theorem} \cite{LW} \label{thm-0LW} $G(z)$ has at most one pair of nontrivial critical points, or equivalently, $G(z)$ has either $3$ or $5$ critical points (depends on the choice of $\tau$).
 For example, 
\begin{itemize}
\item[(1)] When $\tau=ib$ with $b>0$, i.e., $E_{\tau}$ is a rectangular torus,  $G(z)$ has exactly $3$ critical points  $\frac{\omega_k}{2}$, $k\in \{1,2,3\}$. Furthermore, $\frac{\omega_1}{2}$ and $\frac{\omega_2}{2}$ are both non-degenerate saddle points, while $\frac{\omega_3}{2}$ is a non-degenerate minimal point.
\item[(2)] When $\tau=\frac{1}{2}+ib$ with $b>0$, i.e., $E_{\tau}$ is a rhombus torus, there are $0<b_0<\frac12<b_1<\sqrt{3}/2$ such that $G(z)$ has exactly $5$ critical points if and only if $b\in (0, b_0)\cup (b_1,+\infty)$. 
\end{itemize}
 \end{Theorem}
 
See also Bergweiler-Eremenko \cite{BE} for a new proof of the first statement of Theorem \ref{thm-0LW}. Later, Lin-Wang \cite{LW4} proved that nontrivial critical points of $G(z)$ must be minimal points,
and we \cite{CFL} proved that nontrivial critical points of $G(z)$ are always non-degenerate. 
It is known that critical points of Green functions have many important applications, such as in constructing minimal surfaces (see \cite{CW-TAMS,CT-SIAM}), in constructing bubbling solutions via the reduction method and in proving the local uniqueness and non-degeneracy of bubbling solutions for elliptic PDEs; see e.g. \cite{BKLY, BYZ, CLW, CL-2, GGLY, LW4, LY} and the references therein.

Our original goal of this series of papers is to generalize Theorem \ref{thm-0LW} to the sum of two Green functions $G(z-p_1)+G(z-p_2)$. By changing variable $z\mapsto z+\frac{p_1+p_2}{2}$, we can always assume $p_2=-p_1$, so it is enough to study the Green function
 $G_p(z)=G_{-p}(z)$ defined by
\begin{equation}G_p(z;\tau)=G_p(z):=\frac12\big(G(z-p)+G(z+p)\big).\end{equation}
We want to study how the geometry of the flat torus and the location of singularities $\pm p$ affect the structure of critical points of $G_p(z)$. 
Remark that critical points of $G_p(z)$ have interesting applications to Painlev\'{e} VI equations and curvature equations; see \cite{CFL, CFL-II} for details.

\subsection{Known results}

If $p\in E_{\tau}[2]$, i.e., $p=\frac{\omega_k}{2}$ for some $k$, then $G_p(z)=G(z-p)$ and so Theorem \ref{thm-0LW} implies that $G_p(z)$ has either $3$ or $5$ critical points. 
Thus, we only consider the case $p\in E_{\tau}\setminus E_{\tau}[2]$. Clearly $G_p(z)$ is also even, so $\frac{\omega_k}{2}$, $k=0,1,2,3$, are all trivial critical points of $G_p(z)$, namely the number of critical points of $G_p(z)$ is an even number of at least $4$. 
We generalized the first statement of Theorem \ref{thm-0LW} to $G_p(z)$ in Part I \cite{CFL}.

\begin{Theorem}\cite{CFL}\label{main-thm-1} For any $p\in E_{\tau}\setminus E_{\tau}[2]$, $G_p(z)$ has at most $3$ pairs of nontrivial critical points, or equivalently, the number of critical points of $G_p(z)$ belongs to $\{4,6,8,10\}$, and each number in $\{4,6,8,10\}$ really occurs for different $(\tau, p)$'s. 

Moreover,  fix any $\tau$, then for almost all $p\in E_{\tau}\setminus E_{\tau}[2]$, all critical points of $G_p(z)$ are non-degenerate.
\end{Theorem}

Recall that $\wp(z)=\wp( z;\tau)$ is the Weierstrass $\wp$-function with periods
$\Lambda_{\tau}$, defined by%
\[\wp(z):=\frac{1}{z^{2}}+\sum_{\omega \in \Lambda_{\tau
}\backslash\{0\}  }\left(  \frac{1}{(z-\omega)^{2}}-\frac
{1}{\omega^{2}}\right) ,
\]
which satisfies the well-known cubic equation
\[\wp^{\prime}(z)^{2}=4\wp(z)^{3}-g_{2}\wp
(z)-g_{3}=4\prod_{k=1}^3(\wp(z)-e_k),
\]
where $g_2=g_2(\tau), g_3=g_3(\tau)$ are known as invariants of the elliptic curve, and $$e_k(\tau)=e_k:=\wp\Big(\frac{\omega_k}{2}\Big),\qquad k=1,2,3.$$ It is well known that $\wp(\cdot): E_{\tau}\to \mathbb{C}\cup\{\infty\}$ is a double cover with branch points at $\{\frac{\omega_k}{2}\}_{k=0}^3$, i.e., for any $c\in \mathbb{C}\cup\{\infty\}$, there is a unique pair $\pm z_c\in E_{\tau}$ such that $\wp(\pm z_{c})=c$. Thus $$p\in E_{\tau}\setminus E_{\tau}[2]\;\text{ is equivalent to }\;\wp(p)\in\mathbb C\setminus\{e_1, e_2, e_3\}.$$ 

In Part II \cite{CFL-II}, we studied the important case $\tau=ib$ with $b>0$, that is, $E_{\tau}$ is a rectangular torus.
The main result of Part II \cite{CFL-II} is as follows, which gives a complete answer for $\wp(p)\in\mathbb{R}$.

\begin{Theorem}\cite{CFL-II}\label{II-thm}
Let $\tau=ib$ with $b>0$. Then there are $8$ real values, denoted by
$$d_1<d_2<\cdots<d_7<d_8,$$
such that the following statements hold.
\begin{itemize}
\item[(1)] $G_p(z)$ has no nontrivial critical points for
$$\wp(p)\in (-\infty, d_1]\cup [d_2, d_3]\cup [d_4, d_5]\cup [d_6, d_7]\cup [d_8,+\infty).$$
\item[(2)] $G_p(z)$ has a unique pair of nontrivial critical points $\pm a$ that are always non-degenerate for
$$\wp(p)\in (d_1, d_2)\cup (d_3, d_4)\cup (d_5, d_6)\cup (d_7, d_8).$$
Furthermore, $a=\pm\bar{a}$ in $E_{\tau}$.
\end{itemize}
\end{Theorem}
Here, for $E_{\tau}$ being a rectangular torus and $\wp(p)\in\mathbb{R}$, it was proved in \cite{CFL-II} that if $a$ is a critical point of $G_p(z)$, then so is the complex conjugate $\bar{a}$.

\subsection{New results} Clearly Theorem \ref{main-thm-1} shows that critical points of $G_p(z)$ would be degenerate for some $(\tau,p)$'s.
In this paper, first we prove the following non-degeneracy result, which solves \cite[Conjecture A]{CFL}.

\begin{theorem}\label{thm-nondegenerate} Let $p\in E_{\tau}\setminus E_{\tau}[2]$.
If $G_p(z)$ has $10$ critical points, then these critical points are all non-degenerate.
\end{theorem}

Secondly,
in this paper we study the other important case $\tau=\frac12+ib$ with $b>0$, that is, $E_{\tau}$ is a rhombus torus. One can see from Theorems \ref{III-thm1}-\ref{III-thm4} below that this case is much more complicated than the case of rectangular torus.

\begin{remark}
For $E_{\tau}$ being a rhombus torus and $\wp(p)\in\mathbb{R}$, we will prove in Lemma \ref{lemma-s5-13} that if $a$ is a critical point of $G_p(z)$, then so is the complex conjugate $\bar{a}$. Furthermore, $a=\pm \bar{a}$ in $E_{\tau}$ if and only if $\wp(a)\in\mathbb{R}$.
In other words, if $\wp(a)\notin\mathbb{R}$, then $\pm a$ and $\pm\bar a$ are two different pairs of critical points of $G_p(z)$. In this paper, we are interested in the existence of critical points satisfying $\{\pm a\}=\{\pm \bar a\}$ in $E_{\tau}$. We will see below that such critical points have applications to even axisymmetric solutions for curvation equations.

It is well-known that the two tori $E_{\tau}$ and $E_{\tilde\tau}$ are conformally equivalent as long as  $\tilde\tau=\frac{a\tau+\hat b}{c\tau+d}$ for some $\begin{pmatrix} a&\hat b
\\c&d \end{pmatrix}\in SL_2(\mathbb Z)$. Letting $\begin{pmatrix} a&\hat b\\c&d \end{pmatrix}=\begin{pmatrix} 1&-1\\2&-1 \end{pmatrix}$, we obtain $\tilde{\tau}=\frac{\tau-1}{2\tau-1}=\frac12+\frac{i}{4b}$ for $\tau=\frac12+ib$, namely $E_{\frac12+ib}$ and $E_{\frac12+\frac{i}{4b}}$ are conformally equivalent. Therefore, for simplicity, we only need to consider $\tau=\frac12+ib$ with $b\geq \frac12$ in the sequel.
\end{remark}

Let $\zeta(z)=\zeta(z;\tau):=-\int^{z}\wp(\xi)d\xi$ be the Weierstrass
zeta function with two quasi-periods
\[
\eta_{k}(\tau)=\eta_k:=2\zeta\Big(\frac{\omega_{k}}{2}\Big)=\zeta(z+\omega_{k}%
)-\zeta(z),\quad k=1,2.
\]
This $\zeta(z)$ is an odd meromorphic function with simple poles at $\Lambda_{\tau}$. We will prove in Lemma \ref{lem04-4} that there is a unique $b_2\in (\sqrt{3}/2, 6/5)$ such that
\begin{equation}\label{def-b2}
12\Big(\eta_1\Big(\frac12+ib\Big)-\frac{2\pi}{b}\Big)^2-g_2\Big(\frac12+ib\Big)=0\quad\text{if and only if}\;b=b_2.
\end{equation}
Recall $b_1\in (1/2, \sqrt{3}/2)$ in Theorem \ref{thm-0LW}-(2).
Our main results are as follows.

\begin{theorem}\label{III-thm1}
Let $\tau=\frac12+ib$ with $\frac12\leq b<b_1$. Then there exist four real values, denoted by
$$d_1<d_2<d_3<d_4,$$
which satisfies $d_2<e_1=\wp(\frac12)<d_3$,
 such that the following statements hold.
\begin{itemize}
\item[(1)] $G_p(z)$ has no nontrivial critical points satisfying $a=\pm \bar a$ for
$$\wp(p)\in (-\infty, d_1]\cup [d_2, d_3]\cup [d_4,+\infty).$$
\item[(2)] $G_p(z)$ has a unique pair of nontrivial critical points $\pm a$ satisfying $a=\pm \bar a$ for
$$\wp(p)\in (d_1, d_2)\cup (d_3, d_4).$$
\end{itemize}
\end{theorem}

\begin{theorem}\label{III-thm2}
Let $\tau=\frac12+ib_1$. Then there exist two real values $d_1<d_2$, which satisfies $d_2<e_1=\wp(\frac12)$,
such that the following statements hold.
\begin{itemize}
\item[(1)] $G_p(z)$ has no nontrivial critical points satisfying $a=\pm \bar a$ for
$$\wp(p)\in (-\infty, d_1]\cup [d_2, e_1].$$
\item[(2)] $G_p(z)$ has a unique pair of nontrivial critical points $\pm a$ satisfying $a=\pm \bar a$ for
$$\wp(p)\in (d_1, d_2)\cup (e_1, +\infty).$$
\end{itemize}
\end{theorem}

\begin{theorem}\label{III-thm3}
Let $\tau=\frac12+ib$ with $b_1< b\leq b_2$. Then there exist four real values, denoted by
$$d_1<d_2<d_3<d_4,$$
which satisfies $d_4<e_1=\wp(\frac12)$,
 such that the following statements hold.
\begin{itemize}
\item[(1)] $G_p(z)$ has no nontrivial critical points satisfying $a=\pm \bar a$ for
$$\wp(p)\in [d_1, d_2]\cup [d_3, d_4].$$
\item[(2)] $G_p(z)$ has a unique pair of nontrivial critical points $\pm a$ satisfying $a=\pm \bar a$ for
$$\wp(p)\in (-\infty, d_1)\cup (d_2, d_3)\cup (d_4,+\infty).$$
\end{itemize}
\end{theorem}

\begin{theorem}\label{III-thm4}
Let $\tau=\frac12+ib$ with $b> b_2$. Then there exist six real values, denoted by
$$d_1<d_2<d_3<d_4<d_5<d_6,$$
which satisfies $d_6<e_1=\wp(\frac12)$,
such that the following statements hold.
\begin{itemize}
\item[(1)] $G_p(z)$ has no nontrivial critical points satisfying $a=\pm \bar a$ for
$$\wp(p)\in (d_2, d_3]\cup [d_4, d_5).$$
\item[(2)] $G_p(z)$ has a unique pair of nontrivial critical points $\pm a$ satisfying $a=\pm \bar a$ for
$$\wp(p)\in (-\infty, d_1]\cup\{d_2\}\cup(d_3, d_4)\cup\{d_5\}\cup [d_6, +\infty).$$
In particular, these nontrivial critical points $\pm a$ are degenerate when $\wp(p)\in \{d_2, d_5\}$.
\item[(3)] $G_p(z)$ has exactly two pairs of nontrivial critical points satisfying $a=\pm \bar a$ for
$$\wp(p)\in (d_1, d_2)\cup(d_5, d_6).$$
\end{itemize}
\end{theorem}

In view of Theorem \ref{main-thm-1},
the above results together indicate that $G_p(z)$ can not have three pairs of nontrivial critical points satisfying $a=\pm \bar a$. 

\begin{remark}
Except $d_2, d_5$ in Theorem \ref{III-thm4}, we can write down the explicit expressions of all other $d_j$'s in Theorems \ref{III-thm1}-\ref{III-thm4}. Moreover, if $\wp(p)=d_j$ (except $d_2, d_5$ in Theorem \ref{III-thm4}), then either $0$ or $\frac12$ is a degenerate critical point of $G_p(z)$. More precisely,
\begin{itemize}
\item[(1)]
In Theorem \ref{III-thm1}, since $\frac12\leq b<b_1$, we have
$$d_1=e_1-\frac{3e_1^2-\frac{g_2}{4}}{e_1+\eta_1},\quad d_2=-\eta_1, \quad d_3=\frac{2\pi }{b}-\eta_1,\quad d_4=e_1+\frac{3e_1^2-\frac{g_2}{4}}{\frac{2\pi}{b}-e_1-\eta_1},$$
and when $\wp(p)\in \{d_1, d_4\}$ (resp. $\wp(p)\in \{d_2, d_3\}$), $\frac12$ (resp. $0$) is a degenerate critical point of $G_p(z)$. Furthermore,
$$\lim_{b\uparrow b_1}d_3=e_1, \quad \lim_{b\uparrow b_1}d_4=+\infty.$$
\item[(2)] In Theorem \ref{III-thm2}, since $e_1+\eta_1=\frac{2\pi}{b_1}$ for $b=b_1$, we have
$$d_1=e_1-\frac{3e_1^2-\frac{g_2}{4}}{\frac{2\pi}{b_1}}, \quad d_2=-\eta_1,$$
and when $\wp(p)=d_1$ (resp. $\wp(p)=d_2$), $\frac12$ (resp. $0$) is a degenerate critical point of $G_p(z)$.
\item[(3)] In Theorem \ref{III-thm3}, since $b_1<b\leq b_2$, we have
$$d_1=e_1+\frac{3e_1^2-\frac{g_2}{4}}{\frac{2\pi}{b}-e_1-\eta_1}, \quad d_2=e_1-\frac{3e_1^2-\frac{g_2}{4}}{e_1+\eta_1},\quad d_3=-\eta_1,\quad d_4=\frac{2\pi}{b}-\eta_1,$$
and when $\wp(p)\in \{d_1, d_2\}$ (resp. $\wp(p)\in \{d_3, d_4\}$), $\frac12$ (resp. $0$) is a degenerate critical point of $G_p(z)$. Furthermore,
$$\lim_{b\downarrow b_1}d_1=-\infty, \quad \lim_{b\downarrow b_1}d_4=e_1.$$
\item[(4)] In Theorem \ref{III-thm4}, since $b>b_2$, we have
$$d_1=e_1+\frac{3e_1^2-\frac{g_2}{4}}{\frac{2\pi}{b}-e_1-\eta_1}, \quad d_3=e_1-\frac{3e_1^2-\frac{g_2}{4}}{e_1+\eta_1},\quad d_4=-\eta_1,\quad d_6=\frac{2\pi}{b}-\eta_1,$$
and when $\wp(p)\in \{d_1, d_3\}$ (resp. $\wp(p)\in \{d_4, d_6\}$), $\frac12$ (resp. $0$) is a degenerate critical point of $G_p(z)$. Furthermore,
$$\lim_{b\downarrow b_2}d_2=e_1+\frac{3e_1^2-\frac{g_2}{4}}{\frac{2\pi}{b}-e_1-\eta_1}, \quad \lim_{b\downarrow b_2}d_5=\frac{2\pi}{b}-\eta_1.$$
\end{itemize}
See Sections \ref{sec-4}-\ref{sec-6} for the proofs.
\end{remark}

A consequence of Theorems \ref{III-thm3}-\ref{III-thm4} is

\begin{corollary}\label{III-coro}
Let $\tau=\frac12+ib$ with $b>b_1$ and $p\in (0, \frac12)$. Then  $G_p(z)$ has a unique pair of nontrivial critical points $\pm a$ satisfying $a=\pm \bar a$. Furthermore, $\pm a$ are non-degenerate except for at most finite $p$'s.
\end{corollary}

The approach of proving Theorems \ref{III-thm1}-\ref{III-thm4} are quite different from that of proving Theorem \ref{II-thm} in Part II \cite{CFL-II}. As introduced in \cite{CFL-II}, there is a deep connection between conditional stability sets $\sigma_j$'s of a second order linear ODE (see \eqref{GLE} below)
and nontrivial critical points of $G_p(z)$. It was proved in \cite{CFL-II} that $\sigma_j$ consists of at most $5$ analytic curves with common endpoints $\{A_0, A_1, A_2, A_3\}$, and the number of pairs of nontrivial critical points of $G_p(z)$ equals to $$\#(\sigma_1\cap\sigma_j\setminus\{A_0, A_1, A_2, A_3\}),\quad j=2,3.$$
We will briefly review this theory in Section \ref{sec-2}.

When $\tau=ib$ with $b>0$ and $\wp(p)\in\mathbb{R}$, the conditional stability sets are easy to study because they admit certain symmetries and all the endpoints $A_0, A_1, A_2, A_3$ locate on the axis of symmetry. For example, if $p\in (0,\frac12)$, then $A_0, A_1, A_2, A_3\in\mathbb R$ and both $\sigma_1, \sigma_2$ are symmetric with respect to the real axis, from which we can get
$$\sigma_2=(-\infty, A_1]\cup [A_3, A_2]\cup [A_0, +\infty),$$
$$\sigma_1=\sigma_{1,\infty}\cup[A_1, A_3]\cup [A_2, A_0],$$ where $\sigma_{1,\infty}$ is an unbounded simple curve with $\sigma_{1,\infty}\cap\mathbb{R}$ containing a single point. From here we immediately obtain 
 $$\#(\sigma_1\cap\sigma_2\setminus\{A_0, A_1, A_2, A_3\})\leq 1$$
and so $G_p(z)$ has at most one pair of nontrivial critical points. See \cite{CFL-II} for details.

For the rhombus torus case studied in this paper, the conditional stability sets $\sigma_1$ and $\sigma_3$ still admit certain symmetries. However, compared with the rectangular torus case studied in \cite{CFL-II}, an additional difficulty is that only two of $\{A_0, A_1, A_2, A_3\}$ locate on the axis of symmetry, so the graphs of $\sigma_j$'s might have multiple possibilities, namely we can not determine the graphs of $\sigma_j$'s precisely; see Section \ref{sec-3} for details. Thus, it is too difficult to compute $\#(\sigma_1\cap\sigma_3\setminus\{A_0, A_1, A_2, A_3\})$.

The new approach of this paper is to apply the conditional stability sets to study the map $f: E_{\tau}\setminus E_{\tau}[2]\to \mathbb{C}\cup\{\infty\}$, defined by
\begin{align}\label{00513-1-0} 
f(a):=\wp (a)+\frac{\wp ^{\prime }(a)}{%
2(\zeta(a)-r\eta_1-s\eta_2)},
\end{align}
where we write $a=r+s\tau$ with $r,s\in\mathbb{R}^2\setminus\frac12\mathbb{Z}^2$. Here $\frac12\mathbb{Z}^2=\{(r,s)\,:\, 2r,2s\in\mathbb Z\}$. The connection is that $\pm a$ is a pair of nontrivial critical points of $G_p(z)$ if $\wp(p)=f(a)$.
This map is related to Hitchin's formula \eqref{513-1-0} that was introduced by Hitchin \cite{Hit1} in his study of Einstein metrics.
Note that $f(a)$ is not a meromorphic function of $a$. We will prove Theorems \ref{III-thm1}-\ref{III-thm4} by analysing the image of $f(a)$ for $\wp(a)\in\mathbb{R}$, with the help of the conditional stability sets.  See Sections \ref{sec-4}-\ref{sec-6}.

\subsection{Applications to the curvature equation}
As applications, we study the following curvature equation 
\begin{equation}\label{mfe}
\Delta u+e^{u}=4\pi(\delta_{p}+\delta_{-p})\quad\text{ on
}\; E_{\tau}. 
\end{equation}
Equation \eqref{mfe} arises from conformal geometry and mathematical physics. Geometrically, a solution $u$ of \eqref{mfe} leads to a spherical metric
$ds^{2}=\frac{1}{2}e^{u}dz^2$ with constant Gaussian curvature $+1$
acquiring conic singularities at $\pm p$. It also
appears in statistical physics as the equation for the mean field
limit of the Euler flow in Onsager's vortex model, hence
also called a mean field equation. 
We refer the readers to \cite{BKLY,BT,BYZ,CLW,CL-2,FSX,LW,LY,MR,WZ,WWX} and the references therein for recent developments on mean field equations and related topics.

Let $\tau=ib$ or $\tau=\frac12+ib$ with $b>0$. Then it is easy to prove that $\wp(p)\in\mathbb{R}$ if and only if $\bar p=\pm p$ in $E_{\tau}$; see Section \ref{sec-3}. In this case, we  see that if $u(z)$ is a solution of \eqref{mfe}, then so are $u(-z)$ and $u(\bar z)$ (here we use complex variable $z=x+iy$). 
A solution $u(z)$ is called \emph{even axisymmetric} if it is symmetric with respect to both the origin and the $x,y$-axis, i.e.,  $$u(z)=u(-z)=u(\bar{z}).$$
In Part II \cite{CFL-II}, we showed that \eqref{mfe} on the rectangular torus has at most one even axisymmetric solution.
Here we can apply Theorems \ref{III-thm1}-\ref{III-thm4} to obtain the following results, which say that  \eqref{mfe} on the rhombus torus could have two even axisymmetric solutions. 

\begin{theorem}\label{III-thm1-1}
Let $\tau=\frac12+ib$ with $\frac12\leq b<b_1$, and 
$d_1<d_2<d_3<d_4$ be given in Theorem \ref{III-thm1}.
Then the following statements hold.
\begin{itemize}
\item[(1)] Equation \eqref{mfe} has no even axisymmetric solutions for
$$\wp(p)\in (-\infty, d_1]\cup [d_2, d_3]\cup [d_4,+\infty).$$
\item[(2)] Equation \eqref{mfe} has a unique even axisymmetric solution for
$$\wp(p)\in (d_1, d_2)\cup (d_3, d_4).$$
\end{itemize}
\end{theorem}

\begin{theorem}\label{III-thm2-1}
Let $\tau=\frac12+ib_1$ and $d_1<d_2$ be given in Theorem \ref{III-thm2}.
Then the following statements hold.
\begin{itemize}
\item[(1)] Equation \eqref{mfe} has no even axisymmetric solutions for
$$\wp(p)\in (-\infty, d_1]\cup [d_2, e_1].$$
\item[(2)] Equation \eqref{mfe} has a unique even axisymmetric solution for
$$\wp(p)\in (d_1, d_2)\cup (e_1, +\infty).$$
\end{itemize}
\end{theorem}

\begin{theorem}\label{III-thm3-1}
Let $\tau=\frac12+ib$ with $b_1< b\leq b_2$, and $d_1<d_2<d_3<d_4$ be given in Theorem \ref{III-thm3}.
Then the following statements hold.
\begin{itemize}
\item[(1)] Equation \eqref{mfe} has no even axisymmetric solutions for
$$\wp(p)\in [d_1, d_2]\cup [d_3, d_4].$$
\item[(2)] Equation \eqref{mfe} has a unique even axisymmetric solution for
$$\wp(p)\in (-\infty, d_1)\cup (d_2, d_3)\cup (d_4,+\infty).$$
\end{itemize}
\end{theorem}

\begin{theorem}\label{III-thm4-1}
Let $\tau=\frac12+ib$ with $b> b_2$, and
$d_1<d_2<d_3<d_4<d_5<d_6$ be given in Theorem \ref{III-thm4}.
Then the following statements hold.
\begin{itemize}
\item[(1)] Equation \eqref{mfe} has no even axisymmetric solutions for
$$\wp(p)\in (d_2, d_3]\cup [d_4, d_5).$$
\item[(2)] Equation \eqref{mfe} has a unique even axisymmetric solution for
$$\wp(p)\in (-\infty, d_1]\cup\{d_2\}\cup(d_3, d_4)\cup\{d_5\}\cup [d_6, +\infty).$$
\item[(3)] Non-uniqueness: Equation \eqref{mfe} has exactly two even axisymmetric solutions for
$$\wp(p)\in (d_1, d_2)\cup(d_5, d_6).$$
\end{itemize}
\end{theorem}

\begin{proof}
We proved in Part I \cite[Theorem 1.13]{CFL} that there is a one-to-one correspondence between pairs of nontrivial critical points of $G_p(z)$ and even solutions of \eqref{mfe}. For $\tau=\frac12+ib$ and $\wp(p)\in\mathbb R$, we mentioned before that if $a$ is a critical point of $G_p(z)$, then so is $\bar a$; if $u(z)$ is an even solution of \eqref{mfe}, then so is $u(\bar z)$.
Furthermore, if the pair of nontrivial critical points $\pm a$ corresponds to the even solution $u(z)$, then it follows from the proof of \cite[Theorem 1.13]{CFL} that $\pm \bar a$ corresponds to $u(\bar z)$. Thus, $\{\pm a\}=\{\pm \bar a\}$ in $E_{\tau}$ if and only if $u(z)=u(\bar z)$, i.e., $u(z)$ is even axisymmetric. In other words, there is a one-to-one correspondence between pairs of nontrivial critical points of $G_p(z)$ satisfying $a=\pm\bar a$ in $E_{\tau}$ and even axisymmetric solutions of \eqref{mfe}.
Therefore, these results follows directly from Theorems \ref{III-thm1}-\ref{III-thm4}. 
\end{proof}

The rest of this paper is organized as follows. 
We prove Theorem \ref{thm-nondegenerate} in Section \ref{sec02}.
In Section \ref{sec-2}, we briefly review the basic theory of the second order linear ODE \eqref{GLE}, including the conditional stability sets and their connection with nontrivial critical points of $G_p(z)$. We consider $\tau=\frac12+ib$ with $b>0$ and $\wp(p)\in\mathbb R$ from Section \ref{sec-3}. In Section \ref{sec-3}, first we give a more precise characterization of the conditional stability sets, and then we study the basic property of the map $f(a)$. In Sections \ref{sec-4}-\ref{sec-6}, we use the conditional stability sets to study the image of $f(a)$ for $\wp(a)\in\mathbb {R}$ and prove our main results Theorems \ref{III-thm1}-\ref{III-thm4}.  In Section \ref{sec-7}, we give a complete characterization of the conditional stability sets for $p\in (0, \frac12)$ by using Theorems \ref{III-thm1}-\ref{III-thm4}, which can be used to prove Corollary \ref{III-coro} concerning the non-degeneracy of the nontrivial critical points.

\section{Proof of Theorem \ref{thm-nondegenerate}}
\label{sec02}

This section is devoted to the proof of Theorem \ref{thm-nondegenerate}. First, we recall the basic function $\phi$ from Part I \cite[Section 2]{CFL}.
By using the complex variable $z\in\mathbb C$ and write $z=r+s\tau$ with $r,s\in \mathbb{R}$, it was proved in \cite{LW} that
\begin{equation}
\label{G_z}-4\pi \frac{\partial G}{\partial z}(z)=\zeta(z)-r\eta_{1}
-s\eta_{2}.
\end{equation}
Consequently, $z=r+s\tau$ with $r,s\in\mathbb R$ is a critical point of $G_p$ if and only if 
\begin{equation}\label{00a+1p}\zeta(z+p)+\zeta(z-p)-2(r\eta_1+s\eta_2)=0.\end{equation}
By the Legendre relation $\tau\eta_1-\eta_2=2\pi i$, we easily obtain that $$\label{rsab}-r\eta_1-s\eta_2=az+b\bar{z},\quad\text{where }a=\frac{\pi}{\operatorname{Im}\tau}-\eta_1,\quad b=-\frac{\pi}{\operatorname{Im}\tau}<0.$$
Inserting this into \eqref{00a+1p}, we obtain
$$\zeta(z+p)+\zeta(z-p)+2(az+b\bar{z})=0,$$
or equivalently,
\begin{equation}\label{00a+2p}
z=-\frac{1}{2b}\left(\overline{\zeta(z+p)}+\overline{\zeta(z-p)}+2\bar{a}\bar{z}\right)=:g(z).
\end{equation}
In conclusion, $z$ is a critical point of $G_p(z)$ if and only if $g(z)=z$, i.e., $z$ is a fixed point of $g$.
Clearly $g(z)$ is an anti-meromorphic function in the plane, and a direct computation shows that
\begin{align}\label{g-w}
g(z+\omega)=g(z)+\omega,\quad\forall\, \omega\in \Lambda_{\tau}=\mathbb Z+\mathbb Z\tau.
\end{align}
Therefore, the map $\phi: E_{\tau}\setminus\{\pm p\}\to \mathbb C$, defined by $$z\mapsto \phi(z):=z-g(z),$$ is a well-defined harmonic map, and $\phi^{-1}(0)$ is precisely the set of critical points of $G_p(z)$. It follows from the proof of \cite[Theorem 1.3]{CFL} that
\begin{equation}\label{abss}
\#\phi^{-1}(w)\leq 10,\quad\forall\, w\in\mathbb C.
\end{equation}
We recall
the following result from \cite{BE-2010}.

\begin{proposition}\cite[Proposition 3]{BE-2010}\label{prop003}
Let $f: D\to \mathbb C$ be a harmonic map defined in a region $D$ in $\mathbb C$. Suppose that there exists $m\in\mathbb{N}$ such that every $w\in\mathbb C$ has at most $m$ preimages. Then 
\begin{itemize}
\item[(1)]
The set of points which have $m$ preimages is open. 
\item[(2)]
If $w_0\in \mathbb C$ satisfies $\#f^{-1}(w_0)=m$, then for any $z_j\in f^{-1}(w_0)$, $f$ is locally surjective at $z_j$ (namely, the image of every neighborhood of $z_j$ contains a neighborhood of $w_0$).
\end{itemize}
\end{proposition}

Remark that the original version of \cite[Proposition 3]{BE-2010} contains only the statement (1), and the statement (2) is contained in the proof of the statement (1); see \cite[p.321]{BE-2010}. 
We also need the following Lewy's non-vanishing theorem \cite{Lewy}.

\begin{theorem}\label{thm-Lewy}\cite{Lewy}
If $u(x,y)$ and $v(x,y)$ are harmonic, $u(0,0)=v(0,0)=0$, and if there exists a neighborhood $N_1$ of the origin of the $xy$ plane and a neighborhood $N_2$ of the origin of the $uv$ plane such that $(u, v): N_1\to N_2$ is bijective. Then the Jacobian $\partial (u,v)/\partial (x,y)$ does not vanish at the origin. 
\end{theorem}

Now we are ready to prove Theorem \ref{thm-nondegenerate}.

\begin{proof}[Proof of Theorem \ref{thm-nondegenerate}]
Suppose $G_p(z)$ has $10$ critical points $q_1, \cdots, q_{10}$. Then $\phi(q_j)=0$.
Take pairwise disjoint neighborhoods
$U_j\subset E_{\tau}\setminus\{\pm p\}$ of $q_j$.
It follows from \eqref{abss} and Proposition \ref{prop003} that $\phi$ is locally surjective at $q_j$ for each $j$, so there is $\varepsilon>0$ such that
$$
  B_\varepsilon(0):=\{z\in\mathbb{C} \;:\;|z|<\varepsilon\}\subset \phi(U_j),
  \qquad j=1,\cdots,10.
$$
By the continuity at $q_j$, we get that
\begin{equation}\label{eqeee}
 V_j:=\phi^{-1}(B_\varepsilon(0))\cap U_j
\end{equation}
is a neighborhood of $q_j$.
We claim that  $\phi|_{V_j}: V_j\to B_\varepsilon(0)$ is injective for all $j$. Suppose for some $j$, say $j=1$ for example, that  $\phi|_{V_1}: V_1\to B_\varepsilon(0)$ is not injective, then there are 
$z_1\neq z_2\in V_1$ and $w'\in B_\varepsilon(0)$ such that $\phi(z_1)=\phi(z_2)=w'$. Note from
\eqref{eqeee} that $V_j\subset U_j$ contains at least one preimage of $w'$
for $2\leq j\leq 10$, and since $U_j$ are disjoint, we get
\[
  \#\phi^{-1}(w')\geq 2+(10-1)=11,
\]
a contradiction with \eqref{abss}.  

Therefore, $\phi|_{V_j}: V_j\to B_\varepsilon(0)$ is bijective for all $j$.  Then by Lewy's Theorem \ref{thm-Lewy},  we get the Jacobian
$J_\phi(q_j)\neq0$ for every $j$. Finally, it follows from \cite[(3.1)]{CFL} that the Hessian $\det D^2G_p(q_j)$ satisfies
$$\det D^2G_p(q_j)
  =\frac{1}{4(\operatorname{Im}\tau)^2}\,J_\phi(q_j)\neq 0,\quad \forall j=1,\cdots, 10,$$
namely all critical points of $G_p(z)$ are non-degenerate.
\end{proof}

\section{Preliminaries}
\label{sec-2}

\subsection{Generalized Lam\'{e} equation}
In this subsection, we briefly review the connection between the Green function $G_p(z)$ and
 the following second order generalized Lam\'{e} equation (GLE for short)
\begin{align}\label{GLE}
y''(z)=I(z; p, A, \tau)y(z),\quad z\in\mathbb{C},
\end{align}
with the potential given by
\begin{align}\label{potential-I}
I(z; p, A, \tau):=&\frac{3}{4}
\big(\wp(z+p)+\wp(z-p)-\wp(2p)\big)\nonumber\\
&+A\big(\zeta(z+p)-\zeta(z-p)-\zeta(2p)\big)+A^{2},
\end{align}
where $p\in E_{\tau}\setminus E_{\tau}[2]$ and $A\in \mathbb{C}$. See Part II \cite[Section 2]{CFL-II} for more details.


GLE \eqref{GLE}  is of Fuchsian type with regular singularities at $\pm p$.
The local exponents of GLE (\ref{GLE}) at $\pm p$ are $-\frac12$, $\frac32$, so solutions might have logarithmic singularities at $\pm p$.
It was proved in \cite{CKL-JMPA2016} that $\pm p$ are always apparent singularities, namely all solutions of \eqref{GLE} are free of
logarithmic singularities at $\pm p$. 
Consequently, the local monodromy matrix at $\pm p$ is $-I_2$, where $I_2$ is the identity matrix. Denote by $L_p\subset\{t_1+t_2\tau : t_1,t_2\in[-\frac12,\frac12]\}$ the straight segment crossing $0$ and connecting $\pm p$. Then by analytic continuation, any solution $y(z)$ of GLE (\ref{GLE}) can be viewed as a single-valued
meromorphic function in $\mathbb{C}\backslash(L_p+\Lambda_{\tau})$, and in this
region $y(-z)$ and $y(z+\omega_j)$ are well-defined. Here
$L_p+\Lambda_{\tau}=\cup_{\omega\in\Lambda_{\tau}} (\omega+L_p)$ is the preimage of $L_p$ under the projection $\pi: \mathbb{C}\to E_{\tau}$.

Fix any base point $q_{0}\in E_{\tau}\backslash\{ \pm 
p\}$ such that any of $q_0+\mathbb{R}, q_0+\tau\mathbb{R}, q_0+(2\tau-1)\mathbb{R}$ has no intersection with $\pm p+\Lambda_{\tau}$. The monodromy representation of GLE \eqref{GLE} is a group homomorphism $\rho:\pi_{1}(  E_{\tau}\backslash\{\pm p\},q_{0})  \rightarrow
SL(2,\mathbb{C})$ defined as follows. Take any basis of solutions $(y_1(z), y_2(z))$ of GLE  \eqref{GLE}. For any loop $\gamma\in \pi_{1}(  E_{\tau}\backslash
\{\pm p\},q_{0})$, let $\gamma^*y(z)$ denote the analytic continuation of $y(z)$ along $\gamma$. Then $\gamma^*(y_1(z), y_2(z))$ is also a basis of solutions, so there is a matrix $\rho(\gamma)\in SL(2,\mathbb{C})$ such that
$$\gamma^*\begin{pmatrix}
y_{1}(z)\\
y_{2}(z)
\end{pmatrix}
=\rho(\gamma)
\begin{pmatrix}
y_{1}(z)\\
y_{2}(z)
\end{pmatrix}.$$
Here $\rho(\gamma)\in SL(2,\mathbb{C})$ (i.e., $\det \rho(\gamma)=1$) follows from the fact that the Wronskian $y_{1}(z)y_2'(z)-y_1'(z)y_2(z)$ is a nonzero constant. The image of $\rho:\pi_{1}(  E_{\tau}\backslash\{\pm p\},q_{0})  \rightarrow
SL(2,\mathbb{C})$ is called the monodromy group of \eqref{GLE}, which is a subgroup of $SL(2,\mathbb{C})$.

Let
$\gamma_{\pm}\in \pi_{1}(  E_{\tau}%
\backslash\{  \pm p\}  ,q_{0})$ be a
simple loop encircling $\pm p$ counterclockwise respectively, and $\ell_{j}%
\in \pi_{1}(  E_{\tau}%
\backslash\{\pm p\}  ,q_{0})  $, $j=1,2$, be two
fundamental cycles of $E_{\tau}$ connecting $q_{0}$ with $q_{0}+\omega_{j}$
such that $\ell_{j}\cap (L_p+\Lambda_{\tau})=\emptyset$ and satisfies%
\begin{equation}
\gamma_{-}\gamma_{+}=\ell_{1}\ell_{2}\ell_{1}^{-1}\ell_{2}^{-1}\text{ \  \ in
}\pi_{1}\left(  E_{\tau}\backslash \left \{\pm p\right \}
,q_{0}\right)  . \label{II-iv1}%
\end{equation}
Denote by $$N_j=N_j(A):=\rho(\ell_j),\quad j=1,2.$$
Since
$
\rho(\gamma_{\pm})=-I_{2}$, 
it follows from \eqref{II-iv1} that $N_1N_2=N_2N_1$, 
and the monodromy group of (\ref{GLE}) is generated by $\{-I_{2},N_1,N_2\}$ and is abelian. So there is a common eigenfunction $y_{1}(z)$ of all
monodromy matrices. Let
$\varepsilon_{j}$ be the eigenvalue of $N_j$: $$\ell_{j}^{\ast}y_{1}(z)=\varepsilon
_{j}y_{1}(z),\quad j=1,2.$$As mentioned before, $y_{1}(z)$ can be viewed as a single-valued
meromorphic function in $\mathbb{C}\backslash(L_p+\Lambda_{\tau})$, and in this
region, $y_{1}(-z)$ and $y_1(z+\omega_j)$ are well-defined. Then it follows from $\ell_{j}\cap (L_p+\Lambda_{\tau})=\emptyset$ that
\begin{equation}
y_{1}(z+\omega_{j})=\ell_{j}^{\ast}y_{1}(z)=\varepsilon_{j}y_{1}(z),\quad j=1,2. \label{304-111}%
\end{equation}

Let $y_{2}(z)=y_{1}(-z)$ in $\mathbb{C}\backslash(L_p+\Lambda_{\tau})$. Clearly
$y_{2}(z)$ is also a solution of (\ref{GLE}) and (\ref{304-111}) implies
\begin{equation}
y_{2}(z+\omega_{j})=\ell_{j}^{\ast}y_{2}(z)=\varepsilon_{j}^{-1}%
y_{2}(z),\quad j=1,2, \label{304-12}%
\end{equation}
i.e. $y_{2}(z)$ is also a common eigenfunction with eigenvalue $\varepsilon_{j}^{-1}$.

Let $\sigma(z;\tau)=\sigma(z):=\exp(\int^{z}\zeta(\xi)d\xi)$ be the Weierstrass sigma function, which is an odd
entire function with simple zeros at $\Lambda_{\tau}$, and satisfies the following transformation law
\begin{equation}  \label{123123}
\sigma(z+\omega_j)=-e^{(z+\frac{1}{2}\omega_j)\eta_j}\sigma(z),\quad j=1,2.
\end{equation}
Then
the explicit expression of the common eigenfunction $y_1(z)$ can be written down by the theory of elliptic functions, which has been studied in \cite{CKL1}.

\begin{lemma}\cite{CKL1}\label{lemma2-3} Given $A\in\mathbb C$, let $y_{1}(z)$ be
the common eigenfunction mentioned above. Then there is $a=a(A)\neq \pm p$ in $E_{\tau}$ such that up to a nonzero
constant, 
\begin{equation}\label{exp1}
y_{1}(z)=y_{a}(z):=\frac{e^{c(a)z}\sigma(z-a)}%
{\sqrt{\sigma(z-p)\sigma(z+p)}},
\end{equation}
with
\begin{equation}
c(a)=\frac{1}{2}[\zeta
(a+p)+\zeta(a-p)],\label{61-38}
\end{equation}
\begin{equation}
A=\frac{1}{2}\left[  \zeta(p+a)+\zeta(p-a)-\zeta(2p)\right]  . \label{Aap}%
\end{equation}
\end{lemma}

Observe that in $\mathbb{C}\backslash(L_p+\Lambda_{\tau})$, $y_2(z)=y_{{a}}(-z)=-y_{-{a}}(z)$, so
$y_{-a}(z)$ is also a solution of the same GLE \eqref{GLE}. It was proved in \cite[Section 2]{CFL-II} that
 $\pm a
\operatorname{mod}\Lambda_{\tau}$ is uniquely determined for given GLE \eqref{GLE} (i.e. uniquely determined by $A$ for fixed $p, \tau$),
and for different representatives $a, \tilde{a}\in\mathbb{C}$ of the same $a_1\in E_{\tau}$ (i.e., $a-\tilde{a}\in \Lambda_{\tau}$), it follows from the transformation law \eqref{123123} that
\begin{equation}
y_{a}(z)=y_{\tilde{a}}(z)\text{ \ up to a nonzero
constant}. \label{aunique11}%
\end{equation}

Define
\[Q(A):=\prod_{k=0}^3(A-A_k),\]
where
\begin{align}\label{Ak}
A_k:=&\frac{1}{2}\left[  \zeta\Big(p+\frac{\omega_k}{2}\Big)+\zeta\Big(p-\frac{\omega_k}{2}\Big)-\zeta(2p)\right] \nonumber\\
=&\begin{cases}
-\frac{\wp''(p)}{4\wp'(p)} & k=0,\\
A_0+\frac{\wp'(p)}{2(\wp(p)-e_k)}&k=1,2,3,
\end{cases}
\end{align}
where we have used the following additional formulas of elliptic functions to obtain the equality:
$$2\zeta(z)-\zeta(2z)=-\frac{\wp''(z)}{2\wp'(z)},$$
\begin{equation}\label{eqfc-add}\zeta(z+w)+\zeta(z-w)-2\zeta(z)=\frac{\wp'(z)}{\wp(z)-\wp(w)}.\end{equation}
Then it was proved in \cite[Section 2]{CFL-II} that $\sum_{k=0}^{3}A_k=0$, and $y_a(z)$, $y_{-a}(z)$ are linearly independent, if and only if $a\neq -a$ in $E_{\tau}$, or equivalently, $a\neq \frac{\omega_k}{2}$ for any $k$, if and only if $A\neq A_k$ for any $k$, i.e., $Q(A)\neq 0$.

To avoid the branch points $\pm p$ of $y_{\pm a}(z)$, as in \cite{CKL1}, we define
\begin{equation}
\Psi_{p}(z):=\frac{\sigma(z) }{\sqrt{\sigma(z-p) \sigma (z+p)}}.
\label{61-3}
\end{equation}
By using the transformation law \eqref{123123},
we see that $\Psi_p(z)^2$ is an {elliptic function}. Since $\ell_{j}, j=1,2,$ are the two fixed fundamental
cycles of $E_{\tau}$ which do not intersect with $L_p+\Lambda_{\tau}$, it was proved in  \cite{CKL1} that
 the analytic continuation of $\Psi_p(z)$ along $\ell_{j}$ satisfies%
\begin{equation}
\ell_{j}^{\ast}\Psi_{p}(z)=\Psi_{p}(z),\quad j=1,2.   \label{304-1}
\end{equation}
Write
\begin{equation}
y_{\pm a}(z)=\tilde{y}_{\pm a}(z)\Psi_{p}(z),   \label{y-}
\end{equation}
where
$$\tilde{y}_{\pm a}(z)=e^{\pm c(a)z}\frac{\sigma(z\mp a)}%
{\sigma(z)}=e^{\frac{\pm1}{2}[\zeta
(a+p)+\zeta(a-p)]z}\frac{\sigma(z\mp a)}%
{\sigma(z)}$$
is meromorphic.
Define $(r,s)=(r(A), s(A))\in \mathbb{C}^2$ by
\begin{equation}\label{61-37-1}-2\pi i s:=c(a)-\eta_1a,\qquad 2\pi i r:=c(a)\tau-\eta_2a,
\end{equation}
or equivalently,  by using the Legendre relation $\tau\eta_1-\eta_2=2\pi i$,
\begin{equation}\label{61-37-2} \begin{cases}r+s\tau =a,\\
 r\eta_1+s\eta_2=c(a)=\frac{1}{2}[\zeta
(a+p)+\zeta(a-p)].\end{cases}
\end{equation}
Clearly we can rewrite \eqref{61-37-2} as
\begin{equation}\label{61-37-22}
\zeta
(a+p)+\zeta(a-p)-2(r\eta_1+s\eta_2)=0,\quad\text{with }\;a=r+s\tau.
\end{equation}
Then
by applying the
transformation law \eqref{123123}
to $\tilde{y}_{\pm a}(z)$, we obtain
$$\tilde{y}_{\pm a}(z+1)=e^{\mp 2\pi is}\tilde{y}_{\pm a}(z),\qquad \tilde{y}_{\pm a}(z+\tau)=e^{\pm 2\pi ir}\tilde{y}_{\pm a}(z),$$
and so it follows from \eqref{304-1} and \eqref{y-} that
\begin{equation}\label{eqfc-1}{y}_{\pm a}(z+1)=\ell_1^*{y}_{\pm a}(z)=e^{\mp 2\pi is}{y}_{\pm a}(z),\end{equation}
\begin{equation}\label{eqfc-2} {y}_{\pm a}(z+\tau)=\ell_2^*{y}_{\pm a}(z)=e^{\pm 2\pi ir}{y}_{\pm a}(z).\end{equation}

{\bf Case 1.} $Q(A)\neq 0$. Then $y_a(z)$ and $y_{-a}(z)$ are linearly independent, and it follows from \eqref{eqfc-1}-\eqref{eqfc-2} that the monodromy matrices with respect to $(y_a(z), y_{-a}(z))$ are given by
\begin{equation}
N_1=N_1(A)=%
\begin{pmatrix}
e^{-2\pi is} & 0\\
0 & e^{2\pi is}%
\end{pmatrix}
,\text{ \  \  \ }N_2=N_2(A)=%
\begin{pmatrix}
e^{2\pi ir} & 0\\
0 & e^{-2\pi ir}%
\end{pmatrix}.
\label{Mono-001}%
\end{equation}
Furthermore, $(r,s)\in \mathbb{C}^2\setminus\frac12\mathbb{Z}^2$ (see \cite{CKL1} for a proof), so
$$
(\text{tr}N_1(A),\text{tr}N_2(A))=(2\cos2\pi s,2\cos2\pi
r)\not \in \{ \pm(2,2),\pm(2,-2)\}. 
$$
Clearly,
\begin{equation} \label{eq-data5}
  \parbox{\dimexpr\linewidth-5em}{$N_1(A), N_2(A)$ are unitary matrices, i.e., the monodromy group of \eqref{GLE} is a subgroup of $SU(2)$, if and only if the corresponding $(r,s)=(r(A), s(A))\in\mathbb{R}^2\setminus\frac12\mathbb Z^2$.
  }
\end{equation}

{\bf Case 2.} $Q(A)= 0$, i.e., $A=A_k$ and $a=\frac{\omega_k}{2}$ for $k\in\{0,1,2,3\}$. Then we see from \eqref{61-37-2} that $(r,s)\in\frac12\mathbb{Z}^2$ and more precisely,
$$(r,s)\equiv\begin{cases}(0,0)& k=0\\
(\frac12,0)&k=1\\
(0,\frac12)&k=2\\
(\frac12,\frac12)&k=3\end{cases}\quad\operatorname{mod}\;\mathbb{Z}^2.$$
Denote
\begin{equation}\label{var-jjkk}\varepsilon_{j,k}=\begin{cases}1& (j,k)=(1,0),(1,1),(2,0),(2,2),\\
-1&(j,k)=(1,2),(1,3),(2,1),(2,3).\end{cases}\end{equation}
Then  \eqref{eqfc-1}-\eqref{eqfc-2} imply that the eigenvalues of $N_j(A_k)$ are $\{\varepsilon_{j,k},\varepsilon_{j,k}\}$. From here and noting that $y_a(z)$ and $y_{-a}(z)$ are linearly dependent, it follows from the uniqueness of $\pm a$ that \eqref{GLE} has no another eigenfunction that is linearly independent with $y_a(z)$, thus, up to a common conjugation,
\begin{equation}
N_1=N_1(A_k)=\varepsilon_{1,k}%
\begin{pmatrix}
1 & 0\\
1 & 1
\end{pmatrix}
,\text{ \  \  \ }N_2=N_2(A_k)=\varepsilon_{2,k}%
\begin{pmatrix}
1 & 0\\
\mathcal{C}_k & 1
\end{pmatrix}
, \label{Mono-21}%
\end{equation}
where $\mathcal{C}_k%
\in \mathbb{C}\cup \{ \infty \}$. In particular,
\begin{equation}
(\text{tr}N_1(A_k),\text{tr}N_2(A_k))=(2\varepsilon_{1,k}%
,2\varepsilon_{2,k})\in \{ \pm(2,2),\pm(2,-2)\}. \label{notcompleteC}%
\end{equation}
Here if $\mathcal{C}_k=\infty$, then (\ref{Mono-21}) should be understood
as%
\[
N_1=N_1(A_k)=\varepsilon_{1,k}%
\begin{pmatrix}
1 & 0\\
0 & 1
\end{pmatrix}
,\text{ \  \  \ }N_2=N_2(A_k)=\varepsilon_{2,k}%
\begin{pmatrix}
1 & 0\\
1 & 1
\end{pmatrix}
. \]
See \cite{CKL1} for a proof. In particular, at least one of $N_1(A_k), N_2(A_k)$ are not unitary matrices for all $k=0,1,2,3$. 

Clearly there is a one-to-one correspondence between $\{A_k\}_{k=0}^3$ and the trivial critical points $\{\frac{\omega_k}{2}\}_{k=0}^3$ of $G_p(z)$.
We recall the following result from Part II \cite{CFL-II}.

\begin{theorem}\cite[Theorem 2.5]{CFL-II}\label{rmk2-5}
There is a one-to-one correspondence between those $A$'s such that $(r,s)=(r(A), s(A))\in\mathbb{R}^2\setminus\frac12\mathbb Z^2$ and pairs of nontrivial critical points $\pm a=\pm (r+s\tau)$ of $G_p(z)$. 
\end{theorem}

\begin{proof} We provide the proof here for its importance.
Note that $z\notin E_{\tau}[2]$ is equivalent to $(r,s)\in\mathbb{R}^2\setminus\frac12\mathbb{Z}^2$.
Let $a\in E_{\tau}\setminus E_{\tau}[2]$ and write $a=r+s\tau$ with $(r,s)\in\mathbb{R}^2\setminus\frac12\mathbb Z^2$. It follows from \eqref{00a+1p} that $\pm a$ is a pair of nontrivial critical points of $G_p(z)$ if and only if 
\begin{equation}\label{61-37-22-2}
\zeta
(a+p)+\zeta(a-p)-2(r\eta_1+s\eta_2)=0,\;\text{with}\;a=r+s\tau, \;(r,s)\in\mathbb{R}^2\setminus\frac12\mathbb Z^2,
\end{equation}
which is precisely  \eqref{61-37-22}.

Suppose $(r(A), s(A))\in\mathbb{R}^2\setminus\frac12\mathbb Z^2$ for some $A$, then it follows from \eqref{61-37-22} that $\pm(r(A)+s(A)\tau)$ is a pair of nontrivial critical points of $G_p(z)$.

Conversely, suppose $\pm a=\pm (r+s\tau)$ with 
$(r,s)\in\mathbb{R}^2\setminus\frac12\mathbb Z^2$ is a pair of nontrivial critical points of $G_p(z)$, then \eqref{61-37-22-2} holds. Define $A$ by \eqref{Aap}, i.e.,
\[A:=\frac{1}{2}\left[  \zeta(p+a)+\zeta(p-a)-\zeta(2p)\right].\]
Then a direct computation shows that $y_a(z)$ given by \eqref{exp1}-\eqref{61-38}
is a solution of GLE \eqref{GLE} and so is $y_{-a}(z)$. Furthermore, by \eqref{61-37-1} and \eqref{61-37-22-2}, we see that $(r(A), s(A))=(r,s)\in\mathbb{R}^2\setminus\frac12\mathbb Z^2$. This completes the proof.
\end{proof}

\begin{remark}\label{rmk20-5}
Applying the additional formula \eqref{eqfc-add}, it is easy to see that \eqref{61-37-22} is equivalent to
\begin{align}\label{513-1} 
\wp(p)=\wp (r+s\tau)+\frac{\wp ^{\prime }(r+s\tau)}{%
2(\zeta(r+s\tau)-r\eta_1-s\eta_2)}.
\end{align}
The \eqref{513-1} is well known as Hitchin's formula \cite{Hit1}, where Hitchin studied Einstein metrics and proved that for any fixed $(r,s)\in \mathbb{C}^{2}\setminus \frac{1}{2}\mathbb{Z}^{2}$, the $p_{r,s}(\tau)$ defined by $\wp(p_{r,s}(\tau);\tau)=\text{RHS of \eqref{513-1}}$, as a function of $\tau\in \mathbb{H}$, is a solution of the following Painlev\'{e} VI equation
$$
\frac{d^{2}p(\tau)}{d\tau^{2}}=\frac{-1}{32\pi^{2}}\sum_{k=0}^{3}
\wp^{\prime}\left( p(\tau)+\frac{\omega_{k}}{2};
\tau \right).$$
See \cite[Theorem 6]{Hit1}.
Remark that Hitchin's formula \eqref{513-1} has the following important corollary:
\begin{equation} \label{eq: data0}
  \parbox{\dimexpr\linewidth-5em}{If $p\neq \pm p'$ in $E_{\tau}$, then $G_p(z)$ and $G_{p'}(z)$ can not have common nontrivial critical points.
  }
\end{equation}
\end{remark}

For later usage, we recall the criterion of $\frac{\omega_k}{2}$ being a degenerate critical point of $G_p(z)$ from Part I \cite{CFL}.
Define
\begin{equation}\label{B00}
\mathcal{B}_0:=\Big\{z\in\mathbb{C}\; :\; \Big|z-\Big(\frac{\pi}{\operatorname{Im}\tau}-\eta_1\Big)\Big|<\frac{\pi}{\operatorname{Im}\tau}\Big\},
\end{equation}
and for $k\in\{1,2,3\}$, we define
\begin{align}\label{alphak0}
\alpha_k:=\frac{\frac{\pi}{\operatorname{Im}\tau}-(\eta_1+e_k)}{3e_k^2-\frac{g_2}{4}},\quad \beta_k:=\frac{\frac{\pi}{\operatorname{Im}\tau}}{|3e_k^2-\frac{g_2}{4}|}>0,
\end{align}
\begin{equation}\label{alphak1}
\mathcal{B}_k:=\begin{cases}\bigg\{z\in\mathbb{C}\; :\; \bigg|z-e_k-\frac{\overline{\alpha_k}}{|\alpha_k|^2-\beta_k^2}\bigg|<\frac{\beta_k}{\left||\alpha_k|^2-\beta_k^2\right|}\bigg\}\quad\text{if }|\alpha_k|\neq \beta_k,\\
\Big\{z\in\mathbb{C}\; :\; \operatorname{Re}(\alpha_k (z-e_k))>\frac12\Big\}\quad\text{if }|\alpha_k|=\beta_k,\end{cases}
\end{equation}
that is, $\mathcal{B}_k$ is either an open disk or an open half plane. It was proved in Part I \cite{CFL} that $|\alpha_k|=\beta_k$ is equivalent to that $\frac{\omega_k}{2}$ is a degenerate critical point of $G(z)$. 
\begin{Theorem}\cite{CFL}\label{main-thm-01} Let $k\in\{0,1,2,3\}$ and $p\neq \frac{\omega_k}{2}$. Then $\frac{\omega_k}{2}$ is a degenerate critical point of $G_p(z)$ if and only if $\wp(p-\frac{\omega_k}{2})\in \partial\mathcal{B}_0$, if and only if $\wp(p)\in \partial\mathcal{B}_k$, where $\mathcal{B}_k$ is defined by \eqref{B00}-\eqref{alphak1}.
\end{Theorem}

\subsection{Conditional stability sets}
In this subsection, we briefly review the theory of conditional stability sets from Part II \cite[Section 3]{CFL-II}.
Recalling \eqref{eqfc-1}-\eqref{eqfc-2}, we have
\begin{equation}\label{eqfc-001}{y}_{\pm a}(z+2\tau-1)=e^{\pm 2\pi i(2r+s)}{y}_{\pm a}(z).\end{equation}
Define 
\begin{equation}\label{eqfc-3}\triangle_j (A):=\frac12\operatorname{tr} N_j(A)=\begin{cases}\frac12(e^{-2\pi is(A)}+e^{2\pi is(A)}),& j=1,\\
\frac12(e^{2\pi ir(A)}+e^{-2\pi ir(A)}),& j=2,\end{cases}\quad A\in\mathbb{C},\end{equation}
and
\begin{align}\label{eqfc-003}\triangle_3 (A):=&\frac12\operatorname{tr} N_2(A)^2N_1(A)^{-1}\nonumber\\
=&\frac12(e^{2\pi i(2r(A)+s(A))}+e^{-2\pi i(2r(A)+s(A))}),\quad A\in\mathbb{C}.\end{align}
Clearly $\triangle_j (A)$ are holomorphic functions (see \cite[Section 3]{CFL-II} for a proof). Define
\begin{equation}\label{eqfc-4}\sigma_j:=\triangle_j^{-1}([-1,1]),\quad j=1,2,3.\end{equation}
Then $\sigma_j$ consists of analytic arcs. It was proved in \cite[Section 3]{CFL-II} that $\sigma_j$ contains at most $5$ analytic arcs; see Corollaries \ref{coro2-11}-\ref{coro2-11-3} below. Observe that
\begin{itemize}
\item When $A\in \sigma_1$, the corresponding $s(A)\in\mathbb{R}$, so $y_{a}(z)$ is bounded for $z\in q_0+\mathbb{R}$. 
\item When $A\in \sigma_2$, the corresponding $r(A)\in\mathbb{R}$, so $y_{a}(z)$ is bounded for $z\in q_0+\tau\mathbb{R}$. 
\item When $A\in\sigma_3$, the corresponding $2r(A)+s(A)\in\mathbb{R}$, so $y_{a}(z)$ is bounded for $z\in q_0+(2\tau-1)\mathbb{R}$. 
\end{itemize}
Therefore, $\sigma_j$ can be considered as {\it the conditional stability sets} of GLE (\ref{GLE}) (cf. \cite{GW}).

\begin{theorem}\cite[Theorem 3.1]{CFL-II}\label{rmk2-8}
The number of pairs of nontrivial critical points of $G_p(z)$ equals to $\#(\sigma_1\cap\sigma_j\setminus\{A_0, A_1, A_2, A_3\})$ for $j=2,3$. 

More precisely, $A\in \sigma_1\cap\sigma_j\setminus\{A_0, A_1, A_2, A_3\}$ if and only if the corresponding $\pm a=\pm(r(A)+s(A)\tau)$ is a pair of nontrivial critical points of $G_p(z)$.
\end{theorem}

\begin{proof} We provide the proof here for its importance. Note that $(r,s)\in\mathbb R^2$ if and only if $(2r+s, s)\in\mathbb R^2$.
Suppose $A\in \sigma_1\cap\sigma_2\setminus\{A_0, A_1, A_2, A_3\}$, then $Q(A)\neq 0$, and it follows from \eqref{Mono-001} and \eqref{eqfc-3}-\eqref{eqfc-4} that the corresponding $(r(A),s(A))\in \mathbb R^2\setminus\frac12\mathbb{Z}^2$, so $A\in \sigma_1\cap\sigma_3\setminus\{A_0, A_1, A_2, A_3\}$.
Therefore,
$$\sigma_1\cap\sigma_2\setminus\{A_0, A_1, A_2, A_3\}=\sigma_1\cap\sigma_3\setminus\{A_0, A_1, A_2, A_3\},$$
and $(r(A),s(A))\in \mathbb R^2\setminus\frac12\mathbb{Z}^2$ if and only if $A\in \sigma_1\cap\sigma_2\setminus\{A_0, A_1, A_2, A_3\}$. The proof is complete by using Theorem \ref{rmk2-5}.
\end{proof}

In Part II \cite{CFL-II}, we used $\sigma_1\cap\sigma_2$ to study the rectangular torus case. In this paper, we will use $\sigma_1\cap\sigma_3$ to study the rhombus torus case. Thus, we need to recall some basic properties of $\sigma_1$ and $\sigma_3$.

First, we need to give the definitions of endpoints, cusps and branch points.
\begin{figure}[ht]
\includegraphics[scale=0.5, trim=250 370 50 80]{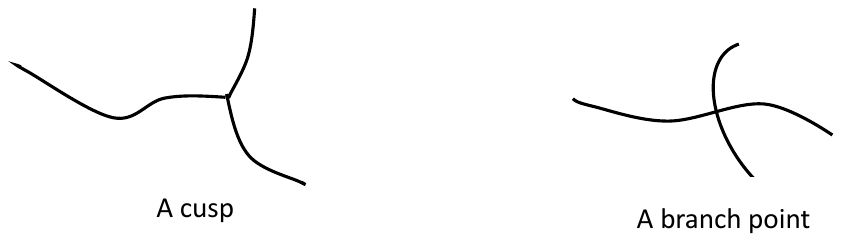}\captionof{figure}{Cusp and branch point}\label{O511-0}
\end{figure}

\begin{remark}\label{rmk2-6}
An endpoint $\tilde{A}$ of an arc of $\sigma_j$ is a point where the arc can not be analytically continued (cf. \cite{GW}). By the Taylor expansion, we see that $\tilde{A}$ is an endpoint if and only if $\triangle_j(\tilde A)^2-1=0$ and
\[d_j(\tilde{A}):=\text{ord}_{\tilde{A}}(\triangle_j(\cdot)^2-1)\]
is \emph{odd}, and in this case there are $d_j(\tilde{A})$ arcs meeting at $\tilde{A}$, each of which can not be analytically continued at $\tilde{A}$ because adjacent arcs meet at $\tilde{A}$ with the same angle $2\pi/d_j(\tilde{A})$. 
Notice that if $\triangle_j(\tilde A)^2-1=0$ and $d_j(\tilde{A})=2k\geq 2$ is even, then $\tilde{A}$ is an inner point of $k$ arcs which are all analytic at $\tilde{A}$, so such $\tilde{A}$ is not considered as an endpoint.
\begin{itemize}
\item
We call an endpoint $\tilde{A}$ of $\sigma_j$ to be {\bf a cusp} if $d_j(\tilde{A})\geq 3$, that is, there are odd numbers (at least $3$) of arcs meeting at $\tilde{A}$. See the left one in Figure \ref{O511-0} for example.
\item
We call a point $\tilde{A}$ of $\sigma_j$ to be {\bf a branch point} if $\tilde{A}$ is not an endpoint and $\tilde{A}$ is a common inner point of at least two analytic arcs. See the right one in Figure \ref{O511-0} for example.
\end{itemize}
\end{remark}

Note from \eqref{var-jjkk} that $\{A_0, A_1, A_2, A_3\}\subset \sigma_j$ with
\begin{equation}\label{eqfc-vjk1}\triangle_j(A_k)=\varepsilon_{j,k}=\begin{cases}1& (j,k)=(1,0),(1,1),(3,0),(3,1)\\
-1&(j,k)=(1,2),(1,3),(3,2),(3,3),\end{cases}\end{equation}
i.e., $\triangle_j(A_k)^2-1=0$. 

\begin{lemma}\cite[Lemma 3.3]{CFL-II}\label{lemma2-7} Fix $j\in \{1,3\}$. Then 
\begin{itemize}
\item[(1)]
 $\{A_0, A_1, A_2, A_3\}$ are precisely the set of endpoints of $\sigma_j$.
 \item[(2)] Any connected component of $\mathbb{C}\setminus\sigma_j$ is unbounded.
 \end{itemize}
\end{lemma}

\begin{lemma}\cite[Lemma 3.4]{CFL-II}\label{lemma2-10}
\begin{itemize}
\item[(1)]
The endpoint $A_k$ is a cusp of $\sigma_1$ if and only if $\wp(p-\frac{\omega_k}{2})+\eta_1=0$.
\item[(2)] The endpoint $A_k$ is a cusp of $\sigma_3$ if and only if $(2\tau-1)\wp(p-\frac{\omega_k}{2})+2\eta_2-\eta_1=0$, or equivalently, $\wp(p-\frac{\omega_k}{2})+\eta_1=\frac{4\pi i}{2\tau-1}$.
\end{itemize}
Consequently, $\sigma_1$ and $\sigma_3$ can not have common cusps.
\end{lemma}

\begin{lemma}\cite[Lemma 3.5]{CFL-II}\label{lemma2-10-3} Let $A\in\sigma_j\setminus\{A_0, A_1, A_2, A_3\}$. 
\begin{itemize}
\item[(1)]
 $A$ is a branch point of $\sigma_1$ if and only if $\wp(a+p)+\wp(a-p)+2\eta_1=0$.
\item[(2)] $A$ is a branch point of $\sigma_3$ if and only if $(2\tau-1)(\wp(a+p)+\wp(a-p))+4\eta_2-2\eta_1=0$, or equivalently, $\wp(a+p)+\wp(a-p)+2\eta_1=\frac{8\pi i}{2\tau-1}$.
\end{itemize}
Consequently, $\sigma_1$ and $\sigma_3$ can not have common branch points.
\end{lemma}

\begin{lemma}\cite[Lemma 3.6]{CFL-II}\label{lemma2-8}
Denote $B_R:=\{A\in\mathbb{C} : |A|<R\}$. Then for $R>0$ large, the following statements hold.
\begin{itemize}
\item[(1)] $\sigma_1\setminus{B_R}$ consists of only two disjoint analytic arcs that can be parametrized by 
$$
A=it-\frac{2p\eta_1-\zeta(2p)}{2}+O(t^{-1})
,\quad t\in(-\infty, -t_1]\cup[t_2,+\infty),$$
for some $t_j>0$ large.

\item[(2)] $\sigma_3\setminus{B_R}$ consists of only two disjoint analytic arcs that can be parametrized by 
$$
A=\frac{it}{2\tau-1}-\frac{2p(2\eta_2-\eta_1)-(2\tau-1)\zeta(2p)}{2(2\tau-1)}+O(t^{-1})
,$$
where $t\in(-\infty, -t_3]\cup[t_4,+\infty)$ for some $t_j>0$ large.
\end{itemize}

\end{lemma}

\begin{corollary}\cite[Corollary 3.7]{CFL-II}\label{coro2-11} Let $p\in E_{\tau}\setminus E_{\tau}[2]$. Then $\sigma_1$ consists of $m_1\in [3,5]$ analytic arcs with finite endpoints $A_0, A_1, A_2$, $A_3$, among which there is one or two unbounded arcs. 
More precisely,
\begin{itemize}
\item[(1)] If $\wp(p-\frac{\omega_k}{2})+\eta_1\neq 0$ for all $k$, then $\sigma_1$ has no cusps and $\sigma_1$ consists of $3$ analytic arcs.

\item[(2)] If $\wp(p-\frac{\omega_k}{2})+\eta_1=0$ and $12\eta_1^2-g_2\neq 0$ for some $k$, then $A_k$ is a cusp of $\sigma_1$  with $\text{ord}_{{A}_k}(\triangle_1(\cdot)^2-1)=3$. Furthermore,
\begin{itemize}
\item[(2-1)] If $\wp(p-\frac{\omega_l}{2})+\eta_1\neq 0$ for any $l\neq k$, then $A_k$ is the unique cusp of $\sigma_1$,
and $\sigma_1$ consists of $4$ analytic arcs, among which there are three of them having the common endpoint $A_k$.

\item[(2-2)] If $\wp(p-\frac{\omega_l}{2})+\eta_j=0$ for some $l\neq k$, then $A_l$ is also a cusp of $\sigma_1$ with $\text{ord}_{{A}_l}(\triangle_1(\cdot)^2-1)=3$,
and $\sigma_1$ is path-connected and consists of $5$ analytic arcs, among which one has endpoints $A_k$ and $A_l$, two others have the common endpoints $A_k$, and the remaining two have the common endpoints $A_l$.
\end{itemize}

\item[(3)] If $\wp(p-\frac{\omega_k}{2})+\eta_1=0$ and $12\eta_1^2-g_2=0$ for some $k$, then $A_k$ is the unique cusp of $\sigma_1$  with $\text{ord}_{{A}_k}(\triangle_1(\cdot)^2-1)=5$, and $\sigma_1$ is path-connected and consists of $5$ analytic arcs that have the common endpoint $A_k$.
\end{itemize}

\end{corollary}

\begin{corollary}\cite[Corollary 3.8]{CFL-II}\label{coro2-11-3} Let $p\in E_{\tau}\setminus E_{\tau}[2]$. Then $\sigma_3$ consists of $m_3\in [3,5]$ analytic arcs with finite endpoints $A_0, A_1, A_2$, $A_3$, among which there is one or two unbounded arcs. 
More precisely,
\begin{itemize}
\item[(1)] If $(2\tau-1)\wp(p-\frac{\omega_k}{2})+2\eta_2-\eta_1\neq 0$ for all $k$, then $\sigma_3$ has no cusps and $\sigma_3$ consists of $3$ analytic arcs.

\item[(2)] If $(2\tau-1)\wp(p-\frac{\omega_k}{2})+2\eta_2-\eta_1=0$ and $12(2\eta_2-\eta_1)^2-(2\tau-1)^2g_2\neq 0$ for some $k$, then $A_k$ is a cusp of $\sigma_3$  with $\text{ord}_{{A}_k}(\triangle_3(\cdot)^2-1)=3$. 
Furthermore,
\begin{itemize}
\item[(2-1)] If $(2\tau-1)\wp(p-\frac{\omega_l}{2})+2\eta_2-\eta_1\neq0$ for any $l\neq k$, then $A_k$ is the unique cusp of $\sigma_3$,
and $\sigma_3$ consists of $4$ analytic arcs, among which there are three of them having the common endpoint $A_k$.

\item[(2-2)] If $(2\tau-1)\wp(p-\frac{\omega_l}{2})+2\eta_2-\eta_1=0$ for some $l\neq k$, then $A_l$ is also a cusp of $\sigma_3$ with $\text{ord}_{{A}_l}(\triangle_3(\cdot)^2-1)=3$,
and $\sigma_3$ is path-connected and consists of $5$ analytic arcs, among which one has endpoints $A_k$ and $A_l$, two others have the common endpoints $A_k$, and the remaining two have the common endpoints $A_l$.
\end{itemize}

\item[(3)] If $(2\tau-1)\wp(p-\frac{\omega_k}{2})+2\eta_2-\eta_1=0$ and $12(2\eta_2-\eta_1)^2-(2\tau-1)^2g_2=0$ for some $k$, then $A_k$ is the unique cusp of $\sigma_3$  with $\text{ord}_{{A}_k}(\triangle_3(\cdot)^2-1)=5$, and  $\sigma_3$ is path-connected and consists of $5$ analytic arcs that have the common endpoint $A_k$.
\end{itemize}

\end{corollary}

We conclude this section by proving the following new observation.

\begin{lemma}\label{lem4-s5-13}
For $p\notin E_{\tau}[2]$, we write $\sigma_j=\sigma_j(p)$ to emphasize its dependence on $p$.
Then
\begin{equation}\label{eq7-16}
	-\sigma_j(p)=\sigma_j\Big(\frac12-p\Big),\qquad j=1,3.
\end{equation}
\end{lemma}

\begin{proof}
Given $p\not\in E_{\tau}[2]$ and $A\in\mathbb{C}$, by Lemma \ref{lemma2-3}, there is $a=a(A; p)\neq \pm p$ such that
\begin{equation}
A=\frac{1}{2}\left[  \zeta(p+a)+\zeta(p-a)-\zeta(2p)\right]  . \label{005Aap}
\end{equation}
Then
\begin{align*}
&\frac{1}{2}\left[  \zeta\Big(\frac12-p+\frac12-a\Big)+\zeta\Big(\frac12-p-\frac12+a\Big)-\zeta(1-2p)\right]\\
=&-\frac{1}{2}\left[  \zeta(p+a)+\zeta(p-a)-\zeta(2p)\right] =-A.
\end{align*}
This implies
\begin{equation}\label{00a-aa}
a\Big(-A;\frac12-p\Big)=\frac12-a(A; p).
\end{equation}
Recalling \eqref{61-38} that
\begin{equation*}
c(A; p)=\frac{1}{2}[\zeta
(a+p)+\zeta(a-p)],
\end{equation*}
we have
\begin{align}
c\Big(-A; \frac12-p\Big)&=\frac{1}{2}\left[\zeta
\Big(\frac12-a+\frac12-p\Big)+\zeta\Big(\frac12-a-\frac12+p\Big)\right]\nonumber
\\&=\frac12\eta_1-\frac{1}{2}[\zeta
(a+p)+\zeta(a-p)]=-c(A;p)+\frac12\eta_1.\label{00611-38}
\end{align}
Recalling $(r,s)=(r(A;p), s(A;p))\in \mathbb{C}^2$ defined by \eqref{61-37-1}-\eqref{61-37-2}, it follows from \eqref{00a-aa} and \eqref{00611-38} that
$$r\Big(-A;\frac12-p\Big)=\frac12-r(A;p),\quad s\Big(-A;\frac12-p\Big)=-s(A;p).$$
This, together with \eqref{eqfc-3}-\eqref{eqfc-003}, yields
$$\triangle_1 \Big(-A;\frac12-p\Big)=\frac12(e^{-2\pi is(-A;\frac12-p)}+e^{2\pi is(-A;\frac12-p)})=\triangle_1(A;p),$$
and
\begin{align*}&\triangle_3\Big (-A;\frac12-p\Big)\\
=&
\frac12(e^{2\pi i(2r(-A;\frac12-p)+s(-A;\frac12-p))}+e^{-2\pi i(2r(-A;\frac12-p)+s(-A;\frac12-p))})=\triangle_3(A;p).\end{align*}
This implies that $-A\in \sigma_j(\frac12-p)$ if and only if $A\in\sigma_j(p)$, so \eqref{eq7-16} holds.
\end{proof}

\section{The rhombus torus}

\label{sec-3}

From now on, we consider $\tau=\frac12+ib$ with $b>0$, that is, $E_{\tau}$ is a rhombus torus.
Note from $\tau=\frac12+ib$ that $\bar\tau=1-\tau$, i.e., $\mathbb Z+\mathbb Z\tau=\mathbb Z+\mathbb Z\bar{\tau}$. Then it follows from the definition of $\wp(z)=\wp(z;\tau)$ and $\zeta(z)=\zeta(z;\tau)$ that
\begin{equation}\label{barwp-5}\overline{\wp(\bar{z})}=\wp(z),\quad\overline{\zeta(\bar{z})}=\zeta(z).\end{equation}
In particular, $\wp(z)\in\mathbb{R}$ if and only if $\wp(z)=\wp(\bar z)$, if and only if $z=\pm \bar z$ in $E_{\tau}$, which is equivalent to (using $z=r+s\tau$ with $r,s\in[-\frac12,\frac12]$ and $\bar z=r+s-s\tau$)
\begin{equation}\label{eqfc-s5-3}\pm z\in \Big(0,\frac{1}{2}\Big]\cup \Big[\frac{1}{2},\frac{1}{4}+\frac{\tau}{2}\Big]\cup \Big[\frac14-\frac{\tau}{2},0\Big),\quad\operatorname{mod}\;\mathbb{Z}+\mathbb{Z}\tau.\end{equation}
Here and in the sequel, for $z_1\neq z_2\in \mathbb{C}$, $$[z_1, z_2]:=\big\{(1-t)z_1+tz_2\,:\, t\in [0,1]\big\},$$ and $[z_1, z_2)$, $(z_1, z_2]$, $(z_1, z_2)$ are defined in similar ways. 

Indeed, $z=\bar z$ in $E_{\tau}$ corresponds to $s=0$ and so $z\in \pm (0, \frac12]$, while $z=-\bar z$ in $E_{\tau}$ is equivalent to $2r+s=0,\pm 1$ and so
$$\begin{cases}z=-\frac{s}{2}+s\tau\in \pm [\frac14-\frac{\tau}{2},0)\quad\text{for }2r+s=0,\\
z=-\frac{s\pm 1}{2}+s\tau\in \pm[\frac{1}{2},\frac{1}{4}+\frac{\tau}{2}]\quad\text{for }2r+s=\pm 1.\end{cases}$$
Note that $\frac{1}{4}+\frac{\tau}{2}=\frac{1}{4}-\frac{\tau}{2}$ in $E_{\tau}$.
For convenience, we denote
$$I=\Big(0, \frac12\Big),\quad II:=\Big(\frac{1}{2},\frac{1}{4}+\frac{\tau}{2}\Big]=\Big(\frac12,\frac12+i\frac{b}{2}\Big]\subset\frac12+i\mathbb R,$$
$$III:=\Big[\frac14-\frac{\tau}{2},0\Big)=\Big[-i\frac{b}{2},0\Big)\subset i\mathbb{R}.$$
The figures of $I, II, III$ are given in Figure \ref{figure2-1}.

\begin{figure}[htbp]\label{figure2-1}
\centering
\begin{tikzpicture}[ scale=0.6,
    dot/.style={circle,fill=black,inner sep=1pt}
]
\draw[thick, black] (0,-2)--(0,0) -- (2,0) -- (2,2);
\draw[thick, black] (3,2)--(-1,2)--(-3,-2)--(1,-2)--cycle;

\draw[dashed] (-3.5,0) -- (3.5,0) node[right] {};
\draw[dashed] (0,-3) -- (0,3) node[above] {};

\node[above] at (1,0) {I};
\node[left] at (2,1.2) {II};
\node[right] at (0,-1.2) {III};
\node[left] at (0,0) {\footnotesize$0$};
\node[right] at (2,0) {\footnotesize$\frac12$};
\node[above] at (2,2) {\footnotesize$\frac14+\frac{\tau}2$};
\node[below] at (0,-2) {\footnotesize$\frac14 - \frac\tau2$};
\end{tikzpicture}
\captionof{figure}{I,II,III}
\label{figure2-1}
\end{figure}
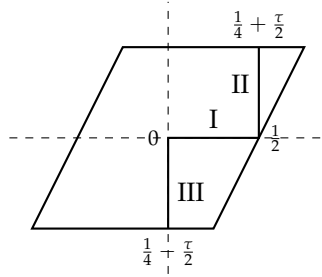

\begin{lemma}
\label{Lemma53-1} Let $\tau=\frac12+ib$ with $b>0$. Then $\wp$ is strictly increasing and bijective from $(0,\frac14-\frac{\tau}{2}]\cup [\frac{1}{4}+\frac{\tau}{2},\frac{1}{2}]\cup \lbrack \frac{1}{2},0)$ onto $(-\infty,+\infty)$, and
\begin{equation}\label{eqfc-s5-1}
\lim_{\,[\frac{1}{2},0)\ni z\rightarrow0}\wp(z)=+\infty,\quad \lim
_{[\frac14-\frac{\tau}{2},0)\ni z\rightarrow0}\wp(z)=-\infty.
\end{equation}
Furthermore,
\begin{equation}
\label{eqfc-s5-2} e_1, \eta_1, g_2, g_3\in\mathbb{R},\quad e_2=\overline{e_3}\notin\mathbb{R},\quad \operatorname{Im} e_2=-\operatorname{Im} e_3<0.
\end{equation}
\end{lemma}

\begin{proof}
 By
$$\wp(z)=\frac1{z^2}+O(z^2),\quad\text{as }z\to 0,$$
we obtain \eqref{eqfc-s5-1}. From here and the fact that $\wp: E_{\tau}\to \mathbb{C}\cup\{\infty\}$ is a double cover (i.e., for any $c\in\mathbb{C}\cup\{\infty\}$, there is a unique pair of $\pm z_c\in E_{\tau}$ such that $\wp(\pm z_c)=c$) and $\wp'(z)=0$ if and only if $z\in\{\frac{\omega_k}{2}\}_{k=1}^3$, we easily see that $\wp$ is strictly increasing and bijective from $(0,\frac14-\frac{\tau}{2}]\cup [\frac{1}{4}+\frac{\tau}{2},\frac{1}{2}]\cup \lbrack \frac{1}{2},0)$ onto $(-\infty,+\infty)$.

On the other hand, since $e_1=\wp(\frac12)$, $\eta_1=2\zeta(\frac12)$, $e_2=\wp(\frac{\tau}{2})$ and $e_3=\wp(\frac{1+\tau}{2})$, we see from \eqref{barwp-5} that $e_1, \eta_1\in\mathbb R$ and $e_2=\overline{e_3}$. In particular, $e_2\neq e_3$ implies $e_2=\overline{e_3}\notin\mathbb{R}$. Since
\begin{align*}
 &\wp'(z)^2 = 4\wp(z)^3 - g_2\wp(z) - g_3=4\prod_{k=1}^3(\wp(z)-e_k), \end{align*}
 we see that $e_1+e_2+e_3=0$ and so
\begin{equation}\label{eqfc-g2g3}g_2=4(e_1^2-e_2e_3)=4e_1^2-4|e_2|^2\in\mathbb{R},\quad g_3=4e_1e_2e_3=4e_1|e_2|^2\in\mathbb{R}.\end{equation}

On the other hand, it is known (see e.g. \cite[p.31]{CFL}) that for $b=\frac{\sqrt{3}}2$, it holds that $e_1>0$ and $e_2=-e^{\pi i/3}e_1$, i.e., $\operatorname{Im} e_2=-\frac{\sqrt{3}}{2}e_1<0$. Then by $\operatorname{Im}e_2\neq 0$ for all $b>0$ and the continuity, we conclude that $\operatorname{Im}e_2< 0$ for all $b>0$.
 This completes the proof.
\end{proof}

\begin{remark}\label{rmk-s5-2} 
Write $z=x_1+ix_2$ with $x_1, x_2\in\mathbb R$. Then by \eqref{barwp-5} and Lemma \ref{Lemma53-1}, we have 
\begin{itemize}
\item For $z\in I=(0,\frac12)$, $\wp^{\prime
}(z)=\frac{\partial \wp(z)}{\partial x_{1}}<0$.
\item For $z\in II\cup III=(\frac{1}{2},\frac{1}{4}+\frac{\tau}{2}]\cup [\frac14-\frac{\tau}{2}, 0)$, $\wp^{\prime}(z)=-i\frac{\partial \wp(z)}{\partial x_{2}}$ with $\frac{\partial \wp(z)}{\partial x_{2}}<0$.
\item For $z\in III=[\frac14-\frac{\tau}{2}, 0)\subset i\mathbb{R}$, $\zeta(2z)\in i\mathbb{R}$. 
\item For $z\in II=(\frac{1}{2},\frac{1}{4}+\frac{\tau}{2}]\subset\frac12+i\mathbb R$, $\zeta(2z)=\eta_1+\zeta(2z-1)\in\eta_1+i\mathbb{R}$, so $2z\eta_1-\zeta(2z)\in i\mathbb{R}$.
\end{itemize}
\end{remark}

\subsection{Characterization of conditional stability sets}

In this subsection, we will show that the conditional stability sets $\sigma_1$ and $\sigma_3$ admit certain symmetries.

\begin{lemma}\label{lemma53-3}
Let $\tau=\frac12+ib$ with $b>0$, $p\notin E_{\tau}[2]$ and $\wp(p)\in\mathbb{R}$. Recall $A_k$ defined in \eqref{Ak}. 
\begin{itemize}
\item[(1)] For $p\in I=(0, \frac12)$, we have $A_1<A_0$, $A_2=\overline{A_3}\notin\mathbb{R}$ and $\operatorname{Im}A_2>0$.
\item[(2)] For $p\in II\cup III=(\frac{1}{2},\frac{1}{4}+\frac{\tau}{2}]\cup [\frac14-\frac{\tau}{2}, 0)$, we have $A_0, A_1\in i\mathbb{R}\setminus\{0\}$ with $\operatorname{Im}A_1<\operatorname{Im} A_0$,  $A_2=-\overline{A}_3\notin i\mathbb{R}$ and $\operatorname{Re}A_2>0$. 

\end{itemize}
\end{lemma}

\begin{proof}
By \eqref{eqfc-s5-3} and by replacing $p$ with $-p$ if necessary, the assumption $\wp(p)\in\mathbb{R}$ and $p\notin E_{\tau}[2]$ imply that \begin{equation}\label{eqfc-p}p\in I\cup II\cup III=\Big(0,\frac{1}{2}\Big)\cup \Big(\frac{1}{2},\frac{1}{4}+\frac{\tau}{2}\Big]\cup \Big[\frac14-\frac{\tau}{2}, 0\Big).\end{equation}
Furthermore, $\wp''(p)=6\wp(p)^2-g_2/2\in\mathbb{R}$.

(1). Since $p\in I=(0, \frac12)$, we have $e_1<\wp(p)$ and $\wp'(p)<0$, which together with \eqref{Ak} implies
 $A_1<A_0$. Furthermore, we see from \eqref{Ak} and $e_2=\overline{e_3}\notin\mathbb{R}$ that $A_2=\overline{A_3}\notin\mathbb{R}$, and it follows from $\operatorname{Im}{e_2}<0$ that
 $$\operatorname{Im} A_2=\frac{\wp'(p) \operatorname{Im}{e_2}}{2|\wp(p)-e_2|^2}>0.$$

(2). Since $p\in II\cup III=(\frac{1}{2},\frac{1}{4}+\frac{\tau}{2}]\cup [\frac14-\frac{\tau}{2}, 0)$ implies $\wp(p)<e_1$ and $\wp^{\prime}(p)=-i\frac{\partial \wp(p)}{\partial x_{2}}$ with $\frac{\partial \wp(p)}{\partial x_{2}}<0$, it follows  that $A_0, A_1\in i\mathbb{R}$ with $\operatorname{Im}A_1<\operatorname{Im} A_0$.  Furthermore, we see from \eqref{Ak} and $e_2=\overline{e_3}\notin\mathbb{R}$ that $A_2=-\overline{A}_3\notin i\mathbb{R}$ and
  $$\operatorname{Re} A_2=-\frac{|\wp'(p)| \operatorname{Im}{e_2}}{2|\wp(p)-e_2|^2}>0.$$
  This completes the proof.
\end{proof}


\begin{lemma}\label{lemma53-4}
Let $\tau=\frac12+ib$ with $b>0$.
\begin{itemize}
\item[(1)] If $p\in I=(0, \frac12)$, then $\sigma_j$ is symmetric with respect to the real axis for each $j=1,3$.
\item[(2)] If $p\in II\cup III=(\frac{1}{2},\frac{1}{4}+\frac{\tau}{2}]\cup [\frac14-\frac{\tau}{2}, 0)$, then $\sigma_j$ is symmetric with respect to the imaginary axis for each $j=1,3$.
\end{itemize}

\end{lemma}

\begin{proof}  
Note from $\tau=\frac12+ib$ that $\overline{2\tau-1}=1-2\tau$.
Set $\hat{y}_a(z):=\overline{y_{a}(\bar z)}$. Then \eqref{eqfc-1} and \eqref{eqfc-001} imply that
\begin{equation}\label{eqfc-s5-5}\hat{y}_{a}(z+1)=\overline{e^{-2\pi is}}\hat{y}_{a}(z),\end{equation}
\begin{equation}\label{eqfc-s5-6} \hat{y}_{a}(z+2\tau-1)=\overline{e^{-2\pi i(2r+s)}}\hat{y}_{ a}(z).\end{equation}
Furthermore, $\hat{y}_a''(z)=\overline{I(\bar z; p, A, \tau)}\hat{y}_a(z)$, where by using \eqref{barwp-5}, 
\begin{align}\label{eqfc-s5-7}
\overline{I(\bar z; p, A, \tau)}=
&\frac{3}{4}
(\wp(z+\bar{p})+\wp(z-\bar{p})-\wp (2\bar p))\nonumber\\
&+\overline{A}(\zeta(z+\bar{p})-\zeta(z-\bar{p})-\zeta(2\bar p))+\overline{A}^{2}\nonumber\\
=&I(z; \bar p, \overline A, \tau).
\end{align}

(1) For $p\in (0, \frac12)$, we have $\bar p=p$, so $I(z; \bar p, \bar A, \tau)=I(z; p, \bar A, \tau)$. Then it follows from \eqref{eqfc-s5-5}-\eqref{eqfc-s5-6} that \begin{equation}\label{eqfc-8}\triangle_1(\overline{A})=\overline{\cos 2\pi s}=\overline{\triangle_1(A)},\quad \triangle_3(\overline{A})=\overline{\cos 2\pi (2r+s)}=\overline{\triangle_3(A)}.\end{equation}
Thus, $\overline{A}\in \sigma_j$ if and only if $A\in\sigma_j$, namely $\sigma_j$ is symmetric with respect to the real axis for $j=1,3$.

(2) For $p\in II\cup III= (\frac{1}{2},\frac{1}{4}+\frac{\tau}{2}]\cup [\frac14-\frac{\tau}{2}, 0)$, we have either $\bar{p}=-p$ or $\bar p=1-p\equiv -p$ mod $\Lambda_{\tau}$, so $I(z; \bar p, \overline A, \tau)=I(z; p, -\overline A, \tau)$. Consequently,
\begin{equation}\label{eqfc-s5-9}\triangle_1(-\overline{A})=\overline{\cos 2\pi s}=\overline{\triangle_1(A)},\quad \triangle_3(-\overline{A})=\overline{\cos 2\pi (2r+s)}=\overline{\triangle_3(A)}.\end{equation}
Thus, $-\overline{A}\in \sigma_j$ if and only if $A\in\sigma_j$, namely $\sigma_j$ is symmetric with respect to the imaginary axis for $j=1,3$.
\end{proof}

\begin{lemma}\label{Lemma53-5} Let $\tau=\frac12+ib$ with $b>0$ and $p\in I=(0,\frac12)$. Then
\begin{itemize}

\item[(1)] $\sigma_3=(-\infty, A_1]\cup [A_0, +\infty)\cup\sigma_{3,A_2A_3}$, where $\sigma_{3,A_2A_3}$ is a simple curve connecting $A_2$ and $A_3$, and is symmetric with respect to the real axis with $\sigma_{3,A_2A_3}\cap\mathbb{R}$ containing a single point. 
\item[(2)] $[A_1, A_0]\subset\sigma_1$, and one of the following holds.
\begin{itemize}
\item[(2-1)] $\sigma_1=[A_1, A_0]\cup\sigma_{1,\infty A_2}\cup\overline{\sigma_{1,\infty A_2}}$, where $\sigma_{1,\infty A_2}\subset\{z : \operatorname{Im}z>0\}$ is a semi-unbounded simple curve with the endpoint $A_2$ and tending to $-\frac{2p\eta_1-\zeta(2p)}{2}+i\infty$, while $\overline{\sigma_{1,\infty A_2}}\subset\{z : \operatorname{Im}z<0\}$ is the complex conjugate of $\sigma_{1,\infty A_2}$ such that $\sigma_{1,\infty A_2}\cup\overline{\sigma_{1,\infty A_2}}$ is symmetric with respect to the real axis.
\item[(2-2)] $\sigma_1=[A_1, A_0]\cup\sigma_{1,A_2A_3}\cup\sigma_{1,\infty}$, where $\sigma_{1,\infty}$ is an unbounded simple curve tending to $-\frac{2p\eta_1-\zeta(2p)}{2}\pm i\infty$, which is symmetric with respect to the real axis with $\sigma_{1,\infty}\cap\mathbb{R}$ containing a single point, and $\sigma_{1,A_2A_3}$ is a simple curve connecting $A_2$ and $A_3$, and is symmetric with respect to the real axis with $\sigma_{1,A_2A_3}\cap\mathbb{R}$ containing a single point. Remark that it might happen $\sigma_{1,\infty}\cap\mathbb{R}=\sigma_{1,A_2A_3}\cap\mathbb{R}$; see Section \ref{sec-7} for examples.
\end{itemize}
\end{itemize}
\end{lemma}

\begin{proof} Note that $2\tau-1=2ib$.
For $p\in (0,\frac12)$, by the Legendre relation $\tau\eta_1-\eta_2=2\pi i$, we have $2p\eta_1-\zeta(2p)\in \mathbb{R}$ and
$$\frac{2p(2\eta_2-\eta_1)-(2\tau-1)\zeta(2p)}{2(2\tau-1)}=\frac{2p\eta_1-\zeta(2p)}{2}-\frac{2p\pi}{b}\in\mathbb{R}.$$
Together with Lemmas \ref{lemma2-7}, \ref{lemma2-8}, \ref{lemma53-3}, \ref{lemma53-4} and Corollaries \ref{coro2-11}-\ref{coro2-11-3}, we obtain the following facts:

\begin{itemize}
\item[(a)] $A_1<A_0$ and $A_2=\overline{A_3}\notin\mathbb R$ are all the endpoints of $\sigma_j$;

\item[(b)] $\sigma_j$ is symmetric with respect to the
real line $\mathbb{R}$ for $j=1,3$;

\item[(c)] any connected component of
$\mathbb{C}\setminus \sigma_j$ is unbounded;

\item[(d)] for $R>0$ large, $\sigma_3\setminus B_R$ consists of exactly two unbounded arcs that tend to $+\infty$ and $-\infty$ separately, and $\sigma_1\setminus B_R$ consists of exactly two unbounded arcs that tend to $-\frac{2p\eta_1-\zeta(2p)}{2}+ i\infty$ and $-\frac{2p\eta_1-\zeta(2p)}{2}-i\infty$ separately.
\end{itemize}
Then we easily obtain the desired assertions. 
\end{proof}

\begin{lemma}\label{Lemma53-6} Let $\tau=\frac12+ib$ with $b>0$ and $p=\frac14$. Then $A_0=-A_1>0$, $A_2=\overline{A_3}\in i\mathbb{R}\setminus\{0\}$ with $\operatorname{Im}A_2>0$ and 
\begin{itemize}
\item[(1)] $\sigma_3=(-\infty, A_1]\cup [A_0, +\infty)\cup [A_3, A_2]$.
\item[(2)] $\sigma_1=[A_1, A_0]\cup (-i\infty, A_3]\cup [A_2, +i\infty)$.
\end{itemize}
Consequently, $\sigma_1\cap\sigma_3\setminus\{A_0,A_1, A_2,A_3\}=\{0\}$.
\end{lemma}

\begin{proof}
Recall \eqref{eq7-16} that for $p=\frac14$, 
$$-\sigma_j=\sigma_j,\qquad j=1,3.$$
From here and Lemma \ref{lemma53-4}, we see that $\sigma_j$ is symmetric with respect to both the real axis and the imaginary axis for each $j=1,3$. Together with Lemma \ref{Lemma53-5}, we obtain the desired assertions.
\end{proof}

In Section \ref{sec-7}, we will improve Lemma \ref{Lemma53-5}  to a sharp version for general $p\in (0, \frac12)$.


\begin{lemma}\label{Lemma53-7} Let $\tau=\frac12+ib$ with $b>0$ and $p\in II\cup III=(\frac{1}{2},\frac{1}{4}+\frac{\tau}{2}]\cup [\frac14-\frac{\tau}{2}, 0)$. Then
\begin{itemize}
\item[(1)] $\sigma_1=(-i\infty, A_1]\cup [A_0, +i\infty)\cup\sigma_{1,A_2A_3}$, where $\sigma_{1,A_2A_3}$ is a simple curve connecting $A_2$ and $A_3=-\overline{A_2}$, and is symmetric with respect to the imaginary axis with $\sigma_{1,A_2A_3}\cap i\mathbb{R}$ containing a single point. 

\item[(2)] $[A_1, A_0]\subset\sigma_3$, and one of the following holds.
\begin{itemize}
\item[(2-1)] $\sigma_3=[A_1, A_0]\cup\sigma_{3,\infty A_2}\cup -\overline{\sigma_{3,\infty A_2}}$, where $\sigma_{3,\infty A_2}\subset\{z : \operatorname{Re}z>0\}$ is a semi-unbounded simple curve with the endpoint $A_2$ and tending to $-\frac{2p(2\eta_2-\eta_1)-(2\tau-1)\zeta(2p)}{2(2\tau-1)}+\infty$, while $-\overline{\sigma_{3,\infty A_2}}\subset\{z : \operatorname{Re}z<0\}$ is the reflection of $\sigma_{3,\infty A_2}$ with respect to the imaginary axis.
\item[(2-2)] $\sigma_3=[A_1, A_0]\cup\sigma_{3,A_2A_3}\cup\sigma_{3,\infty}$, where $\sigma_{3,\infty}$ is an unbounded simple curve tending to $-\frac{2p(2\eta_2-\eta_1)-(2\tau-1)\zeta(2p)}{2(2\tau-1)}\pm\infty$, which is symmetric with respect to the imaginary axis with $\sigma_{3,\infty}\cap i\mathbb{R}$ containing a single point, and $\sigma_{3,A_2A_3}$ is a simple curve connecting $A_2$ and $A_3$, and is symmetric with respect to the imaginary axis with $\sigma_{3,A_2A_3}\cap i\mathbb{R}$ containing a single point. It might happen that $\sigma_{3,\infty}\cap i\mathbb{R}=\sigma_{3,A_2A_3}\cap i\mathbb{R}$; see \eqref{eqq04-21} for example.
\end{itemize}
\end{itemize}

\end{lemma}

\begin{proof}
For $p\in II\cup III$, it follows from Remark \ref{rmk-s5-2} that $2p\eta_1-\zeta(2p)\in i\mathbb{R}$.
Then Lemma \ref{lemma2-8} says that for $R>0$ large, $\sigma_1\setminus B_R$ consists of exactly two unbounded arcs that tend to $+i\infty$ and $-i\infty$ separately, and $\sigma_3\setminus B_R$ consists of exactly two unbounded arcs that tend to $-\frac{2p(2\eta_2-\eta_1)-(2\tau-1)\zeta(2p)}{2(2\tau-1)}+ \infty$ and $-\frac{2p(2\eta_2-\eta_1)-(2\tau-1)\zeta(2p)}{2(2\tau-1)}-\infty$ separately.
Consequently, this lemma can be proved similarly as Lemma \ref{Lemma53-5}.
\end{proof}

\subsection{The map $f(a)$}
For $a\in E_{\tau}\setminus E_{\tau}[2]$, we can write $a=r+s\tau$ with $r,s\in\mathbb{R}^2\setminus\frac12\mathbb{Z}^2$.
Recalling Hitchin's formula \eqref{513-1-0}, we define a map $f: E_{\tau}\setminus E_{\tau}[2]\to \mathbb{C}\cup\{\infty\}$ by
\begin{align}\label{513-1-0} 
f(a)=f(r+s\tau):=&\wp (a)+\frac{\wp ^{\prime }(a)}{%
2(\zeta(a)-r\eta_1-s\eta_2)}\\
=&\wp (a)+\frac{\wp ^{\prime }(a)}{%
2(\zeta(a)+(\frac{\pi}{\operatorname{Im}\tau}-\eta_1)a-\frac{\pi}{\operatorname{Im}\tau}\bar{a})}.\nonumber
\end{align}
Note that $f(a)$ is not a meromorphic function of $a$.
By defining $\wp(p):=f(a)$, we see from Remark \ref{rmk20-5} that \eqref{513-1-0} is equivalent to
\begin{equation}\label{61-37-22--1}
\zeta
(a+p)+\zeta(a-p)-2r\eta_1-2s\eta_2=0,
\end{equation}
so $\pm a$ is a pair of nontrivial critical points of $G_p(z)=G_{-p}(z)$.
Therefore, it is important to study the image of $f(a)$.

When $\wp(p)=f(a)\notin\{e_1, e_2, e_3,\infty\}$, i.e., $p\notin E_{\tau}[2]$, then $\wp(a-p)-\wp(a+p)\neq 0$, $\wp'(p)\neq 0$ and then it was computed in Part I \cite[Theorem 4.2]{CFL} that
\begin{equation}\label{deri-r}
 \frac{\partial f}{\partial r}(a)=\wp'(p)\frac{\wp(a+p)+\wp(a-p)+2\eta_1}{\wp(a-p)-\wp(a+p)}.
 \end{equation}

To study the image of $f(a)$ for $\wp(a)\in\mathbb{R}$ in Sections \ref{sec-4}-\ref{sec-6}, we need to establish the following lemmas.

\begin{lemma}\label{lemma-s5-13}
Let $\tau=\frac12+ib$ with $b>0$ and $p\in I\cup II\cup III$. Suppose $\pm a$ is a pair of nontrivial critical points of $G_p(z)$. Then $\pm \bar{a}$ is also a pair of nontrivial critical points of $G_p(z)$. Furthermore, $\{\pm a\}=\{\pm \bar a\}$ in $E_{\tau}$ if and only if $\wp(a)\in\mathbb{R}$, if and only if the corresponding $A\in\mathbb{R}$ for $p\in I$ (resp. $A\in i\mathbb{R}$ for $p\in II\cup III$).
\end{lemma}

\begin{proof}
Write $a=r+s\tau$ with $(r,s)\in\mathbb R^2\setminus\frac12\mathbb Z^2$. Then $\bar a=r+s-s\tau$. Note from the Legendre relation $\eta_2=\tau\eta_1-2\pi i$ that $\overline{\eta_2}=(1-\tau)\eta_1+2\pi i=\eta_1-\eta_2$. Recall that $a$ is a critical point of $G_p(z)$ if and only if \eqref{61-37-22--1} holds.
From here, $\overline{\zeta(z)}=\zeta(\bar z)$ and $\bar p=\pm p$ in $E_{\tau}$, we obtain
\begin{align*}
\zeta(\bar a+p)+\zeta(\bar a-p)-2(r+s)\eta_1+2s\eta_2=0,
\end{align*}
so $\bar a=r+s-s\tau$ is a nontrivial critical point of $G_p(z)$. Since $\overline{\wp(a)}=\wp(\bar a)$, we have $\wp(a)\in\mathbb R$ if and only if $\bar a=\pm a$ in $E_{\tau}$, i.e., $\{\pm a\}=\{\pm \bar a\}$ in $E_{\tau}$.
Recall from \eqref{Aap} that
$$A=\frac{1}{2}\left[  \zeta(p+a)+\zeta(p-a)-\zeta(2p)\right].$$

{\bf Case 1.} $p\in I=(0,\frac12)$.

Then $\bar p=p$, so
$$\overline A=\frac{1}{2}\left[  \zeta(p+\bar a)+\zeta(p-\bar a)-\zeta(2p)\right].$$
By the one-to-one correspondence of $A\leftrightarrow \pm a$, we see that $\{\pm a\}=\{\pm \bar a\}$ if and only if $A=\overline A$, i.e. $A\in\mathbb R$.

{\bf Case 2.} $p\in II\cup III=(\frac{1}{2},\frac{1}{4}+\frac{\tau}{2}]\cup [\frac14-\frac{\tau}{2}, 0)$.

Then $\bar p=-p$ in $E_{\tau}$, so
$$\overline A=-\frac{1}{2}\left[  \zeta(p+\bar a)+\zeta(p-\bar a)-\zeta(2p)\right].$$
Consequently, $\{\pm a\}=\{\pm \bar a\}$ if and only if $A=-\overline A$, i.e. $A\in i\mathbb R$.
\end{proof}

\begin{lemma}\label{lemma-aup}
Let $\tau=\frac12+ib$ with $b>0$. Given any $a\in I\cup II\cup III$, there is a unique $p\in I\cup II\cup III\cup \{0,\frac12\}$ such that $\pm a$ is a pair of nontrivial critical points of $G_p(z)$. 
\end{lemma}

\begin{proof}
Write $a=r+s\tau$ with $(r,s)\in\mathbb R^2\setminus\frac12\mathbb Z^2$. Recall that $\pm a$ is a pair of nontrivial critical points of $G_p(z)=G_{-p}(z)$ by letting $\wp(p)=f(a)$, this proves the existence of $p$. The uniqueness of $\pm p$ follows from \eqref{eq: data0}. Thus, we only need to prove $\wp(p)\in\mathbb{R}\cup\{\infty\}$. 

By $\bar a=r+s-s\tau$ and $\overline{\eta_2}=(1-\tau)\eta_1+2\pi i=\eta_1-\eta_2$, we see from \eqref{61-37-22--1} that
\begin{align*}
\zeta(\bar a+\bar p)+\zeta(\bar a-\bar p)-2(r+s)\eta_1+2s\eta_2=0,
\end{align*}
so $\pm \bar a$ is a pair of nontrivial critical points of $G_{\bar p}(z)$. However, $a\in I\cup II\cup III$ implies that $\pm\bar a=\pm a$  is also a pair of nontrivial critical points of $G_p(z)$. Then it follows the uniqueness of $\pm p$ that $p=\pm \bar p$ in $E_{\tau}$, which implies $\wp(p)\in\mathbb{R}\cup\{\infty\}$, so we may assume $p\in I\cup II\cup III\cup \{0,\frac12\}$.
\end{proof}

\subsection{Known results for the Green function $G(z)$}
For later usage, we recall some results concerning the Green function $G(z)$.

\begin{theorem}\cite[Theorem 1.6]{LW}\label{thm-LW-16}
Let $\tau=\frac12+ib$ with $b>0$. Then
\begin{itemize}
\item[(1)] There are $b_0\approx0.35473<\frac12<b_1=\frac{1}{4b_0}\approx 0.70476<\frac{\sqrt{3}}{2}$ such that $\frac{\omega_1}{2}=\frac12$ is a degenerate critical point of $G(z;\tau)$ if and only if $b=b_0$ or $b=b_1$. Moreover, $\frac12$ is a non-degenerate minimum point of $G(z;\tau)$ if $b_0<b<b_1$ and it is a non-degenerate saddle point if $b<b_0$ or $b>b_1$.
\item[(2)] Both $\frac{\omega_2}{2}$ and $\frac{\omega_3}{2}$ are non-degenerate saddle points of $G(z;\tau)$.
\item[(3)] $G(z;\tau)$ has a unique pair of nontrivial critical points $\pm z_0(\tau)$ when $b<b_0$ or $b>b_1$. They are non-degenerate minimum points of $G(z;\tau)$ and
$$\begin{cases}\operatorname{Re} z_0(\tau)=\frac12,\quad 0<\operatorname{Im} z_0(\tau)<\frac{b}{2},\\
-\eta_1<\wp(z_0(\tau))<\frac{2\pi}{b}-\eta_1,\end{cases}\quad\text{for }b>b_1,$$
namely $z_0(\tau)\in II^{\circ}=(\frac12,\frac{1}{4}+\frac{\tau}{2})$ and $\frac12-z_0\in III^{\circ}=(\frac14-\frac{\tau}{2}, 0)$ for $b>b_1$.
\end{itemize}
\end{theorem}

\begin{remark}\label{rmk5-9}
Let $\tau=\frac12+ib$ with $b>0$. It was proved in \cite{LW} that
$$\det D^2G\Big(\frac{1}{2}\Big)=-\frac1{4\pi^2}(e_1+\eta_1)\left(e_1+\eta_1-\frac{2\pi}{b}\right),$$
and $\frac{d}{db}(e_1+\eta_1)>0$. Thus
\begin{equation}\label{eqfc-s5-10}e_1+\eta_1=0\;\text{at }b=b_0,\quad e_1+\eta_1=\frac{2\pi}{b}\;\text{at }b=b_1.\end{equation}
Besides, in the proof of Theorem \ref{thm-LW-16}-(2), Lin and Wang proved that (see \cite[(6.25), (6.26)]{LW})
\begin{equation}\label{eqfc-s5-11}
-\eta_1>\wp\Big(\frac{\omega_2+\omega_3}{4}\Big)=\wp\Big(\frac14+\frac{\tau}{2}\Big)=e_1-|e_1-e_2|,
\end{equation}
\begin{equation}\label{eqfc-s5-12}
\frac{2\pi}{b}-\eta_1<\wp\Big(\frac14\Big)=e_1+|e_1-e_2|.
\end{equation}
\end{remark}

\begin{theorem}\cite[Theorem 3.6]{CFL}\label{thm-5c}
Fix any $\tau$ such that $G(z)=G(z;\tau)$ has $5$ critical points. Then there exists $\varepsilon>0$ small such that if $0<|p-\frac{\omega_k}{2}|<\varepsilon$ for some $k\in \{0,1,2,3\}$, $G_p(z;\tau)$ has exactly $6$ critical points, which are all non-degenerate. Furthermore, the unique pair of nontrivial critical points are non-degenerate minimal points of $G_p(z)$.
\end{theorem}

\section{The case $\tau=\frac12+ib$ with $b>b_1$.}

\label{sec-4}

In this section, we study the case $\tau=\frac12+ib$ with $b>b_1$ and prove Theorems \ref{III-thm3}-\ref{III-thm4}. 
First, recall from Remark \ref{rmk5-9} that
\begin{align}\label{eqfc4-1}
e_1+\eta_1>\frac{2\pi}{b}, \quad\forall b>b_1.
\end{align}
Recall $\mathcal{B}_k, \alpha_k, \beta_k$ defined in \eqref{B00}-\eqref{alphak1}.
Since Theorem \ref{thm-LW-16} says that $\frac{\omega_k}{2}$ are non-degenerate saddle points of $G(z)$ for all $k=1,2,3$, it follows from Part I \cite[Section 3]{CFL} that
\begin{align}\label{alphak0-5}
|\alpha_k|=\left|\frac{\frac{\pi}{b}-(\eta_1+e_k)}{3e_k^2-\frac{g_2}{4}}\right|>\beta_k=\frac{\frac{\pi}{b}}{|3e_k^2-\frac{g_2}{4}|}>0,\quad k=1,2,3,
\end{align}
\begin{equation}\label{alphak1-5}
\mathcal{B}_k=\bigg\{z\in\mathbb{C}\; :\; \bigg|z-e_k-\frac{\overline{\alpha_k}}{|\alpha_k|^2-\beta_k^2}\bigg|<\frac{\beta_k}{|\alpha_k|^2-\beta_k^2}\bigg\},\quad k=1,2,3.
\end{equation}
By \eqref{eqfc-s5-2}, we see that the center of $\mathcal{B}_1$ is on the real axis, while $\mathcal{B}_2$ is the refection of $\mathcal{B}_3$ with respect to the real axis.
Also recall that
\begin{equation}\label{B00-5}
\mathcal{B}_0=\Big\{z\in\mathbb{C}\; :\; \Big|z-\Big(\frac{\pi}{b}-\eta_1\Big)\Big|<\frac{\pi}{b}\Big\}.
\end{equation}
The figures of these four disks for $\tau=\frac12+i\frac{\sqrt{3}}{2}$ are given in Figure \ref{tO5-1}.

\begin{figure}
\includegraphics[width=1.6in]{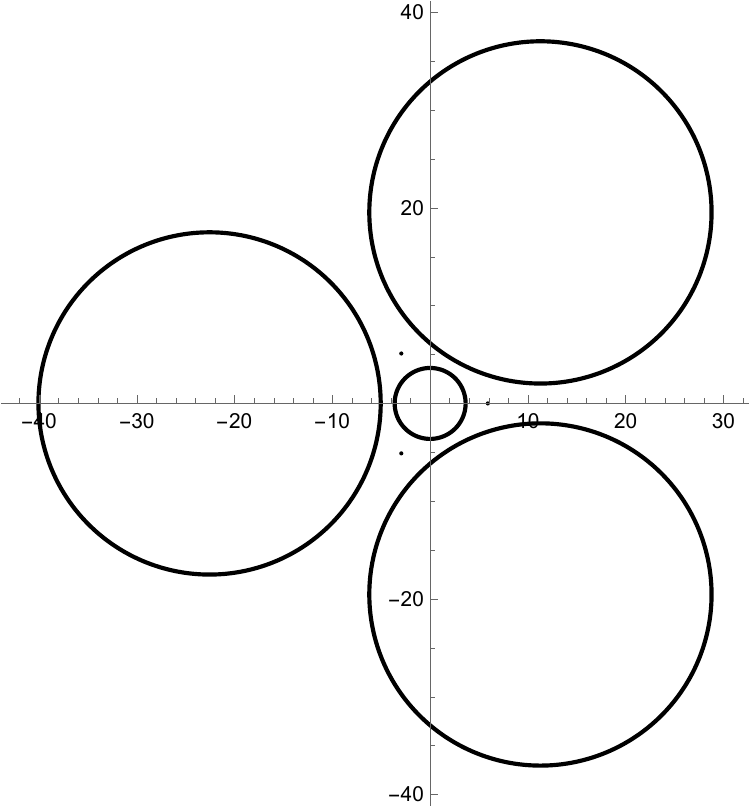}\captionof{figure}{The four disks and $\{e_1,e_2,e_3\}$ for $\tau=\frac12+i\frac{\sqrt{3}}{2}$: the smallest circle for $\partial\mathcal{B}_0$, left for $\partial\mathcal{B}_1$, upper right for $\partial\mathcal{B}_2$ and lower right for $\partial\mathcal{B}_3$.}\label{tO5-1}
\end{figure}

\begin{lemma}\label{lemma-s5-9} Let $\tau=\frac12+ib$ with $b>b_1$.
Let $d_{1,1}<d_{1,2}$ be the two intersection points of $\partial\mathcal{B}_1$ with $\mathbb{R}$, and $d_{0,1}:=-\eta_1<d_{0,2}:=\frac{2\pi}{b}-\eta_1$ be the two intersection points of $\partial\mathcal{B}_0$ with $\mathbb{R}$. Then $d_{1,2}<\wp(\frac{1}4+\frac{\tau}{2})<d_{0,1}$, so $\partial\mathcal{B}_0\cap \partial\mathcal{B}_1=\emptyset$ and
\begin{itemize}
\item[(1)] there are $p_{1,1}, p_{1,2}\in III^{\circ}$ such that $p_{1,1}\in (p_{1,2}, 0)$ and $\wp(p_{1,1})=d_{1,1}$, $\wp(p_{1,2})=d_{1,2}$.
\item[(2)] there are $p_{0,1}=\frac12-p_{1,2}, p_{0,2}=\frac12-p_{1,1}\in II^{\circ}$ such that $p_{0,2}\in (p_{0,1}, \frac12)$ and $\wp(p_{0,1})=d_{0,1}$, $\wp(p_{0,2})=d_{0,2}$. 
\end{itemize} 
Consequently, $z_0(\tau)\in (p_{0,2}, p_{0,1})$ and $\frac12-z_0(\tau)\in (p_{1,2}, p_{1,1})$, where $z_0(\tau)$ is introduced in Theorem \ref{thm-LW-16}. The relative locations of these points are seen in the following future.
\end{lemma}

\begin{figure}[htbp]
\centering
\begin{tikzpicture}[ scale=0.7,
    dot/.style={circle,fill=black,inner sep=1pt}
]
\draw[thick, red] (0,-2)--(0,0) -- (2,0) -- (2,2);
\draw[thick, black] (3,2)--(-1,2)--(-3,-2)--(1,-2)--cycle;

\node[dot] at (2, 1.5) (p01) {};
\node[dot] at (2, 1) (z0) {};
\node[dot] at (2, 0.5) (p02) {};
\node[dot] at (0, -0.5) (p11) {};
\node[dot] at (0, -1) (mid) {};
\node[dot] at (0, -1.5) (p12) {};

\node[above] at (1,0) {I};
\node[left] at (2,1.2) {II};
\node[right] at (0,-1.2) {III};
\node[left] at (0,0) {\scriptsize$0$};
\node[right] at (2,0) {\scriptsize$\frac12$};
\node[above] at (2,2) {\scriptsize$\frac14+\frac{\tau}2$};
\node[right] at (p01) {\scriptsize$p_{0,1}$};
\node[right] at (z0) {\scriptsize$z_0$};
\node[left] at (p02) {\scriptsize$p_{0,2}$};
\node[left] at (p11) {\scriptsize$p_{1,1}=\frac12-p_{0,2}$};
\node[left] at (mid) {\scriptsize$\frac12 - z_0$};
\node[left] at (p12) {\scriptsize$p_{1,2}=\frac12-p_{0,1}$};
\node[below] at (0,-2) {\scriptsize$\frac14 - \frac\tau2$};
\end{tikzpicture}
\captionof{figure}{Relative locations of those points in Lemma \ref{lemma-s5-9}}
\label{figure4-1}
\end{figure}
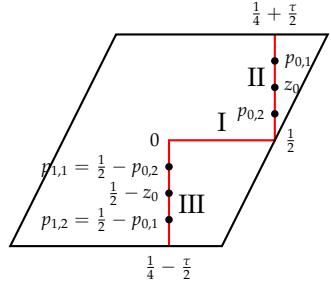

%
%
%
%
%
%

\begin{proof} 
Since $e_1+e_2+e_3=0$ and $e_3=\overline{e_2}$, so $e_2=-\frac12 e_1+i\operatorname{Im}e_2$, from which we have $|e_1-e_2|^2=\frac94e_1^2+(\operatorname{Im}e_2)^2=2e_1^2+|e_2|^2$.
Consequently, it follows from $g_2=4e_1^2-4|e_2|^2$ that $$3e_1^2-\frac{g_2}{4}=2e_1^2+|e_2|^2=|e_1-e_2|^2>0,$$and then \eqref{alphak0} implies
$$\alpha_1=\frac{\frac{\pi}{b}-(\eta_1+e_1)}{3e_1^2-\frac{g_2}{4}}\in\mathbb R,\qquad \beta_1=\frac{\frac{\pi}{b}}{3e_1^2-\frac{g_2}{4}}.$$
It follows from \eqref{alphak1-5} that $$d_{1,1}=e_1+\frac{\alpha_1-\beta_1}{\alpha_1^2-\beta_1^2}=e_1+\frac{1}{\alpha_1+\beta_1}=e_1+\frac{3e_1^2-\frac{g_2}{4}}{\frac{2\pi}{b}-(e_1+\eta_1)}\notin\{e_1, e_2, e_3\},$$
and
\begin{align}\label{eqfc-d12}d_{1,2}&=e_1+\frac{\alpha_1+\beta_1}{\alpha_1^2-\beta_1^2}=e_1+\frac{1}{\alpha_1-\beta_1}=e_1-\frac{3e_1^2-\frac{g_2}{4}}{e_1+\eta_1}\\
&=e_1-\frac{|e_1-e_2|^2}{e_1+\eta_1}\notin\{e_1, e_2, e_3\}.\nonumber\end{align}
Recalling \eqref{eqfc-s5-10}-\eqref{eqfc-s5-11} that $|e_1-e_2|>e_1+\eta_1>0$, we obtain
$$d_{1,2}=e_1-\frac{|e_1-e_2|^2}{e_1+\eta_1}<e_1-|e_1-e_2|=\wp\Big(\frac14+\frac{\tau}{2}\Big)<-\eta_1=d_{0,1}.$$
This proves $d_{1,1}<d_{1,2}<\wp(\frac14+\frac{\tau}{2})=\wp(\frac14-\frac{\tau}{2})$, so the assertion (1) follows from Lemma \ref{Lemma53-1}.
Besides, together with \eqref{eqfc-s5-10}, we have $$\wp\Big(\frac14+\frac{\tau}{2}\Big)<d_{0,1}=-\eta_1<d_{0,2}=\frac{2\pi}{b}-\eta_1<e_1,$$ so it follows from Lemma \ref{Lemma53-1} that $p_{0,1}, p_{0,2}\in II^\circ$. 

Note that $p\in II$ if and only if $\frac12-p\in III$.
Recalling \eqref{eqfc-d12} that $d_{1,2}=\wp(p_{1,2})=e_1-\frac{3e_1^2-\frac{g_2}{4}}{e_1+\eta_1}$, it follows from the additional formula of elliptic functions that
\begin{align}\label{eqfc-s5-31}\wp\Big(p_{1,2}-\frac12\Big)&=e_1+\frac{\wp''(\frac12)}{2(\wp(p_{1,2})-e_1)}\\
&=e_1+\frac{3e_1^2-\frac{g_2}{4}}{\wp(p_{1,2})-e_1}=-\eta_1=\wp(p_{0,1}),\nonumber\end{align}
so $p_{0,1}=\frac12-p_{1,2}$.
Similarly, by $d_{1,1}=\wp(p_{1,1})=e_1+\frac{3e_1^2-\frac{g_2}{4}}{\frac{2\pi}{b}-(e_1+\eta_1)}$, we have
\begin{align}\label{eqfc-s5-22}\wp\Big(p_{1,1}-\frac12\Big)&=e_1+\frac{\wp''(\frac12)}{2(\wp(p_{1,1})-e_1)}\\
&=e_1+\frac{3e_1^2-\frac{g_2}{4}}{\wp(p_{1,1})-e_1}=\frac{2\pi}{b}-\eta_1=\wp(p_{0,2}),\nonumber\end{align}
so $p_{0,2}=\frac12-p_{1,1}$, which proves the assertion (2).

Finally, recall from Theorem \ref{thm-LW-16} that $\wp(p_{0,1})<\wp(z_0(\tau))<\wp(p_{0,2})$, we get $z_0(\tau)\in (p_{0,2}, p_{0,1})$ and then $\frac12-z_0(\tau)\in (p_{1,2}, p_{1,1})$.
\end{proof}

\begin{corollary}\label{coro-s5-10}
Let $\tau=\frac12+ib$ with $b>b_1$ and $p\in I\cup II\cup III$. Then for $j=1,3$, $\sigma_j$ has at most one cusp, which must be one of $\{A_0, A_1\}$ if exists. Furthermore,
\begin{itemize}
\item[(1)]  $A_0$ is a cusp of $\sigma_3$ if and only if $\wp(p)=\wp(p_{0,2})=d_{0,2}=\frac{2\pi}{b}-\eta_1$.
\item[(2)]  $A_0$ is a cusp of $\sigma_1$ if and only if $\wp(p)=\wp(p_{0,1})=d_{0,1}=-\eta_1$.
\item[(3)]  $A_1$ is a cusp of $\sigma_1$ if and only if $\wp(p)=\wp(p_{1,2})=d_{1,2}$, or equivalently $\wp(p-\frac12)=-\eta_1$.
\item[(4)]  $A_1$ is a cusp of $\sigma_3$ if and only if $\wp(p)=\wp(p_{1,1})=d_{1,1}$, or equivalently $\wp(p-\frac12)=\frac{2\pi}{b}-\eta_1$.
\end{itemize}
Therefore, if $\wp(p)\notin\{\wp(p_{j,k}), j=0,1, k=1,2\}$, then $\sigma_j$ has no cusps for $j=1,3$.
\end{corollary}

\begin{proof}
By Lemmas \ref{Lemma53-5}-\ref{Lemma53-7}, we see that none of $\{A_2, A_3\}$ can be cusps of $\sigma_j$, so if $\sigma_j$ has cusps, then they must be $A_0$ or $A_1$.
Consequently, the assertions (1)-(4) follow directly from Lemma \ref{lemma2-10}, $\frac{4\pi i}{2\tau-1}=\frac{2\pi}{b}$ and \eqref{eqfc-s5-31}-\eqref{eqfc-s5-22}. The proof is complete.
\end{proof}

Now we begin to study the image of the map $f(a)$ for $a\in I\cup II\cup III$ by using the conditional stability sets.

\begin{lemma}\label{lemma-aup-1}
Let $\tau=\frac12+ib$ with $b>b_1$. Recall Lemma \ref{lemma-aup} that for any $a\in I\cup II\cup III$, there is a unique $p\in I\cup II\cup III\cup \{0,\frac12\}$ such that  $\wp(p)=f(a)$ and $\pm a$ is a pair of nontrivial critical points of $G_p(z)$, where $f(a)$ is defined in \eqref{513-1-0}.
Then $f: I\to f(I)$ is a differmorphism and
\begin{equation}\label{dddd}f(I)=(\wp(p_{1,2}), \wp(p_{0,1})).\end{equation}
\end{lemma}

\begin{proof}
Let $a=r\in I=(0,\frac12)$.
Then $\wp(r)>e_1, \wp'(r)<0$. Since \eqref{eqfc4-1} implies
$$\frac{d}{dr}(\zeta(r)-r\eta_1)=-\wp(r)-\eta_1<-(e_1+\eta_1)<0,\quad\forall r\in \Big(0,\frac12\Big),$$ and $\zeta(\frac12)-\frac12\eta_1=0$, we obtain $\zeta(r)-r\eta_1>0$ for any $r\in (0,\frac12)$. Inserting this into \eqref{513-1-0} leads to $$\wp(p)=f(a)=\wp(r)+\frac{2\wp'(r)}{2(\zeta(r)-r\eta_1)}<\wp(r)=\wp(a),$$
so $p\neq 0$, i.e., $p\in I\cup II\cup III\cup\{\frac12\}$ and $\infty\notin f(I)$. 

When $r>0$ is sufficiently small,
$$\wp(r)=\frac{1}{r^2}+O(r^2), \quad \wp'(r)=\frac{-2}{r^3}+O(r),$$
$$\zeta(r)-r\eta_1=\frac1r-r\eta_1+O(r^3),$$
so 
\begin{equation}\label{eqq-d01}\wp(p)=f(a)=\wp(r)+\frac{\wp'(r)}{2(\zeta(r)-r\eta_1)}=-\eta_1+O(r^2)=\wp(p_{0,1})+o(1).\end{equation}
Simiarly, when $\frac12-r>0$ is sufficiently small, 
$$\wp(r)=e_1+O\left(\Big(\frac12-r\Big)^2\right), \quad \wp'(r)=\wp''\Big(\frac12\Big)\Big(r-\frac12\Big)+O\left(\Big(\frac12-r\Big)^3\right),$$
$$\zeta(r)-r\eta_1=-(e_1+\eta_1)\Big(r-\frac12\Big)+O\left(\Big(\frac12-r\Big)^3\right),$$
so
\begin{align}\label{eqfc-s5-98}\wp(p)&=f(a)=\wp(r)+\frac{\wp'(r)}{2(\zeta(r)-r\eta_1)}=e_1-\frac{\wp''(\frac12)}{2(e_1+\eta_1)}+o(1)\nonumber\\
&=e_1-\frac{3e_1^2-\frac{g_2}{4}}{e_1+\eta_1}+o(1)=\wp(p_{1,2})+o(1).\end{align}

On the other hand, the same proof as \cite[Lemma 5.5-(3)]{CFL-II} implies that $G_{\frac14+\frac{\tau}{2}}(z)$ has a unique pair of nontrivial critical points $\pm\frac14$, so \begin{equation}\label{eq4-11}f\Big(\frac14\Big)=\wp\Big(\frac14+\frac{\tau}{2}\Big)\in (\wp(p_{1,2}), \wp(p_{0,1})).\end{equation} Since $f: I\to f(I)$ is continuous, so $f(I)\subset\mathbb R$ is connected, which together with \eqref{eqq-d01}-\eqref{eq4-11} implies
$$(\wp(p_{1,2}), \wp(p_{0,1}))\subset f(I).$$

Recall Theorem \ref{thm-LW-16} that $G_{\frac12}(z)=G(z-\frac12)$ has a unique pair of nontrivial critical points $\pm (\frac12-z_0(\tau))$ with $\frac12-z_0(\tau)\in III^\circ$, so  $e_1=\wp(\frac12)\notin f(I)$ and then
$$(\wp(p_{1,2}), \wp(p_{0,1}))\subset f(I)\subset (-\infty, e_1),$$
i.e., $p\in II\cup III$.
Together with $a=r\in I=(0, \frac12)$, we see from \eqref{deri-r} that $f: I\to f(I)\subset \mathbb{R}$ satisfies
$$f'(r)=\wp'(p)\frac{\wp(r+p)+\wp(r-p)+2\eta_1}{\wp(r-p)-\wp(r+p)}\in\mathbb{R}.$$
We want to show $f'(r)\neq 0$ for any $r\in (0,\frac12)$.

Since $a=r$ is a nontrivial critical point of $G_p(z)$, it follows from Lemma \ref{lemma-s5-13} that the corresponding
\begin{align}\label{eqq-ar}A(r)&=\frac{1}{2}\left[  \zeta(p+r)+\zeta(p-r)-\zeta(2p)\right]\\
&=A_0+\frac{\wp'(p)}{2(\wp(p)-\wp(r))}\in\sigma_1\cap i\mathbb{R}\setminus\{A_k\}_{k=0}^3,\nonumber\end{align}
where we have used the additional formula \eqref{eqfc-add} to obtain the second equality.
Since $f(\frac14)=\wp(\frac14+\frac{\tau}{2})\in (\wp(p_{1,2}), \wp(p_{0,1}))$, we have $\wp(\frac14+\frac{\tau}{2})<e_1<\wp(\frac14)$, and then we see from Remark \ref{rmk-s5-2} and \eqref{Ak} that
$$A\Big(\frac14\Big)=A_0+\frac{\wp'(\frac14+\frac{\tau}{2})}{2(\wp(\frac14+\frac{\tau}{2})-\wp(\frac14))}\in (A_1, A_0).$$
By the continuity of $A(r)$ and \eqref{eqq-ar}, we obtain $A(r)\in (A_1, A_0)$ for all $r\in (0,\frac12)$. Consequently, we see from Lemma \ref{Lemma53-7} and Lemma \ref{lemma2-10-3} that $A(r)$ is not a branch point of $\sigma_1$, so $\wp(r+p)+\wp(r-p)+2\eta_1\neq0$ for all $r\in (0,\frac12)$, which implies $f'(r)\neq 0$ for all $r\in (0, \frac12)$. Thus, $f: I\to f(I)$ is differomorphism and we can conclude from \eqref{eqq-d01} and \eqref{eqfc-s5-98} that $f(I)=(\wp(p_{1,2}), \wp(p_{0,1}))$.
\end{proof}

\begin{lemma}\label{lem04-4}
Let $\tau=\frac12+ib$ with $b>0$ and define
\begin{equation}\label{betab}
\beta=\beta(b):=12\Big(\eta_1\Big(\frac12+ib\Big)-\frac{2\pi}{b}\Big)^2-g_2\Big(\frac12+ib\Big).
\end{equation}
Then $\beta=0$ if and only if
\begin{equation}\label{eqq0-4}
12\big(2\eta_2-\eta_1\big)^2-(2\tau-1)^2g_2=0.
\end{equation}
Furthermore, there is $b_2\in (\sqrt{3}/2, 6/5)$ such that
\begin{equation}\label{eqq-413}
\beta(b)\begin{cases} >0\quad\text{for }0<b<b_2,\\
=0\quad\text{for }b=b_2,\\
<0\quad\text{for }b>b_2.
\end{cases}
\end{equation}
\end{lemma}
\begin{proof}
For any 
$\begin{pmatrix} a&\hat{b}\\c&d \end{pmatrix}\in SL_2(\mathbb Z),$
we recall the modular properties of $\eta_1$ and $g$:
$$g_2\Big(\frac{a\tau+\hat{b}}{c\tau+d}\Big)=(c\tau+d)^4g_2(\tau),\quad \eta_1\Big(\frac{a\tau+\hat{b}}{c\tau+d}\Big)=(c\tau+d)(c\eta_2(\tau)+d\eta_1(\tau)).$$
Using $\begin{pmatrix} a&\hat{b}\\c&d \end{pmatrix}=\begin{pmatrix} 1&-1\\2&-1 \end{pmatrix}$, $\tau\eta_1-\eta_2=2\pi i$ and $\tau=\frac12+ib$, i.e. $2\tau-1=2bi$, we see that
{\allowdisplaybreaks
\begin{align}\label{eqq0-5}
&12\eta_1\Big(\frac{\tau-1}{2\tau-1}\Big)^2-g_2\Big(\frac{\tau-1}{2\tau-1}\Big)\nonumber\\
=&(2\tau-1)^2\big[12(2\eta_2(\tau)-\eta_1(\tau))^2-(2\tau-1)^2g_2(\tau)\big]\nonumber\\
=&(2\tau-1)^2\big[12\big(2(\tau\eta_1(\tau)-2\pi i)-\eta_1(\tau)\big)^2-(2\tau-1)^2g_2(\tau)\big]\nonumber\\
=&16b^4\Big(12\Big(\eta_1\Big(\frac12+ib\Big)-\frac{2\pi}{b}\Big)^2-g_2\Big(\frac12+ib\Big)\Big)=16b^4\beta(b).
\end{align}
}%
Clearly \eqref{eqq0-5} implies that $\beta=0$ if and only if \eqref{eqq0-4} holds.
Recall the classical formula (see e.g. \cite[(1.5)]{CL-JDG19}) that
$$\frac{d}{db}\eta_1\Big(\frac12+ib\Big)=-\frac{1}{24\pi}\left(12\eta_1\Big(\frac12+ib\Big)^2-g_2\Big(\frac12+ib\Big)\right).$$
It was proved in \cite[Corollary 1.5]{CL-JDG19} that there exists $\tilde b\in (\frac5{24},\frac1{2\sqrt{3}})$ such that 
\begin{equation}\label{derivative-eta}\frac{d}{db}\eta_1\Big(\frac12+ib\Big)\begin{cases}>0&\text{for }0<b<\tilde b,\\
=0&\text{for }b=\tilde b,\\
<0&\text{for }b>\tilde b.\end{cases}\end{equation}
Together these with $\frac{\tau-1}{2\tau-1}=\frac12+\frac{1}{4b}i$ for $\tau=\frac12+ib$, by letting $b_2=\frac{1}{4\tilde b}\in (\frac{\sqrt{3}}{2}, \frac65)$, we easily see that \eqref{eqq-413} holds.
\end{proof}


\begin{lemma}\label{lemma-aup-2}
Let $\tau=\frac12+ib$ with $b>b_1$. Recall Lemma \ref{lemma-aup} that for any $a\in I\cup II\cup III$, there is a unique $p\in I\cup II\cup III\cup \{0,\frac12\}$ such that  $\wp(p)=f(a)$ and $\pm a$ is a pair of nontrivial critical points of $G_p(z)$, where $f(a)$ is defined in \eqref{513-1-0}. Recall $z_0(\tau)\in II^\circ$ defined in Theorem \ref{thm-LW-16}.
If $a\in II=(\frac{1}{2},\frac{1}{4}+\frac{\tau}{2}]$, then $$f(z_0(\tau))=\infty,\qquad f\Big(\frac{1}{4}+\frac{\tau}{2}\Big)=\wp\Big(\frac14\Big),$$ and the following statements hold.
\begin{itemize}
\item[(1)] $f: (z_0(\tau), \frac14+\frac{\tau}2)\to f((z_0(\tau), \frac14+\frac{\tau}2))$ is a homeomorphism, and
\begin{equation}\label{dddd2-0}
f\Big(\Big(z_0(\tau), \frac14+\frac{\tau}2\Big)\Big)=\Big(\wp\Big(\frac14\Big),+\infty\Big).
\end{equation}
\item[(2)] We have
\begin{align}\label{dddd2}(-\infty, \wp(p_{1,1}))\subset f\Big(\Big(\frac12, z_0(\tau)\Big)\Big)\subset  \Big(-\infty, \wp\Big(\frac14-\frac{\tau}{2}\Big)\Big).\end{align}
Furthermore, recall $b_2\in (\sqrt{3}/2, 6/5)$ defined in Lemma \ref{lem04-4}.
\begin{itemize}
\item[(2-1)]  If $b\in (b_1, b_2]$, then $f: (\frac12, z_0(\tau))\to f((\frac12, z_0(\tau)))=(-\infty, \wp(p_{1,1}))$ is a homeomorphism.
\item[(2-2)] If $b>b_2$, then there is $\tilde{p}_{1,1}=\tilde{p}_{1,1}(b)\in (\frac14-\frac{\tau}{2}, p_{1,1})$ such that
$$f\Big(\Big(\frac12, z_0(\tau)\Big)\Big)=(-\infty, \wp(\tilde p_{1,1})].$$
More precisely, there is a unique $\tilde{a}_{1,1}\in (\frac12, z_0(\tau))$ such that
\begin{equation}\label{eqq-64}\begin{cases}f(\tilde{a}_{1,1})=\wp(\tilde p_{1,1}),\\
f: (\tilde{a}_{1,1}, z_0(\tau))\to  (-\infty, \wp(\tilde p_{1,1}))\quad\text{is a homeomorphism},\\
f: (\frac12,\tilde{a}_{1,1})\to  (\wp(p_{1,1}), \wp(\tilde p_{1,1}))\quad\text{is a homeomorphism}.\end{cases}\end{equation}
In particular, $G_p(z)$ has exactly two pairs of nontrivial critical points $\pm a_j$ satisfying $a_1, a_2\in (\frac12, z_0(\tau))$ for any $p\in (\tilde{p}_{1,1}, p_{1,1})$, and $\tilde{a}_{1,1}$ is a degenerate nontrivial critical point of $G_{\tilde p_{1,1}}(z)$.
\end{itemize}
\end{itemize}
\end{lemma}

\begin{proof}
Let $a=r+s\tau\in II=(\frac{1}{2},\frac{1}{4}+\frac{\tau}{2}]\subset\frac12+i\mathbb R$. 
Then $s=1-2r$, i.e., $a=r+(1-2r)\tau$ and so $a-\frac12=(1-2\tau)(r-\frac12)$ with $r\in [\frac14, \frac12)$. 

{\bf Step 1.} We prove \eqref{dddd2-0}-\eqref{dddd2}.

When $\frac12-r>0$ is sufficiently small, 
$$\wp(a)=e_1+\frac{\wp''(\frac12)}{2}\Big(a-\frac12\Big)^2+O\left(\Big(\frac12-r\Big)^4\right), $$
$$\wp'(a)=\wp''\Big(\frac12\Big)\Big(a-\frac12\Big)+\frac{\wp^{(4)}(\frac12)}{6}\Big(a-\frac12\Big)^3+O\left(\Big(\frac12-r\Big)^5\right),$$
and
\begin{align*}&\zeta(a)-r\eta_1-s\eta_2\\
=&-\Big[e_1+\eta_1-\frac{4\pi i}{2\tau-1}\Big]\Big(a-\frac12\Big)-\frac{\wp''(\frac12)}{6}\Big(a-\frac12\Big)^3+O\left(\Big(\frac12-r\Big)^5\right).\end{align*}
Note that $\wp^{(4)}(\frac12)=12e_1\wp''(\frac12)$ and $\wp''(\frac12)=6e_1^2-g_2/2=4e_1^2+2|e_2|^2>0$. Consequently,
{\allowdisplaybreaks
\begin{align}\label{eqq-61}\wp(p)&=f(a)=\wp(a)+\frac{\wp'(a)}{2(\zeta(a)-r\eta_1-s\eta_2)}\nonumber\\
&=e_1-\frac{\wp''(\frac12)}{2(e_1+\eta_1-\frac{2\pi}{b})}+\hat{\beta}\Big(a-\frac12\Big)^2+O\left(\Big(\frac12-r\Big)^4\right)\nonumber\\
&=e_1-\frac{3e_1^2-\frac{g_2}{4}}{e_1+\eta_1-\frac{2\pi}{b}}+\hat{\beta}\Big(a-\frac12\Big)^2+O\left(\Big(\frac12-r\Big)^4\right)\nonumber\\&=\wp(p_{1,1})+\hat{\beta}\Big(a-\frac12\Big)^2+O\left(\Big(\frac12-r\Big)^4\right),\end{align}
}%
where, by using $\wp''(\frac12)=6e_1^2-g_2/2>0$,
{\allowdisplaybreaks
\begin{align*}\hat{\beta}:&=\frac{\wp''(\frac12)}{2(e_1+\eta_1-\frac{2\pi}{b})}\left(\Big(e_1+\eta_1-\frac{2\pi}{b}\Big)-2e_1+\frac{\wp''(\frac12)}{6(e_1+\eta_1-\frac{2\pi}{b})}\right)\\
&=\frac{\wp''(\frac12)}{24(e_1+\eta_1-\frac{2\pi}{b})^2}\left(12\Big(e_1+\eta_1-\frac{2\pi}{b}\Big)^2-24e_1\Big(e_1+\eta_1-\frac{2\pi}{b}\Big)+12e_1^2-g_2\right)\\
&=\frac{\wp''(\frac12)}{24(e_1+\eta_1-\frac{2\pi}{b})^2}\left(12\Big(\eta_1-\frac{2\pi}{b}\Big)^2-g_2\right)=\frac{\wp''(\frac12)}{24(e_1+\eta_1-\frac{2\pi}{b})^2}\beta(b).\end{align*}
}%

Recall \cite[Lemma 5.5-(1)]{CFL-II} that $G_{\frac14}(z)$ has a unique pair of nontrivial critical points $\pm(\frac14+\frac{\tau}{2})$, so 
\begin{equation}\label{eqq-f5}f\Big(\frac14+\frac{\tau}{2}\Big)=\wp\Big(\frac14\Big),\qquad
f(a)\neq\wp\Big(\frac14\Big)\quad\text{for any }a\in II^\circ,\end{equation}
where $II^\circ:=(\frac12, \frac12+\frac{\tau}4)$.
 Also recall Theorem \ref{thm-LW-16} that $G_0(z)=G(z)$ has a unique pair of nontrivial critical points $\pm z_0(\tau)$ with $z_0(\tau)\in II^\circ$, so \begin{equation}\label{eq4-22-0}f(z_0(\tau))=\infty\in f(II^\circ),\qquad f(a)\neq\infty\quad\text{for any } a\in II\setminus\{z_0\}.\end{equation} From \eqref{eqq-61}, \eqref{eqq-f5}, \eqref{eq4-22-0} and $f: II^\circ\to f(II^{\circ})$ is continuous, i.e., $f(II^{\circ})\subset\mathbb R\cup\{\infty\}$ is path-connected, we obtain
\begin{equation}\label{eq4-22} (-\infty, \wp(p_{1,1}))\cup \Big(\wp\Big(\frac14\Big),+\infty\Big)
\cup\{\infty\}\subset f(II^\circ).\end{equation}
Recall \eqref{eq4-11} that $G_{\frac14-\frac{\tau}{2}}(z)=G_{\frac14+\frac{\tau}{2}}(z)$ has a unique pair of nontrivial critical points $\pm\frac14\notin II$, so $\wp(\frac14-\frac{\tau}{2})\notin f(II)$. This, together with \eqref{eqq-61}, \eqref{eqq-f5}, \eqref{eq4-22-0} and \eqref{eq4-22}, implies that \eqref{dddd2-0}-\eqref{dddd2} hold.
In particular, $p\in III^\circ$ for $a\in (\frac{1}{2}, z_0(\tau))$.

{\bf Step 2.} We prove that $f: (z_0(\tau), \frac14+\frac{\tau}2)\to (\wp(\frac14),+\infty)$ is a homeomorphism.

For any $a\in  (z_0(\tau), \frac14+\frac{\tau}2)$, we have $\wp(a)<e_1<\wp(p)$, so it follows from Remark \ref{rmk-s5-2} and \eqref{Ak} that the corresponding 
\begin{align}\label{eqq-Aj}A(a)&=\frac{1}{2}\left[  \zeta(p+a)+\zeta(p-a)-\zeta(2p)\right]\\
&=A_0+\frac{\wp'(p)}{2(\wp(p)-\wp(a))}\in(A_1, A_0).\nonumber\end{align}
Since Lemma \ref{Lemma53-5} shows \begin{equation}\label{eq4-22-1}\#(\sigma_1\cap\sigma_3\cap (A_1, A_0))\leq 1,\quad\forall p\in I,\end{equation}  we conclude from Theorem \ref{rmk2-8} that
for any $\wp(p)\in (\wp(\frac14),+\infty)$, $f^{-1}(\wp(p))$ contains exactly one preimage in $(z_0(\tau), \frac14+\frac{\tau}2)$. Thus, $f: (z_0(\tau), \frac14+\frac{\tau}2)\to (\wp(\frac14),+\infty)$ is bijective. The inverse map $f^{-1} : (\wp(\frac14),+\infty)\to (z_0(\tau), \frac14+\frac{\tau}2)$ is also continuous because $a=f^{-1}(\wp(p))$ is the nontrivial critical point of $G_p(z)$. This completes the proof of Step 2.



{\bf Step 3.} We consider $a\in (\frac12, z_0(\tau))$.

For any $a\in (\frac12, z_0(\tau))$, we see from \eqref{dddd2} that $$\wp(p)=f(a)<\wp\Big(\frac14-\frac{\tau}2\Big)<\wp(a)<e_1,$$ then the same proof as \eqref{eqq-Aj} shows that the corresponding \begin{equation}\label{eqq04-1}A(a)=\frac{1}{2}\left[  \zeta(p+a)+\zeta(p-a)-\zeta(2p)\right]\in (-i\infty, A_1).\end{equation} Since Lemma \ref{Lemma53-7} shows \begin{equation}\label{eq4-24}\#(\sigma_1\cap\sigma_3\cap (-i\infty, A_1))\leq 2,\qquad\forall p\in III,\end{equation}we conclude that
\begin{equation} \label{eq: data3}
  \parbox{\dimexpr\linewidth-5em}{
 for any $\wp(p)\in f((\frac12, z_0(\tau)))$, $f^{-1}(\wp(p))$ contains at most two preimages in $(\frac12, z_0(\tau))$.
  }
\end{equation}
We consider two cases.

{\bf Case 3-1. }
Suppose $b_1<b\leq b_2$. We prove that  $f: (\frac12, z_0(\tau))\to (-\infty, \wp(p_{1,1}))$ is a homeomorphism.

Assume first that $b_1<b<b_2$. Then $\beta(b)>0$. From here and $(a-\frac12)^{2}<0$, we see from  \eqref{eqq-61} that 
$$\wp(p)=f(a)\uparrow \wp(p_{1,1})\quad\text{as} \quad \Big(\frac12, z_0(\tau)\Big)\ni a\to \frac12.$$
Recall from \eqref{dddd2} that
$$(-\infty, \wp(p_{1,1}))\subset f\Big(\Big(\frac12, z_0(\tau)\Big)\Big).$$
Assume by contradiction that $(-\infty, \wp(p_{1,1}))\subsetneqq f((\frac12, z_0(\tau)))$, then  there are $p\in (p_{1,1}, 0)$ close to $p_{1,1}$ and $a_1\neq a_2\neq a_3\neq a_1\in (\frac12, z_0(\tau))$ such that $\wp(p)=f(a_1)=f(a_2)=f(a_3)$, a contradiction with \eqref{eq: data3}.
This proves $(-\infty, \wp(p_{1,1}))=f((\frac12, z_0(\tau)))$. Repeating this argument again, we conclude that $f: (\frac12, z_0(\tau))\to (-\infty, \wp(p_{1,1}))$ is indeed bijective and then a homeomorphism.

Now we consider $b=b_2$. Assume by contradiction that $(-\infty, \wp(p_{1,1}))\subsetneqq f((\frac12, z_0(\tau)))$, i.e., $\wp(p_{1,1})\in f((\frac12, z_0(\tau)))$. Then by \eqref{eqq04-1} and \eqref{eq4-24}, we have
\begin{equation}\label{eqq04-8}1\leq\#(\sigma_1\cap\sigma_3\cap (-i\infty, A_1))\leq 2,\quad \text{for }p=p_{1,1}.\end{equation}
However, for $p=p_{1,1}$, it follows from Corollary \ref{coro-s5-10} that $A_1$ is a cusp of $\sigma_3$.
Since $b=b_2$ implies \eqref{eqq0-4}, we see from Corollary \ref{coro2-11-3}-(3) that $\sigma_3$ consists of $5$ analytic arcs that have the common endpoint $A_1$. This, together with Lemma \ref{Lemma53-7}, yields that $\sigma_3=[A_1, A_0]\cup\sigma_{3,A_2A_3}\cup\sigma_{3,\infty}$ with
\begin{equation}\label{eqq04-21}\sigma_{3,A_2A_3}\cap i\mathbb R=\sigma_{3,\infty}\cap i\mathbb{R}=\{A_1\},\end{equation}
so $\sigma_3\cap (-i\infty, A_1)=\emptyset$, a contradiction with \eqref{eqq04-8}. This proves $(-\infty, \wp(p_{1,1}))=f((\frac12, z_0(\tau)))$ for $b=b_2$. Then by repeating the same argument as $b_1<b<b_2$, we obtain that $f: (\frac12, z_0(\tau))\to (-\infty, \wp(p_{1,1}))$ is indeed bijective and then a homeomorphism also for $b=b_2$.

{\bf Case 3-2.} Suppose $b>b_2$. We prove \eqref{eqq-64}.

Since $b>b_2$ implies $\beta(b)<0$, we see from
\eqref{eqq-61} that 
\begin{equation}\label{eqq-63}\wp(p)=f(a)\downarrow \wp(p_{1,1})\quad\text{as} \quad \Big(\frac12, z_0(\tau)\Big)\ni a\to \frac12,\end{equation}
so it follows from \eqref{dddd2} that $(-\infty, \wp(p_{1,1})]\subsetneqq f((\frac12, z_0(\tau)))$. Since $f((\frac12, z_0(\tau)))\subset  (-\infty, \wp(\frac14-\frac{\tau}{2}))$ is path-connected, there is $\tilde{p}_{1,1}=\tilde{p}_{1,1}(b)\in [\frac14-\frac{\tau}{2}, p_{1,1})$ such that
$$(-\infty, \wp(\tilde p_{1,1}))\subset f\Big(\Big(\frac12, z_0(\tau)\Big)\Big)\subset (-\infty, \wp(\tilde p_{1,1})].$$
Take $a_m\in (\frac12, z_0(\tau))$ such that $f(a_m)\uparrow \wp(\tilde{p}_{1,1})$, then up to a subsequence, $a_m\to \tilde a\in (\frac12, z_0(\tau))$ and so $\wp(\tilde{p}_{1,1})=f(\tilde a)$. This proves $\tilde{p}_{1,1}\in (\frac14-\frac{\tau}{2}, p_{1,1})$ and
$$f\Big(\Big(\frac12, z_0(\tau)\Big)\Big)=(-\infty, \wp(\tilde p_{1,1})].$$
Together with \eqref{eqq-63} and \eqref{eq: data3}, we easily see that there is a unique $\tilde{a}_{1,1}\in (\frac12, z_0(\tau))$ such that \eqref{eqq-64} holds.
The proof is complete.
\end{proof}

%

\begin{lemma}\label{lemma-aup-3}
Let $\tau=\frac12+ib$ with $b>b_1$. Recall Lemma \ref{lemma-aup} that for any $a\in I\cup II\cup III$, there is a unique $p\in I\cup II\cup III\cup \{0,\frac12\}$ such that  $\wp(p)=f(a)$ and $\pm a$ is a pair of nontrivial critical points of $G_p(z)$, where $f(a)$ is defined in \eqref{513-1-0}. Recall $\frac12-z_0(\tau)\in III^{\circ}$ defined in Theorem \ref{thm-LW-16}.
If $a\in III=[\frac14-\frac{\tau}{2}, 0)$, then $$f\Big(\frac12-z_0(\tau)\Big)=e_1,\quad f\Big(\frac{1}{4}-\frac{\tau}{2}\Big)=\wp\Big(\frac14\Big),$$ and the following statements hold.
\begin{itemize}
\item[(1)] $f: (\frac14-\frac{\tau}2, \frac12-z_0(\tau))\to f((\frac14-\frac{\tau}2, \frac12-z_0(\tau)))$ is a homeomorphism, and
\begin{equation}\label{dddd3-0}
f\Big(\Big(\frac14-\frac{\tau}2, \frac12-z_0(\tau)\Big)\Big)=\Big(e_1,\wp\Big(\frac14\Big)\Big).
\end{equation}
\item[(2)] We have
\begin{align}\label{dddd3}(\wp(p_{0,2}), e_1)\subset f\Big(\Big(\frac12-z_0(\tau), 0\Big)\Big)\subset  \Big(\wp\Big(\frac14+\frac{\tau}{2}\Big), e_1\Big).\end{align}
Furthermore, recall $b_2\in (\sqrt{3}/2, 6/5)$ defined in Lemma \ref{lem04-4}. 
\begin{itemize}
\item[(2-1)] If $b\in (b_1, b_2]$, then $f: (\frac12-z_0(\tau),0)\to f((\frac12-z_0(\tau),0))=(\wp(p_{0,2}), e_1)$ is a homeomorphism.
\item[(2-2)] If $b>b_2$, recalling $\tilde{p}_{1,1}, \tilde{a}_{1,1}$ in Lemma \ref{lemma-aup-2} (2-2), we define
\begin{equation}\label{eq4-34}
\tilde{p}_{0,2}:=\frac12-\tilde{p}_{1,1}\in \Big(p_{0,2},\frac14+\frac{\tau}{2}\Big),\quad\tilde{a}_{0,2}:=\frac12-\tilde{a}_{1,1}\in \Big(\frac12-z_0(\tau),0\Big),
\end{equation}
then
$$f\Big(\Big(\frac12-z_0(\tau),0\Big)\Big)=[\wp(\tilde p_{0,2}), e_1),$$
and more precisely,
\begin{equation}\label{eqq-64-2}\begin{cases}f(\tilde{a}_{0,2})=\wp(\tilde p_{0,2}),\\
f: (\frac12-z_0(\tau),\tilde{a}_{0,2})\to  (\wp(\tilde p_{0,2}), e_1)\quad\text{is a homeomorphism},\\
f: (\tilde{a}_{0,2}, 0)\to  (\wp(\tilde p_{0,2}), \wp(p_{0,2}))\quad\text{is a homeomorphism}.\end{cases}\end{equation}
In particular, $G_p(z)$ has exactly two pairs of nontrivial critical points $\pm a_j$ satisfying $a_1, a_2\in (\frac12-z_0(\tau),0)$ for any $p\in (p_{0,2}, \tilde{p}_{0,2})$, and $\tilde{a}_{0,2}$ is a degenerate nontrivial critical point of $G_{\tilde p_{0,2}}(z)$.
\end{itemize}
\end{itemize}
\end{lemma}

\begin{proof} There are two proofs for this lemma. The first is to follow the proof of Lemma \ref{lemma-aup-2}.
Indeed, let $a=r+s\tau\in III=[\frac14-\frac{\tau}{2}, 0)$,
 then $s=-2r$, i.e., $a=(1-2\tau)r$ with $r\in (0, \frac14]$. 
Hence, for $r>0$ is sufficiently small,$$\wp(a)=\frac{1}{a^2}+\frac{g_2}{20}a^2+O(r^4), \quad \wp'(a)=\frac{-2}{a^3}+\frac{g_2}{10}a+O(r^3),$$
$$\zeta(a)-r\eta_1-s\eta_2=\frac1{a}+\Big(\frac{4\pi i}{2\tau-1}-\eta_1\Big)a-\frac{g_2}{60}a^3+O(r^5),$$
so 
\begin{align}\label{eqq-61-2}\wp(p)&=f(a)=\wp(a)+\frac{\wp'(a)}{2(\zeta(a)-r\eta_1-s\eta_2)}\nonumber\\
&=\frac{2\pi}{b}-\eta_1+\Big(\frac{g_2}{12}-\Big(\frac{2\pi}{b}-\eta_1\Big)^2\Big)a^2+O(r^4)\nonumber\\
&=\wp(p_{0,2})-\frac1{12}\beta(b)a^2+o(1).\end{align}
Consequently, following the argument of Lemma \ref{lemma-aup-2}, we can obtain all the assertions except the precise relation \eqref{eq4-34} between $\tilde{p}_{0,2}, \tilde{a}_{0,2}$ and $\tilde{p}_{1,1}, \tilde{a}_{1,1}$. We omit the details.

Here we use a different idea to prove this lemma. Note that $p\in II$ if and only if $\frac12-p\in III$, and $p\in (0,\frac12)$ if and only if $\frac12-p\in (0, \frac12)$. Since $G(z)=G(-z)$ implies
$G_p(\frac12-z)=G_{\frac12-p}(z)$, we see that 
\begin{equation} \label{eq: da11ta3}
  \parbox{\dimexpr\linewidth-5em}{
 $a$ is a nontrivial critical point of $G_p(z)$ if and only if $\frac12-a$ is a nontrivial critical point $G_{\frac12-p}(z)$.
  }
\end{equation}
 From here, we immediately see that Lemma \ref{lemma-aup-3} follows directly from Lemma \ref{lemma-aup-2} and Lemma \ref{lemma-s5-9}, from which we can further obtain that if $b>b_2$, i.e., $\beta(b)<0$, then  $$\tilde{p}_{0,2}=\frac12-\tilde{p}_{1,1},\quad\tilde{a}_{0,2}=\frac12-\tilde{a}_{1,1}.$$
The proof is complete.
\end{proof}

\begin{theorem}[=Theorem \ref{III-thm3}]\label{thm4-7}
Let $\tau=\frac12+ib$ with $b_1<b\leq b_2$. Then the following statements hold.
\begin{itemize}
\item[(1)] When $\wp(p)\in [\wp(p_{1,1}), \wp(p_{1,2})]\cup [\wp(p_{0,1}), \wp(p_{0,2})]$, $G_p(z)$ has no nontrivial critical points satisfying $\wp(a)\in\mathbb{R}$.
\item[(2)] When $\wp(p)\in (-\infty, \wp(p_{1,1}))\cup(\wp(p_{1,2}), \wp(p_{0,1}))\cup(\wp(p_{0,2}),+\infty)$, $G_p(z)$ has a unique pair of nontrivial critical points $\pm a$ satisfying $\wp(a)\in\mathbb{R}$.
\end{itemize}
\end{theorem}

\begin{proof} This result follows directly from Lemmas \ref{lemma-aup-1}, \ref{lemma-aup-2} and \ref{lemma-aup-3}. Consequently, by defining
$$d_1:=\wp(p_{1,1})=e_1+\frac{3e_1^2-\frac{g_2}{4}}{\frac{2\pi}{b}-e_1-\eta_1}, \quad d_2:=\wp(p_{1,2})=e_1-\frac{3e_1^2-\frac{g_2}{4}}{e_1+\eta_1},$$
$$d_3:=\wp(p_{0,1})=-\eta_1,\quad d_4:=\wp(p_{0,2})=\frac{2\pi}{b}-\eta_1,$$
we obtain Theorem \ref{III-thm3}.
\end{proof}

\begin{lemma}\label{lem04-9}
Let $\tau=\frac12+ib$ with $b>b_2$, and recalling $\tilde{p}_{1,1}=\tilde{p}_{1,1}(b)\in (\frac14-\frac{\tau}{2}, p_{1,1})$ given in Lemma \ref{lemma-aup-2}. Then $$\wp(\tilde{p}_{1,1})< \wp(p_{1,2}),\quad\text{i.e.,}\quad \tilde{p}_{1,1}\in (p_{1,2}, p_{1,1}),$$ or equivalently,
$$\wp(\tilde{p}_{0,2})>\wp(p_{0,1}),\quad\text{i.e.,}\quad \tilde{p}_{0,2}\in (p_{0,2}, p_{0,1}).$$
\end{lemma}

\begin{figure}[htbp]
\centering
\begin{tikzpicture}[ scale=0.7,
    dot/.style={circle,fill=black,inner sep=1pt}
]
\draw[thick, red] (0,-2)--(0,0) -- (2,0) -- (2,2);
\draw[thick, black] (3,2)--(-1,2)--(-3,-2)--(1,-2)--cycle;

\node[dot] at (2, 1.5) (p01) {};
\node[dot] at (2, 0.9) (z0) {};
\node[dot] at (2, 0.5) (p02) {};
\node[dot] at (0, -0.5) (p11) {};
\node[dot] at (0, -0.9) (mid) {};
\node[dot] at (0, -1.5) (p12) {};

\node[above] at (1,0) {I};
\node[left] at (2,1.2) {II};
\node[right] at (0,-1.2) {III};
\node[left] at (0,0) {\scriptsize$0$};
\node[right] at (2,0) {\scriptsize$\frac12$};
\node[above] at (2,2) {\scriptsize$\frac14+\frac{\tau}2$};
\node[right] at (p01) {\scriptsize$p_{0,1}$};
\node[right] at (z0) {\scriptsize$\tilde p_{0,2}$};
\node[left] at (p02) {\scriptsize$p_{0,2}$};
\node[left] at (p11) {\scriptsize$p_{1,1}=\frac12-p_{0,2}$};
\node[left] at (mid) {\scriptsize$\tilde p_{1,1}=\frac12-\tilde p_{0,2}$};
\node[left] at (p12) {\scriptsize$p_{1,2}=\frac12-p_{0,1}$};
\node[below] at (0,-2) {\scriptsize$\frac14 - \frac\tau2$};
\end{tikzpicture}
\captionof{figure}{Relative locations of those points in Lemma \ref{lem04-9}}
\label{figure4-6}
\end{figure}
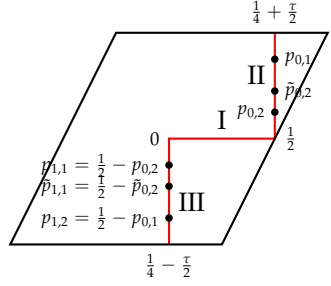

\begin{proof} 
Clearly $\sigma_j$ depends on $\tau=\frac12+ib$ and $p$, and in this proof, we write $\sigma_j=\sigma_j(b, p)$ to emphasize its dependence on $b$ and $p$.

From the proof of Lemma \ref{lemma-aup-2}, it is clear that $\tilde{p}_{1,1}(b)$ is continuous as a fuction of $b>b_2$ and $\lim_{b\to b_2}\tilde{p}_{1,1}=p_{1,1}$.

{\bf Step 1.} We consider the graph of $\sigma_1(b,p)$ for $b\geq b_2$ and $p\in [p_{1,2}, p_{1,1}]$.

Since $p\in [p_{1,2}, p_{1,1}]\subset III$, it follows from Lemma \ref{Lemma53-7} that $$\sigma_1(b,p)=(-i\infty, A_1]\cup [A_0, +i\infty)\cup\sigma_{1,A_2A_3}(b,p),$$ where $\sigma_{1,A_2A_3}(b,p)$ is a simple curve connecting $A_2$ and $A_3=-\overline{A_2}$, and is symmetric with respect to the imaginary axis with $$\sigma_{1,A_2A_3}(b,p)\cap i\mathbb{R}=\{\text{a single point}\}=:\{\hat A(b,p)\}.$$ 
Since $b\geq b_2>b_1$, by Corollary \ref{coro-s5-10} we have
\begin{equation}\label{eeqq-01}
\hat A(b,p_{1,2})=A_1,\qquad \hat A(b,p)\notin\{A_0, A_1\}\;\text{for any }p\in (p_{1,2}, p_{1,1}].
\end{equation}
Then by continuity, we have either $\hat A(b,p)\in (-i\infty, A_1)$ for any $p\in (p_{1,2}, p_{1,1}]$ or 
$\hat A(b,p)\in (A_1, A_0)$ for any $p\in (p_{1,2}, p_{1,1}]$. 

We claim that $\hat A(b,p)\in (-i\infty, A_1)$ for all  $b\geq b_2$ and $p\in (p_{1,2}, p_{1,1}]$.
Assume by contradiction that $\hat A(b,p)\in (A_1, A_0)\subset \sigma_3$ for any $p\in (p_{1,2}, p_{1,1}]$, then $\hat A(b,p)\in\sigma_1\cap\sigma_3\setminus\{A_k\}_{k=0}^3$, so the corresponding $\pm a(b,p)$ of  $\hat A(b,p)$ is a pair of nontrivial critical points of $G_p(z)$, and we see from $\hat A(b,p)\in i\mathbb{R}$ and Lemma \ref{lemma-s5-13} that $\wp(a(b,p))\in\mathbb R$, so we may assume $a(b,p)\in I\cup II\cup III$. Together with $p\in (p_{1,2}, p_{1,1}]$, we see from  Lemmas \ref{lemma-aup-1}, \ref{lemma-aup-2} and \ref{lemma-aup-3}  that $a(b,p)\in (\frac12, z_0(\tau))\subset II^\circ$, which together with \eqref{eqq04-1} yields that $\hat A(b,p)\in (-i\infty, A_1)$, a contradiction. This proves that
\begin{equation}\label{eeq-01}
\sigma_{1,A_2A_3}(b,p)\cap i\mathbb{R}=\{\hat A(b,p)\}\subset (-i\infty, A_1),\;\forall b\geq b_2,\; p\in (p_{1,2}, p_{1,1}].
\end{equation}
In particular, 
\begin{equation} \label{eq: da12ta3}
  \parbox{\dimexpr\linewidth-5em}{
$\hat A(b,p)$ is a branch point of $\sigma_1(b,p)$ for any $b\geq b_2$ and $p\in (p_{1,2}, p_{1,1}]$.  }
\end{equation}

{\bf Step 2.} We consider the graph of $\sigma_3(b,\tilde{p}_{1,1}(b))$ for $b> b_2$.

For $b>b_2$ and $p\in (\tilde{p}_{1,1}(b), p_{1,1})$, Lemma \ref{lemma-aup-2} implies that $G_p(z)$ has exactly two pairs of nontrivial critical points $\pm a_j(b, p)$ satisfying $a_1(b,p), a_2(b,p)\in (\frac12, z_0(\tau))$, and it follows from \eqref{eqq04-1} that the corresponding 
\begin{equation}
\label{eqq04-001}\tilde{A}_j(b,p):=\frac{1}{2}\left[  \zeta(p+a_j(b,p))+\zeta(p-a_j(b,p))-\zeta(2p)\right]\in (-i\infty, A_1).\end{equation}
Then by $\tilde{A}_j(b,p)\in \sigma_1(b,p)\cap\sigma_3(b,p)$ and Lemma \ref{Lemma53-7}, we have $\sigma_3(b, p)=[A_1, A_0]\cup\sigma_{3,A_2A_3}(b, p)\cup\sigma_{3,\infty}(b,p)$ with
\begin{align}\label{eqq04-022}&\sigma_{3,A_2A_3}(b, p)\cap i\mathbb R=\{\tilde{A}_j(b,p)\},\\
&\sigma_{3,\infty}(b, p)\cap i\mathbb{R}=\{\tilde{A}_{3-j}(b,p)\}, \;\text{for some $j\in\{1,2\}$}.\nonumber\end{align}
When $(\tilde{p}_{1,1}(b), p_{1,1})\ni p\to \tilde{p}_{1,1}(b)$, we see from Lemma \ref{lemma-aup-2}  that the two critical points  $a_1(b,p), a_2(b,p)$ converge to the same critical point $\tilde a_{1,1}=\tilde a_{1,1}(b)$ of $G_{\tilde{p}_{1,1}(b)}(z)$, so the corresponding $\tilde{A}_1(b,p), \tilde{A}_2(b,p)$ converge to the corresponding \begin{align*}\tilde{A}(b,\tilde{p}_{1,1}(b)):=&\frac{1}{2}\left[  \zeta(\tilde{p}_{1,1}(b)+\tilde a_{1,1}(b))+\zeta(\tilde{p}_{1,1}(b)-\tilde a_{1,1}(b))-\zeta(2\tilde{p}_{1,1}(b))\right]\\
\in& (-i\infty, A_1).\end{align*}
By the continuous deformation of $\sigma_3(b,p)$ as  $(\tilde{p}_{1,1}(b), p_{1,1})\ni p\to \tilde{p}_{1,1}(b)$, we see from \eqref{eqq04-022} that 
$\sigma_3(b, \tilde{p}_{1,1}(b))=[A_1, A_0]\cup\sigma_{3,A_2A_3}(b, \tilde{p}_{1,1}(b))\cup\sigma_{3,\infty}(b,\tilde{p}_{1,1}(b))$ with
\begin{equation}\label{eqq04-002}\sigma_{3,A_2A_3}(b, \tilde{p}_{1,1}(b))\cap i\mathbb R=\sigma_{3,\infty}(b, \tilde{p}_{1,1}(b))\cap i\mathbb{R}=\{\tilde{A}(b,\tilde{p}_{1,1}(b))\}\in (-i\infty, A_1).\end{equation}
In particular,
\begin{equation} \label{eq: da13ta3}
  \parbox{\dimexpr\linewidth-5em}{
$\tilde{A}(b,\tilde{p}_{1,1}(b))$ is a branch point of $\sigma_3(b, \tilde{p}_{1,1}(b))$ for any $b>b_2$.  }
\end{equation}
Also recall \eqref{eqq04-21} that $\sigma_3(b_2, p_{1,1})=[A_1, A_0]\cup\sigma_{3,A_2A_3}(b_2, p_{1,1})\cup\sigma_{3,\infty}(b_2, p_{1,1})$ with
\begin{equation}\label{eqq04-22}\sigma_{3,A_2A_3}(b_2, p_{1,1})\cap i\mathbb R=\sigma_{3,\infty}(b_2, p_{1,1})\cap i\mathbb{R}=\{A_1\}.\end{equation}

{\bf Step 3.} Recalling $\hat A(b, p)$ in \eqref{eeq-01}, we prove that for $b-b_2>0$ small, $\tilde{p}_{1,1}(b)\in (p_{1,2}, p_{1,1})$ and $\tilde{A}(b,\tilde{p}_{1,1}(b))\in (\hat{A}(b,\tilde{p}_{1,1}(b)), A_1)$, or equivalently, \begin{equation}\label{eeqq-02}\operatorname{Im}\tilde{A}(b,\tilde{p}_{1,1}(b))>\operatorname{Im}\hat{A}(b,\tilde{p}_{1,1}(b))\quad\text{for $b-b_2>0$ small}.\end{equation}

Since $\tilde{p}_{1,1}(b)\in (\frac14-\frac{\tau}{2}, p_{1,1})$ and $\lim_{b\downarrow b_2}\tilde{p}_{1,1}(b)=p_{1,1}$, we see that  for $b-b_2>0$ small, $\tilde{p}_{1,1}(b)\in (p_{1,2}, p_{1,1})$, so it follows from \eqref{eeq-01} that $\hat A(b,\tilde{p}_{1,1}(b))$ is well-defined. Recall $\lim_{b\downarrow b_2}\tilde{p}_{1,1}(b)=p_{1,1}$.
Then \eqref{eeq-01}, \eqref{eqq04-002} and \eqref{eqq04-22} imply  $$\lim_{b\downarrow b_2}\hat{A}(b,\tilde{p}_{1,1}(b))=\hat{A}(b_2, p_{1,1})\in (-i\infty, A_1),\quad \lim_{b\downarrow b_2}\tilde{A}(b,\tilde{p}_{1,1}(b))=A_1,$$ we obtain \eqref{eeqq-02}.

{\bf Step 4.} We prove that for any $b>b_2$, $\tilde{p}_{1,1}(b)\in (p_{1,2}, p_{1,1})$.

Assume by contradiction that $\tilde{p}_{1,1}(\tilde {b})\in (\frac14-\frac\tau2, p_{1,2}]$ for some $\tilde b>b_2$. Then by Step 3 and the continuity, there is the first $b_3>b_2$ such that $\tilde{p}_{1,1}(b_3)=p_{1,2}(b_3)$ (Note that $p_{1,2}=p_{1,2}(b)$ also continuously depends on $b$) and $$\tilde{p}_{1,1}(b)\in (p_{1,2}(b), p_{1,1}(b)),\quad\text{for }b\in (b_2, b_3).$$ Since \eqref{eeqq-01} says $\hat A(b_3$, $p_{1,2}(b_3))=A_1$, we see from \eqref{eqq04-002} that $$\operatorname{Im}\tilde{A}(b_3,\tilde{p}_{1,1}(b_3))<
\operatorname{Im}\hat{A}(b_3,\tilde{p}_{1,1}(b_3)).$$ Together with \eqref{eeqq-02}, there is $b_4\in (b_2, b_3)$ such that $$\operatorname{Im}\tilde{A}(b_4,\tilde{p}_{1,1}(b_4))=\operatorname{Im}\hat{A}(b_4,\tilde{p}_{1,1}(b_4)),$$ i.e., $$\tilde{A}(b_4,\tilde{p}_{1,1}(b_4))=\hat{A}(b_4,\tilde{p}_{1,1}(b_4))\in (-i\infty, A_1).$$ Consequently, it follows from \eqref{eq: da12ta3} and \eqref{eq: da13ta3} that $\tilde{A}(b_4,\tilde{p}_{1,1}(b_4))$ is a common branch point of $\sigma_1(b_4,\tilde{p}_{1,1}(b_4))$ and $\sigma_3(b_4,\tilde{p}_{1,1}(b_4))$, a contradiction with Lemma \ref{lemma2-10-3}.
This completes the proof.
\end{proof}

\begin{theorem}[=Theorem \ref{III-thm4}]\label{thm4-9}
Let $\tau=\frac12+ib$ with $b>b_2$. 
Then  the following statements hold.
\begin{itemize}
\item[(1)] When $\wp(p)\in (\wp(\tilde{p}_{1,1}), \wp(p_{1,2})]\cup [\wp(p_{0,1}),\wp(\tilde p_{0,2}))$, $G_p(z)$ has no nontrivial critical points satisfying $\wp(a)\in\mathbb{R}$.
\item[(2)] When $\wp(p)\in (-\infty, \wp(p_{1,1})]\cup(\wp(p_{1,2}), \wp(p_{0,1}))\cup[\wp(p_{0,2}),+\infty)\cup\{\wp(\tilde{p}_{1,1})\}\cup\{\wp(\tilde{p}_{0,2})\}$, $G_p(z)$ has a unique pair of nontrivial critical points $\pm a$ satisfying $\wp(a)\in\mathbb{R}$. Moreover, $\pm a$ are degenerate when $\wp(p)\in\{\wp(\tilde{p}_{1,1}), \wp(\tilde{p}_{0,2})\}$.
\item[(3)] When $\wp(p)\in (\wp(p_{1,1}),\wp(\tilde{p}_{1,1}))\cup (\wp(\tilde p_{0,2}), \wp(p_{0,2}))$, $G_p(z)$ has exactly two pairs of nontrivial critical points satisfying $\wp(a)\in\mathbb{R}$.
\end{itemize}
\end{theorem}

\begin{figure}[htbp]\label{figure4-3}
\centering
\begin{tikzpicture}[ scale=0.75,
    dot/.style={circle,fill=black,inner sep=1pt}
]
\draw[thick, black] (-7,0)--(7,0);

\node[dot] at (-5, 0) (p11) {};
\node[dot] at (-3, 0) (tp11) {};
\node[dot] at (-1, 0) (p12) {};
\node[dot] at (1, 0) (p01) {};
\node[dot] at (3, 0) (tp02) {};
\node[dot] at (5, 0) (p02) {};

\node[above] at (-6,0) {\footnotesize$1$};
\node[below] at (-5,0) {\footnotesize$\wp(p_{1,1})$};
\node[above] at (-4,0) {\footnotesize$2$};
\node[below] at (-3,0) {\footnotesize$\wp(\tilde p_{1,1})$};
\node[above] at (-5,0) {\footnotesize$1$};
\node[above] at (-3,0) {\footnotesize$1$};
\node[above] at (-2,0) {\footnotesize$0$};
\node[below] at (-1,0) {\footnotesize$\wp(p_{1,2})$};
\node[above] at (-1,0) {\footnotesize$0$};

\node[above] at (0,0) {\footnotesize$1$};
\node[below] at (1,0) {\footnotesize$\wp(p_{0,1})$};
\node[above] at (1,0) {\footnotesize$0$};

\node[above] at (2,0) {\footnotesize$0$};
\node[below] at (3,0) {\footnotesize$\wp(\tilde p_{0,2})$};
\node[above] at (3,0) {\footnotesize$1$};

\node[above] at (4,0) {\footnotesize$2$};
\node[below] at (5,0) {\footnotesize$\wp(p_{0,2})$};
\node[above] at (5,0) {\footnotesize$1$};
\node[above] at (6,0) {\footnotesize$1$};
\node[below] at (7,0) {\footnotesize$+\infty$};
\node[below] at (-7,0) {\footnotesize$-\infty$};

\end{tikzpicture}
\captionof{figure}{Number of pairs of nontrivial critical points satisfying $\wp(a)\in\mathbb{R}$ when $b>b_2$}
\end{figure}
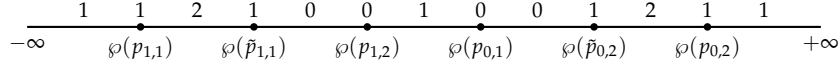

\begin{proof} 
This result follows directly from Lemmas \ref{lemma-aup-1}, \ref{lemma-aup-2}, \ref{lemma-aup-3} and \ref{lem04-9}. Consequently, by defining $$d_2:=\wp(\tilde{p}_{1,1}),\quad d_5:=\wp(\tilde{p}_{0,2}),$$
$$d_1:=\wp(p_{1,1})=e_1+\frac{3e_1^2-\frac{g_2}{4}}{\frac{2\pi}{b}-e_1-\eta_1}, \quad d_3:=\wp(p_{1,2})=e_1-\frac{3e_1^2-\frac{g_2}{4}}{e_1+\eta_1},$$
$$d_4:=\wp(p_{0,1})=-\eta_1,\quad d_6:=\wp(p_{0,2})=\frac{2\pi}{b}-\eta_1,$$
we obtain Theorem \ref{III-thm4}.
\end{proof}

%
%

\section{The case $\tau=\frac12+ib_1$}

\label{sec-5}

In this section, we study the case $\tau=\frac12+ib_1$ and prove Theorem \ref{III-thm2}. 
First, recall from Remark \ref{rmk5-9} that
\begin{align}\label{e6qfc4-1}
e_1+\eta_1=\frac{2\pi}{b_1}.
\end{align}
Recall $\mathcal{B}_k, \alpha_k, \beta_k$ defined in \eqref{B00}-\eqref{alphak1}.
Since Theorem \ref{thm-LW-16} says that $\frac{\omega_1}{2}$ is a degenerate critical point of $G(z)$, it follows from Part I \cite[Section 3]{CFL} that
$$\alpha_1=\frac{\frac{\pi}{b_1}-(\eta_1+e_1)}{3e_1^2-\frac{g_2}{4}}=-\frac{\pi}{b_1(3e_1^2-\frac{g_2}{4})},$$
and
$$\mathcal{B}_1=\Big\{z\in\mathbb{C}\; :\; \operatorname{Re}(\alpha_1 (z-e_1))>\frac12\Big\}.$$
Also recall that
$$
\mathcal{B}_0=\Big\{z\in\mathbb{C}\; :\; \Big|z-\Big(\frac{\pi}{b}-\eta_1\Big)\Big|<\frac{\pi}{b}\Big\},
$$
we see from \eqref{e6qfc4-1} that $\partial\mathcal{B}_0\cap\mathbb{R}=\{-\eta_1, \frac{2\pi}{b_1}-\eta_1\}=\{-\eta_1, e_1\}$.

\begin{lemma}\label{lemma-5s5-9} Let $\tau=\frac12+ib_1$.
Let $d_{1}:=e_1+\frac1{2\alpha_1}$ be the intersection point of $\partial\mathcal{B}_1$ with $\mathbb{R}$. Then $d_{1}<\wp(\frac{1}4+\frac{\tau}{2})<d_{0}:=-\eta_1$. In particular, $\partial\mathcal{B}_0\cap \partial\mathcal{B}_1=\emptyset$ and
\begin{itemize}
\item[(1)] there is $p_{1}\in III^{\circ}$ such that $\wp(p_{1})=d_{1}$.
\item[(2)] there is $p_{0}=\frac12-p_{1}\in II^{\circ}$ such that $\wp(p_{0})=d_{0}=-\eta_1$. 
\end{itemize} 
\end{lemma}

\begin{figure}[htbp]
\centering
\begin{tikzpicture}[ scale=0.7,
    dot/.style={circle,fill=black,inner sep=1pt}
]
\draw[thick, red] (0,-2)--(0,0) -- (2,0) -- (2,2);
\draw[thick, black] (3,2)--(-1,2)--(-3,-2)--(1,-2)--cycle;

\node[dot] at (2, 1.2) (p0) {};
\node[dot] at (0, -1.2) (p1) {};

\node[above] at (1,0) {I};
\node[left] at (2,1) {II};
\node[right] at (0,-1) {III};
\node[left] at (0,0) {\scriptsize$0$};
\node[right] at (2,0) {\scriptsize$\frac12$};
\node[above] at (2,2) {\scriptsize$\frac14+\frac{\tau}2$};
\node[right] at (p0) {\scriptsize$p_{0}$};
\node[left] at (p1) {\scriptsize$p_{1}$};
\node[below] at (0,-2) {\scriptsize$\frac14 - \frac\tau2$};
\end{tikzpicture}
\captionof{figure}{Relative locations of those points in Lemma \ref{lemma-5s5-9}}
\label{figure5-1}
\end{figure}
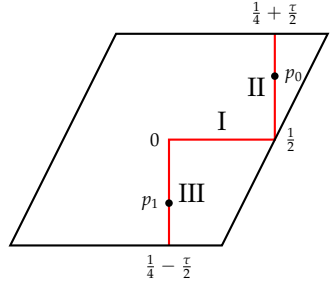

\begin{proof} 
Recall the proof of Lemma \ref{lemma-s5-9} that $3e_1^2-\frac{g_2}{4}=|e_1-e_2|^2>0$, we have
\begin{align}\label{e6qfc-d12}d_{1}=e_1+\frac1{2\alpha_1}=e_1-\frac{3e_1^2-\frac{g_2}{4}}{\frac{2\pi}{b_1}}=e_1-\frac{|e_1-e_2|^2}{e_1+\eta_1}\notin\{e_1, e_2, e_3\}.\end{align}
Recalling \eqref{eqfc-s5-11} that $|e_1-e_2|>e_1+\eta_1>0$, we obtain
$$d_{1}=e_1-\frac{|e_1-e_2|^2}{e_1+\eta_1}<e_1-|e_1-e_2|=\wp\Big(\frac14+\frac{\tau}{2}\Big)<-\eta_1=d_{0}.$$
Then it follows from Lemma \ref{Lemma53-1} that $p_1\in III^\circ$ and $p_0\in II^\circ$.
Furthermore,
\[\wp\Big(p_{1}-\frac12\Big)=e_1+\frac{\wp''(\frac12)}{2(\wp(p_{1})-e_1)}
=e_1+\frac{3e_1^2-\frac{g_2}{4}}{\wp(p_{1})-e_1}=-\eta_1=\wp(p_{0}),\]
so $p_{0}=\frac12-p_{1}$.
\end{proof}

\begin{corollary}\label{c6oro-s5-10}
Let $\tau=\frac12+ib_1$ and $p\in I\cup II\cup III$. Then $\sigma_3$ has no cusps, and
 $\sigma_1$ has at most one cusp, which must be one of $\{A_0, A_1\}$ if exists. Furthermore,
\begin{itemize}
\item[(1)]  $A_0$ is a cusp of $\sigma_1$ if and only if $\wp(p)=\wp(p_{0})=d_{0}=-\eta_1$.
\item[(2)]  $A_1$ is a cusp of $\sigma_1$ if and only if $\wp(p)=\wp(p_{1})=d_{1}$, or equivalently $\wp(p-\frac12)=-\eta_1$.
\end{itemize}
Therefore, if $\wp(p)\notin\{\wp(p_0), \wp(p_1)\}$, then $\sigma_1$ has no cusps.
\end{corollary}

\begin{proof} 
Noting from $p\in I\cup II\cup III$ that $p\notin E_{\tau}[2]$, so $\wp(p)\neq e_1=\frac{2\pi}{b_1}-\eta_1$ and $\wp(p-\frac12)\neq e_1=\frac{2\pi}{b_1}-\eta_1$. The rest proof is the same as that of Corollary \ref{coro-s5-10}.
\end{proof}

\begin{lemma}\label{l6emma-aup-1}
Let $\tau=\frac12+ib_1$. Recall Lemma \ref{lemma-aup} that for any $a\in I\cup II\cup III$, there is a unique $p\in I\cup II\cup III\cup \{0,\frac12\}$ such that  $\wp(p)=f(a)$ and $\pm a$ is a pair of nontrivial critical points of $G_p(z)$, where $f(a)$ is defined in \eqref{513-1-0}.
Then $f: I\to f(I)$ is a differmorphism and
\begin{equation}\label{6dddd}f(I)=(\wp(p_{1}), \wp(p_{0})).\end{equation}
\end{lemma}

\begin{proof} Since Theorem \ref{thm-LW-16} says that $G_{\frac12}(z)=G(z-\frac12)$ has no nontrivial critical points and $G_0(z)=G(z)$ has no nontrivial critical points, we have $e_1, \infty\notin f(I)$.
The rest proof is the same as that of Lemma \ref{lemma-aup-1}.
\end{proof}

\begin{lemma}\label{l6emma-aup-2}
Let $\tau=\frac12+ib_1$. Recall Lemma \ref{lemma-aup} that for any $a\in I\cup II\cup III$, there is a unique $p\in I\cup II\cup III\cup \{0,\frac12\}$ such that  $\wp(p)=f(a)$ and $\pm a$ is a pair of nontrivial critical points of $G_p(z)$, where $f(a)$ is defined in \eqref{513-1-0}. Then $f: II=(\frac{1}{2},\frac{1}{4}+\frac{\tau}{2}]\to f(II)$ is a homeomorphism and
\begin{equation}\label{6dddd2-0}
f(II)=\Big[\wp\Big(\frac14\Big),+\infty\Big).
\end{equation}
\end{lemma}

\begin{proof} The proof is similar to Steps 1-2 of Lemma \ref{lemma-aup-2}.
Let $a=r+s\tau\in II=(\frac{1}{2},\frac{1}{4}+\frac{\tau}{2}]$. 
Then $s=1-2r$, i.e., $a-\frac12=(1-2\tau)(r-\frac12)=-2b_1 (r-\frac12)i$ with $r\in [\frac14, \frac12)$. 
When $\frac12-r>0$ is sufficiently small, we see from $\frac{4\pi i}{2\tau-1}=\frac{2\pi}{b_1}=e_1+\eta_1$ that
$$\wp(a)=e_1+O\left(\Big(\frac12-r\Big)^2\right), \quad \wp'(a)=\wp''\Big(\frac12\Big)\Big(a-\frac12\Big)+O\left(\Big(\frac12-r\Big)^3\right),$$
and
\begin{align*}&\zeta(a)-r\eta_1-s\eta_2\\
=&-\Big[e_1+\eta_1-\frac{4\pi i}{2\tau-1}\Big]\Big(a-\frac12\Big)
-\frac{\wp''(\frac12)}{6}\Big(a-\frac12\Big)^3+O\left(\Big(\frac12-r\Big)^5\right)\\
=&-\frac{\wp''(\frac12)}{6}\Big(a-\frac12\Big)^3+O\left(\Big(\frac12-r\Big)^5\right).\end{align*}
Note that $\wp''(\frac12)=6e_1^2-g_2/2=4e_1^2+2|e_2|^2>0$. Consequently, we see from $(a-\frac12)^2=-4b_1^2 (r-\frac12)^2$ that
\begin{align}\label{e6qq-61}\wp(p)&=f(a)=\wp(a)+\frac{\wp'(a)}{2(\zeta(a)-r\eta_1-s\eta_2)}\nonumber\\
&=-\frac{3}{(a-\frac12)^2}+O(1)\nearrow +\infty\quad \text{as}\quad r\nearrow\frac12.\end{align}

Since $G_0(z)=G(z)$ has no nontrivial critical points, so $\infty\notin f(II)$.
Recall \eqref{eqq-f5} that
$$f\Big(\frac14+\frac{\tau}{2}\Big)=\wp\Big(\frac14\Big),\qquad
f(a)\neq\wp\Big(\frac14\Big)\quad\text{for any }a\in II^\circ.$$
From here and $f: II^\circ\to f(II^{\circ})$ is continuous, i.e., $f(II^{\circ})\subset\mathbb R$ is path-connected, we obtain
$$f(II^\circ)=\Big(\wp\Big(\frac14\Big),+\infty\Big).
$$
In particular, $p\in I=(0,\frac12)$ for $a\in II$.

Now we prove that $f: II^\circ=(\frac12, \frac14+\frac{\tau}2)\to (\wp(\frac14),+\infty)$ is a homeomorphism. Indeed,
for any $a\in  (\frac12, \frac14+\frac{\tau}2)$, we have $\wp(a)<e_1<\wp(p)$, so the corresponding 
\begin{align}\label{e6qq-Aj}A(a)&=\frac{1}{2}\left[  \zeta(p+a)+\zeta(p-a)-\zeta(2p)\right]\\
&=A_0+\frac{\wp'(p)}{2(\wp(p)-\wp(a))}\in(A_1, A_0).\nonumber\end{align}
Since Lemma \ref{Lemma53-5} shows $\#(\sigma_1\cap\sigma_3\cap (A_1, A_0))\leq 1$ for any $p\in I$, we conclude that
for any $\wp(p)\in (\wp(\frac14),+\infty)$, $f^{-1}(\wp(p))$ contains exactly one preimage in $(\frac12, \frac14+\frac{\tau}2)$. Thus, $f: (\frac12, \frac14+\frac{\tau}2)\to (\wp(\frac14),+\infty)$ is bijective. The inverse map $f^{-1} : (\wp(\frac14),+\infty)\to (\frac12, \frac14+\frac{\tau}2)$ is also continuous because $a=f^{-1}(\wp(p))$ is the nontrivial critical point of $G_p(z)$. This completes the proof.
\end{proof}

\begin{lemma}\label{l6emma-aup-3}
Let $\tau=\frac12+ib_1$. Recall Lemma \ref{lemma-aup} that for any $a\in I\cup II\cup III$, there is a unique $p\in I\cup II\cup III\cup \{0,\frac12\}$ such that  $\wp(p)=f(a)$ and $\pm a$ is a pair of nontrivial critical points of $G_p(z)$, where $f(a)$ is defined in \eqref{513-1-0}. 
Then $f: III=[\frac14-\frac{\tau}{2}, 0)\to f(III)$ is a homeomorphism and
\begin{equation}\label{6dddd3-0}
f(III)=\Big(e_1,\wp\Big(\frac14\Big)\Big].
\end{equation}
\end{lemma}

\begin{proof}
Note that $a\in II$ if and only if $\frac12-a\in III$, and $p\in (0,\frac12)$ if and only if $\frac12-p\in (0, \frac12)$, this lemma follows directly from \eqref{eq: da11ta3} and Lemma \ref{l6emma-aup-2}.
\end{proof}

\begin{theorem}[=Theorem \ref{III-thm2}]
Let $\tau=\frac12+ib_1$. Then the following statements hold.
\begin{itemize}
\item[(1)] When $\wp(p)\in (-\infty, \wp(p_{1})]\cup [\wp(p_{0}), e_1]$, $G_p(z)$ has no nontrivial critical points satisfying $\wp(a)\in\mathbb{R}$.
\item[(2)] When $\wp(p)\in (\wp(p_{1}), \wp(p_{0}))\cup(e_1,+\infty)$, $G_p(z)$ has a unique pair of nontrivial critical points $\pm a$ satisfying $\wp(a)\in\mathbb{R}$.
\end{itemize}
\end{theorem}

\begin{proof} This result follows directly from Lemmas \ref{l6emma-aup-1}, \ref{l6emma-aup-2} and \ref{l6emma-aup-3}. Consequently,
by defining
$$d_1:=\wp(p_{1})=e_1-\frac{3e_1^2-\frac{g_2}{4}}{\frac{2\pi}{b_1}}, \quad d_2:=\wp(p_{0})=-\eta_1,$$
 we obtain Theorem \ref{III-thm2}.
\end{proof}

\section{The case $\tau=\frac12+ib$ with $b\in [\frac12, b_1)$}

\label{sec-6}

In this section, we study the case $\tau=\frac12+ib$ with $b\in [\frac12, b_1)$ and prove Theorem \ref{III-thm1}. 
By Remark \ref{rmk5-9} and $b_0<\frac12$, we have
\begin{align}\label{7eqfc4-1}
0<e_1+\eta_1<\frac{2\pi}{b},\qquad \forall\, b\in  \Big[\frac12, b_1\Big).
\end{align}
Recall $\mathcal{B}_k, \alpha_k, \beta_k$ defined in \eqref{B00}-\eqref{alphak1}.
Since Theorem \ref{thm-LW-16} says that $\frac{\omega_1}{2}$ is a non-degenerate minimal point of $G(z)$, it follows from Part I \cite[Section 3]{CFL} and $3e_1^2-\frac{g_2}{4}=|e_1-e_2|^2>0$
 that
\begin{align}\label{alphak0-5}
|\alpha_1|=\frac{|\frac{\pi}{b}-(\eta_1+e_1)|}{3e_1^2-\frac{g_2}{4}}<\beta_1=\frac{\frac{\pi}{b}}{3e_1^2-\frac{g_2}{4}},
\end{align}
\begin{equation}\label{7alphak1-5}
\mathcal{B}_1=\bigg\{z\in\mathbb{C}\; :\; \bigg|z-e_1-\frac{\alpha_1}{\alpha_1^2-\beta_1^2}\bigg|<\frac{\beta_1}{\beta_1^2-\alpha_1^2}\bigg\}.
\end{equation}
By \eqref{eqfc-s5-2}, we see that the center of $\mathcal{B}_1$ is on the real axis.
Also recall that
\[
\mathcal{B}_0=\Big\{z\in\mathbb{C}\; :\; \Big|z-\Big(\frac{\pi}{b}-\eta_1\Big)\Big|<\frac{\pi}{b}\Big\}.
\]

\begin{lemma}\label{7lemma-s5-9} Let $\tau=\frac12+ib$ with $\frac12\leq b<b_1$.
Let $d_{1,1}<d_{1,2}$ be the two intersection points of $\partial\mathcal{B}_1$ with $\mathbb{R}$, and $d_{0,1}:=-\eta_1<d_{0,2}:=\frac{2\pi}{b}-\eta_1$ be the two intersection points of $\partial\mathcal{B}_0$ with $\mathbb{R}$. Then $$d_{1,1}<\wp\Big(\frac{1}4+\frac{\tau}{2}\Big)<d_{0,1}<e_1<d_{0,2}<\wp\Big(\frac14\Big)<d_{1,2}.$$  In particular, $\overline{\mathcal{B}_0}\Subset\mathcal{B}_1$ and
\begin{itemize}
\item[(1)] there are $p_{1,1}\in III^{\circ}$ and $p_{1,2}\in (0, \frac14)$ such that $\wp(p_{1,1})=d_{1,1}$, $\wp(p_{1,2})=d_{1,2}$.
\item[(2)] there are $p_{0,1}=\frac12-p_{1,1}\in II^{\circ}$ and $p_{0,2}=\frac12-p_{1,2}\in (\frac14,\frac12)$ such that $\wp(p_{0,1})=d_{0,1}$, $\wp(p_{0,2})=d_{0,2}$. 
\end{itemize} 
\end{lemma}

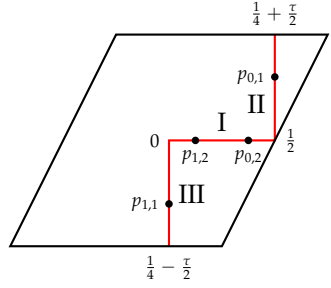
\begin{figure}[htbp]
\centering
\begin{tikzpicture}[ scale=0.7,
    dot/.style={circle,fill=black,inner sep=1pt}
]
\draw[thick, red] (0,-2)--(0,0) -- (2,0) -- (2,2);
\draw[thick, black] (3,2)--(-1,2)--(-3,-2)--(1,-2)--cycle;

\node[dot] at (2, 1.2) (p01) {};
\node[dot] at (1.5, 0) (p02) {};
\node[dot] at (0, -1.2) (p11) {};
\node[dot] at (0.5,0) (p12) {};

\node[above] at (1,0) {I};
\node[left] at (2,0.7) {II};
\node[right] at (0,-1) {III};
\node[left] at (0,0) {\scriptsize$0$};
\node[right] at (2,0) {\scriptsize$\frac12$};
\node[above] at (2,2) {\scriptsize$\frac14+\frac{\tau}2$};
\node[left] at (p01) {\scriptsize$p_{0,1}$};
\node[below] at (p02) {\scriptsize$p_{0,2}$};
\node[left] at (p11) {\scriptsize$p_{1,1}$};
\node[below] at (p12) {\scriptsize$p_{1,2}$};
\node[below] at (0,-2) {\scriptsize$\frac14 - \frac\tau2$};
\end{tikzpicture}
\captionof{figure}{Relative locations of those points in Lemma \ref{7lemma-s5-9}}
\label{figure6-1}
\end{figure}

\begin{proof} 
By $3e_1^2-\frac{g_2}{4}=|e_1-e_2|^2>0$,
we have $$d_{1,2}=e_1+\frac{\beta_1-\alpha_1}{\beta_1^2-\alpha_1^2}=e_1+\frac{1}{\beta_1+\alpha_1}=e_1+\frac{3e_1^2-\frac{g_2}{4}}{\frac{2\pi}{b}-(e_1+\eta_1)}\notin\{e_1, e_2, e_3\},$$
and
\begin{align}\label{7eqfc-d12}d_{1,1}&=e_1-\frac{\alpha_1+\beta_1}{\beta_1^2-\alpha_1^2}=e_1-\frac{1}{\beta_1-\alpha_1}=e_1-\frac{3e_1^2-\frac{g_2}{4}}{e_1+\eta_1}\\
&=e_1-\frac{|e_1-e_2|^2}{e_1+\eta_1}\notin\{e_1, e_2, e_3\}.\nonumber\end{align}
Recalling \eqref{eqfc-s5-10}-\eqref{eqfc-s5-11} that $|e_1-e_2|>e_1+\eta_1>0$, we obtain
$$d_{1,1}=e_1-\frac{|e_1-e_2|^2}{e_1+\eta_1}<e_1-|e_1-e_2|=\wp\Big(\frac14+\frac{\tau}{2}\Big)<-\eta_1=d_{0,1}<e_1.$$
Similarly, noting from \eqref{eqfc-s5-12} that $|e_1-e_2|>\frac{2\pi}{b}-(e_1+\eta_1)>0$, we have
$$d_{1,2}=e_1+\frac{|e_1-e_2|^2}{\frac{2\pi}{b}-(e_1+\eta_1)}>e_1+|e_1-e_2|=\wp\Big(\frac14\Big)>\frac{2\pi}{b}-\eta_1=d_{0,2}>e_1.$$
Furthermore,
\begin{align}\label{7eqfc-s5-31}\wp\Big(p_{1,1}-\frac12\Big)&=e_1+\frac{\wp''(\frac12)}{2(\wp(p_{1,1})-e_1)}\\
&=e_1+\frac{3e_1^2-\frac{g_2}{4}}{\wp(p_{1,1})-e_1}=-\eta_1=\wp(p_{0,1}),\nonumber\end{align}
and 
\begin{align}\label{7eqfc-s5-22}\wp\Big(p_{1,2}-\frac12\Big)&=e_1+\frac{\wp''(\frac12)}{2(\wp(p_{1,2})-e_1)}\\
&=e_1+\frac{3e_1^2-\frac{g_2}{4}}{\wp(p_{1,2})-e_1}=\frac{2\pi}{b}-\eta_1=\wp(p_{0,2}),\nonumber\end{align}
so $p_{0,1}=\frac12-p_{1,1}$ and $p_{0,2}=\frac12-p_{1,2}$. The proof is complete.
\end{proof}

\begin{corollary}\label{7coro-s5-10}
Let $\tau=\frac12+ib$ with $\frac12\leq b<b_1$ and $p\in I\cup II\cup III$. Then for $j=1,3$, $\sigma_j$ has at most one cusp, which must be one of $\{A_0, A_1\}$ if exists. Furthermore,
\begin{itemize}
\item[(1)]  $A_0$ is a cusp of $\sigma_3$ if and only if $\wp(p)=\wp(p_{0,2})=d_{0,2}=\frac{2\pi}{b}-\eta_1$.
\item[(2)]  $A_0$ is a cusp of $\sigma_1$ if and only if $\wp(p)=\wp(p_{0,1})=d_{0,1}=-\eta_1$.
\item[(3)]  $A_1$ is a cusp of $\sigma_1$ if and only if $\wp(p)=\wp(p_{1,1})=d_{1,1}$, or equivalently $\wp(p-\frac12)=-\eta_1$.
\item[(4)]  $A_1$ is a cusp of $\sigma_3$ if and only if $\wp(p)=\wp(p_{1,2})=d_{1,2}$, or equivalently $\wp(p-\frac12)=\frac{2\pi}{b}-\eta_1$.
\end{itemize}
Therefore, if $\wp(p)\notin\{\wp(p_{j,k}), j=0,1, k=1,2\}$, then $\sigma_j$ has no cusps for $j=1,3$.
\end{corollary}

\begin{proof}
The proof is similar to that of Corollary \ref{coro-s5-10}.
\end{proof}

\begin{lemma}\label{7lemma-aup-1}
Let $\tau=\frac12+ib$ with $\frac12\leq b<b_1$. Recall Lemma \ref{lemma-aup} that for any $a\in I\cup II\cup III$, there is a unique $p\in I\cup II\cup III\cup \{0,\frac12\}$ such that  $\wp(p)=f(a)$ and $\pm a$ is a pair of nontrivial critical points of $G_p(z)$, where $f(a)$ is defined in \eqref{513-1-0}.
Then $f: I\to f(I)$ is a differmorphism and
\begin{equation}\label{7dddd}f(I)=(\wp(p_{1,1}), \wp(p_{0,1})).\end{equation}
\end{lemma}

\begin{proof} Since Theorem \ref{thm-LW-16} says that $G_{\frac12}(z)=G(z-\frac12)$ has no nontrivial critical points and $G_0(z)=G(z)$ has no nontrivial critical points, we have $e_1, \infty\notin f(I)$.
The rest proof is the same as that of Lemma \ref{lemma-aup-1}.
\end{proof}

\begin{lemma}\label{7lemma-aup-2}
Let $\tau=\frac12+ib$ with $\frac12\leq b<b_1$. Recall Lemma \ref{lemma-aup} that for any $a\in I\cup II\cup III$, there is a unique $p\in I\cup II\cup III\cup \{0,\frac12\}$ such that  $\wp(p)=f(a)$ and $\pm a$ is a pair of nontrivial critical points of $G_p(z)$, where $f(a)$ is defined in \eqref{513-1-0}. 
Then $f: II=(\frac{1}{2},\frac{1}{4}+\frac{\tau}{2}]\to f(II)$ is a homeomorphism and
\begin{equation}\label{7dddd2-0}
f(II)=\Big[\wp\Big(\frac14\Big),\wp(p_{1,2})\Big).
\end{equation}
\end{lemma}

\begin{proof} The proof is similar to Steps 1-2 of Lemma \ref{lemma-aup-2}.
Let $a=r+s\tau\in II=(\frac{1}{2},\frac{1}{4}+\frac{\tau}{2}]$. 
Then $s=1-2r$, i.e., $a-\frac12=(1-2\tau)(r-\frac12)$ with $r\in [\frac14, \frac12)$. 
When $\frac12-r>0$ is sufficiently small, the same argument as Step 1 of Lemma \ref{lemma-aup-2} implies
\begin{align}\label{7eqq-61}\wp(p)&=f(a)=e_1-\frac{3e_1^2-\frac{g_2}{4}}{e_1+\eta_1-\frac{2\pi}{b}}+O\left(\Big(\frac12-r\Big)^2\right)\nonumber\\&=\wp(p_{1,2})+O\left(\Big(\frac12-r\Big)^2\right).\end{align}

Since $G_0(z)=G(z)$ has no nontrivial critical points, so $\infty\notin f(II)$.
Recall \eqref{eqq-f5} that $$f\Big(\frac14+\frac{\tau}{2}\Big)=\wp\Big(\frac14\Big),\qquad
f(a)\neq\wp\Big(\frac14\Big)\quad\text{for any }a\in II^\circ.$$
From here and $f: II^\circ\to f(II^{\circ})$ is continuous, i.e., $f(II^{\circ})\subset\mathbb R$ is path-connected, we obtain
$$\Big(\wp\Big(\frac14\Big),\wp(p_{1,2})\Big)
\subset f(II^\circ)\subset \Big(\wp\Big(\frac14\Big),+\infty\Big).$$
In particular, $p\in I=(0,\frac12)$ for $a\in II$.
Then by repeating the proof of Lemma \ref{l6emma-aup-2}, we can prove that the continuous map $f: II^\circ=(\frac12, \frac14+\frac{\tau}2)\to f(II^\circ)$ is bijective. 
Together with \eqref{7eqq-61}, we easily conclude that $f(II^\circ)=(\wp(\frac14), \wp(p_{1,2}))$,
namely $f: II^\circ=(\frac12, \frac14+\frac{\tau}2)\to (\wp(\frac14), \wp(p_{1,2}))$ is bijective and then
a homeomorphism. 
\end{proof}

\begin{lemma}\label{7lemma-aup-3}
Let $\tau=\frac12+ib$ with $\frac12\leq b<b_1$. Recall Lemma \ref{lemma-aup} that for any $a\in I\cup II\cup III$, there is a unique $p\in I\cup II\cup III\cup \{0,\frac12\}$ such that  $\wp(p)=f(a)$ and $\pm a$ is a pair of nontrivial critical points of $G_p(z)$, where $f(a)$ is defined in \eqref{513-1-0}. 
Then $f: III=[\frac14-\frac{\tau}{2}, 0)\to f(III)$ is a homeomorphism and
\begin{equation}\label{7dddd3-0}
f(III)=\Big(\wp(p_{0,2}),\wp\Big(\frac14\Big)\Big].
\end{equation}
\end{lemma}

\begin{proof}
Again this lemma follows directly from \eqref{eq: da11ta3}, $p_{0,2}=\frac12-p_{1,2}$ and Lemma \ref{7lemma-aup-2}.
\end{proof}

\begin{theorem}[=Theorem \ref{III-thm1}]\label{sec6tion-thm}
Let $\tau=\frac12+ib$ with $\frac12\leq b<b_1$. Then the following statements hold.
\begin{itemize}
\item[(1)] When $\wp(p)\in (-\infty, \wp(p_{1,1})]\cup [\wp(p_{0,1}), \wp(p_{0,2})]\cup [\wp(p_{1,2}),+\infty)$, $G_p(z)$ has no nontrivial critical points satisfying $\wp(a)\in\mathbb{R}$.
\item[(2)] When $\wp(p)\in (\wp(p_{1,1}), \wp(p_{0,1}))\cup(\wp(p_{0,2}),\wp(p_{1,2}))$, $G_p(z)$ has a unique pair of nontrivial critical points $\pm a$ satisfying $\wp(a)\in\mathbb{R}$.
\end{itemize}
\end{theorem}

\begin{proof} This result follows directly from Lemmas \ref{7lemma-aup-1}, \ref{7lemma-aup-2} and \ref{7lemma-aup-3}. Consequently, by defining
$$d_1:=\wp(p_{1,1})=e_1-\frac{3e_1^2-\frac{g_2}{4}}{e_1+\eta_1}, \quad d_2:=\wp(p_{0,1})=-\eta_1,$$
$$d_3:=\wp(p_{0,2})=\frac{2\pi}{b}-\eta_1,\quad d_4:=\wp(p_{1,2})=e_1+\frac{3e_1^2-\frac{g_2}{4}}{\frac{2\pi}{b}-e_1-\eta_1},$$
 we obtain Theorem \ref{III-thm1}.
\end{proof}

\section{Further results for $p\in (0, \frac12)$}
\label{sec-7}

In this section, we always consider $\tau=\frac12+ib$ with $b\geq \frac12$ and $p\in(0,\frac12)$.
Recall from Lemma \ref{Lemma53-6} that $\sigma_j$ has been completely determined for $p=\frac14$.
The main purpose of this section is to improve Lemma \ref{Lemma53-5} to a sharp version for general $p\in (0,\frac12)$. As applications, we can study the non-degeneracy of the nontrivial critical points and prove Corollary \ref{III-coro}.

\begin{lemma}\label{lem0-s5-13} Let $\tau=\frac12+ib$ with $b\geq \frac12$ and $p\in(0,\frac12)$.
Define
\begin{align}\label{deltap}
\Delta(p):=&16\wp(p)^4-128\eta_1 \wp(p)^3+8g_2\wp(p)^2\\
&+32(\eta_1g_2+g_3)\wp(p)+g_2^2+32\eta_1g_3.\nonumber
\end{align}
Then the following statements hold.
\begin{itemize}
		\item[(1)] If $\Delta(p)< 0$, then $\sigma_1=[A_1, A_0] \cup \sigma_{1, \infty A_2} \cup \overline{\sigma_{1, \infty A_2}}$, where $\sigma_{1, \infty A_2} \subset\{z: \operatorname{Im} z>0\}$ is a semi-unbounded simple curve with the endpoint $A_2$ and tending to $-\frac{2 p \eta_1-\zeta(2 p)}{2}+ i \infty$, while $\overline{\sigma_{1, \infty A_2}} \subset\{z: \operatorname{Im} z<0\}$ is the complex conjugate of $\sigma_{1, \infty A_2}$ such that $\sigma_{1, \infty A_2} \cup \overline{\sigma_{1, \infty A_2}}$ is symmetric with respect to the real axis.
		\item[(2)] If $\Delta(p)\geq 0$, then there exist $A_5, A_6\in (A_1, A_0)$ such that $\sigma_1=[A_1, A_0] \cup \sigma_{1, A_2 A_3} \cup \sigma_{1, \infty}$, where $\sigma_{1, \infty}$ is an unbounded simple curve tending to $-\frac{2 p \eta_1-\zeta(2 p)}{2} \pm i \infty$, which is symmetric with respect to the real axis with $\sigma_{1, \infty} \cap \mathbb{R}=\{A_5\}$, and $\sigma_{1, A_2 A_3}$ is a simple curve connecting $A_2$ and $A_3$, and is symmetric with respect to the real axis with $\sigma_{1, A_2 A_3} \cap \mathbb{R}=\{A_6\}$. In particular, both $A_5$ and $A_6$ are branch points of $\sigma_1$. Moreover, $A_5=A_6$ if and only if $\Delta(p)=0$.
	\end{itemize}
\end{lemma}
\begin{proof}
Recall the additional formula of elliptic functions:
\begin{equation}
\wp(u+v)=-\wp(u)-\wp(v)+\frac{(\wp'(u)-\wp'(v))^2}{4(\wp(u)-\wp(v))^2},
\end{equation}
from which we obtain
\begin{equation}
\wp(u-v)=-\wp(u)-\wp(v)+\frac{(\wp'(u)+\wp'(v))^2}{4(\wp(u)-\wp(v))^2},
\end{equation}
and so
\begin{align}\label{eqfc-add10}
\wp(u+v)&+\wp(u-v)=-2\wp(u)-2\wp(v)+\frac{\wp'(u)^2+\wp'(v)^2}{2(\wp(u)-\wp(v))^2}\nonumber\\
&=-2\wp(u)-2\wp(v)+\frac{4\wp(u)^3+4\wp(v)^3-g_2\wp(u)-g_2\wp(v)-2g_3}{2(\wp(u)-\wp(v))^2}.
\end{align}
Then by \eqref{eqfc-add10}, we see that
\begin{equation}\label{eqfc-s5-40}
\wp(a+p)+\wp(a-p)+2\eta_1=0
\end{equation}
is equivalent to
{\allowdisplaybreaks
\begin{align}\label{eqfc-s5-41}
0=&4(\eta_1+\wp(p))\wp(a)^2+(4\wp(p)^2-8\eta_1\wp(p)-g_2)\wp(a)\\
&+4\eta_1\wp(p)^2-g_2\wp(p)-2g_3\nonumber\\
=&4\eta_1(\wp(a)-\wp(p))^2+(4\wp(p)\wp(a)-g_2)(\wp(p)+\wp(a))-2g_3\nonumber\\
=&4\eta_1(\wp(a)-\wp(p))^2+4(\wp(p)\wp(a)-e_1^2+|e_2|^2)(\wp(p)+\wp(a))-8e_1|e_2|^2\nonumber\\
=&:T(\wp(a)),\nonumber
\end{align}
}%
where we have used \eqref{eqfc-g2g3}. 

Recalling \eqref{derivative-eta} that $\eta_1(\frac12+ib)$ is strictly decreasing for $b\geq \frac12$, so
$$\eta_1\Big(\frac12+ib\Big)> \lim_{b\to+\infty}\eta_1\Big(\frac12+ib\Big)=\frac{\pi^2}{3},\quad\forall b\geq \frac12.$$
Meanwhile, it is known (cf. \cite[Theorem 1.7]{LW}) that $e_1(\frac12+ib)$ is strictly increasing for $b>0$, so
$$ e_1\Big(\frac12+ib\Big)> e_1\Big(\frac12+i\frac12\Big)=0,\quad\forall b>\frac12.$$ 
Since $p\in(0,\frac12)$ and $b\geq \frac12$, we have $\wp(p)>e_1\geq 0$ and so
\begin{equation}\label{axsd}T(\wp(a))>8e_1(e_1^2-e_1^2+|e_2|^2)-8e_1|e_2|^2=0,\quad\forall \wp(a)\geq e_1.\end{equation}

Also note that for $p\in (0,\frac12)$ and $b\geq \frac12$, $$\eta_1+\wp(p)>\eta_1+e_1\geq \frac{\pi^2}{3}>0,$$ so \eqref{eqfc-s5-41} implies that $T(\wp(a))=0$
is a quadratic equation of $\wp(a)$, and a direct computation shows that the discriminant $\Delta(p)$ of $T(\wp(a))=0$ is given by \eqref{deltap}.
Clearly when $(0,\frac12)\ni p$ is sufficiently small,  $\wp(p)>e_1$ is sufficiently large, so $\Delta(p)>0$.
Define
\begin{equation}\label{eq-pc}p_c:=\sup\Big\{\tilde p\in \Big(0,\frac12\Big) \;:\; \Delta(p)>0\quad\text{for all }p\in (0,\tilde p)\Big\}.\end{equation}
Then $0<p_c\leq\frac12$, and for any $p\in(0,p_c)$, it follows from \eqref{axsd} that $T(\wp(a))=0$ has two real-valued roots $\wp(a_5)\neq\wp(a_6)<e_1<\wp(p)$.
Consequently,
\begin{align*}A_k:=&\frac{1}{2}\left[  \zeta(p+a_k)+\zeta(p-a_k)-\zeta(2p)\right]\\
=&A_0+\frac{\wp'(p)}{2(\wp(p)-\wp(a_k))}\in (A_1, A_0),\quad k=5,6.\end{align*}
From here and Lemma \ref{Lemma53-5}, we have $A_5\neq A_6\in (A_1, A_0)\subset\sigma_1$.
On the other hand, it follows from \eqref{eqfc-s5-40}-\eqref{eqfc-s5-41} that
\begin{equation}\label{eqfc-s5-43}
\wp(a_k+p)+\wp(a_k-p)+2\eta_1=0,\quad k=5,6.
\end{equation}
Then Lemma \ref{lemma2-10-3} implies that both $A_5$ and $A_6$ are branch points of $\sigma_1$. Together with Lemma  \ref{Lemma53-5} and by renaming $A_5, A_6$ if necessary, we conclude that Lemma \ref{Lemma53-5} (2-2) holds, i.e., $\sigma_1=[A_1, A_0]\cup\sigma_{1,A_2A_3}\cup\sigma_{1,\infty}$ with 
\begin{equation}\label{eqfc-s5-42}\sigma_{1,\infty}\cap\mathbb{R}=\{A_{5}\},\quad \sigma_{1,A_2A_3}\cap \mathbb R=\{A_6\}.\end{equation}
This proves the assertion (2) for $p\in (0, p_c)$.

Since Lemma \ref{Lemma53-6} says that $\sigma_1$ has no branch points for $p=\frac14$, so 
\begin{equation}\label{eq-pc1}p_c\in \Big(0, \frac14\Big)\quad\text{ and }\quad\Delta(p_c)=0.\end{equation} 
Consequently, for $p=p_c$, $T(\wp(a))=0$ has a double real-valued root $\wp(a_5)=\wp(a_6)<e_1$, and the corresponding $A_5=A_6\in (A_1, A_0)$. By the continuous deformation of $\sigma_1$ with respect to $p\uparrow p_c$, we have that for $p=p_c$, $\sigma_1=[A_1, A_0]\cup\sigma_{1,A_2A_3}\cup\sigma_{1,\infty}$ with 
\begin{equation}\label{eqfc-0s5-42}\sigma_{1,A_2A_3}\cap \mathbb R=\{A_5\}=\sigma_{1,\infty}\cap\mathbb{R}.\end{equation}
This proves the assertion (2) for $p=p_c$.

Let $p_c<p_1<\cdots<p_m<\frac12$ such that $\Delta(p_j)=\Delta(p_c)=0$. Note $m\leq 3$ by \eqref{deltap}.

{\bf Case 1.} $\Delta(p)>0$ for any $p\in (p_c, p_1)$. 

Then the same argument as above implies that for any $p\in (p_c, p_1]$, $\sigma_1=[A_1, A_0]\cup\sigma_{1,A_2A_3}\cup\sigma_{1,\infty}$ with \eqref{eqfc-s5-42} for $p\in (p_c, p_1)$ and \eqref{eqfc-0s5-42} for $p=p_1$. In particular, the assertion (2) holds for $p\in (p_c, p_1]$.

{\bf Case 2.} $\Delta(p)<0$ for any $p\in (p_c, p_1)$. 

Then for $p\in (p_c, p_1)$,  $T(\wp(a))=0$ has no real-valued roots, so any $A\in\mathbb R$ can not be 
branch points of $\sigma_1$. Recall Corollary \ref{coro-s5-10} that $\sigma_1$ has no cusps. 
Assume by contradiction that suppose Lemma \ref{Lemma53-5} (2-2) holds for some  $p\in (p_c, p_1)$, i.e., $\sigma_1=[A_1, A_0]\cup\sigma_{1,A_2A_3}\cup\sigma_{1,\infty}$ with 
$$\sigma_{1,\infty}\cap\mathbb{R}=\{\tilde A_{5}\},\quad \sigma_{1,A_2A_3}\cap \mathbb R=\{\tilde A_6\}.$$
Then $\tilde A_j\in (-\infty, A_1)\cup (A_0,+\infty)$ (otherwise, $\tilde A_j\in [A_1, A_0]$ and then it is a cusp or a branch point of $\sigma_1$, a contradiction). Thus, Lemma \ref{Lemma53-5} implies that $\tilde A_j\in \sigma_1\cap\sigma_3\cap\mathbb{R}\setminus\{A_k\}_{k=0}^3$, then $G_p(z)$ has a pair of nontrivial critical points $\pm a_j$ satisfying $\wp(a_j)\in \mathbb R$ and 
$$\tilde A_j=\frac{1}{2}\left[  \zeta(p+a_j)+\zeta(p-a_j)-\zeta(2p)\right]\in (-\infty, A_1)\cup (A_0,+\infty).$$
But this is a contradiction with Lemma \ref{lem-s5-20}-(1) below.
Therefore, for any $p\in (p_c, p_1)$, Lemma \ref{Lemma53-5} (2-1) holds. This proves the assertion (1) for  $p\in (p_c, p_1)$.

For $p=p_1$, $T(\wp(a))=0$ has a double real-valued root $\wp(a_5)=\wp(a_6)<e_1$, and the corresponding $A_5=A_6\in (A_1, A_0)$. Then \eqref{eqfc-s5-43} holds, so Lemma \ref{lemma2-10-3} implies that $A_5$ is a branch point of $\sigma_1$.
By the continuous deformation of $\sigma_1$ with respect to $p\uparrow p_1$, we have that for $p=p_1$, $\sigma_1=[A_1, A_0]\cup\sigma_{1,A_2A_3}\cup\sigma_{1,\infty}$ with \eqref{eqfc-0s5-42} holds.
Thus, the assertion (2) holds for $p=p_1$.

We can repeat the same argument to treat $p\in (p_1, p_2]$. In conclusion, by repeating this argument at most three times, we finally obtain the assertions (1)-(2) for $p\in (0, \frac12)$. The proof is complete.
\end{proof}

Now we can improve Lemma \ref{Lemma53-5}-(2) to a sharp version by using  Lemma \ref{lem0-s5-13} and \eqref{eq7-16}.

\begin{lemma}\label{lem2-s5-13}
Let $\tau=\frac12+ib$ with $b\geq \frac12$, and recall $p_c\in (0,\frac14)$ defined in \eqref{eq-pc} and \eqref{eq-pc1}. 
\begin{itemize}
		\item[(1)] If $p\in(p_c,\frac{1}{2}-p_c)$, then $\sigma_1=[A_1, A_0] \cup \sigma_{1, \infty A_2} \cup \overline{\sigma_{1, \infty A_2}}$, where $\sigma_{1, \infty A_2} \subset\{z: \operatorname{Im} z>0\}$ is a semi-unbounded simple curve with the endpoint $A_2$ and tending to $-\frac{2 p \eta_1-\zeta(2 p)}{2}+ i \infty$, while $\overline{\sigma_{1, \infty A_2}} \subset\{z: \operatorname{Im} z<0\}$ is the complex conjugate of $\sigma_{1, \infty A_2}$ such that $\sigma_{1, \infty A_2} \cup \overline{\sigma_{1, \infty A_2}}$ is symmetric with respect to the real axis.
		\item[(2)] If $p\in(0,p_c]\cup[\frac12-p_c,\frac12)$, then there exist $A_5, A_6\in (A_1, A_0)$ such that $\sigma_1=[A_1, A_0] \cup \sigma_{1, A_2 A_3} \cup \sigma_{1, \infty}$, where $\sigma_{1, \infty}$ is an unbounded simple curve tending to $-\frac{2 p \eta_1-\zeta(2 p)}{2} \pm i \infty$, which is symmetric with respect to the real axis with $\sigma_{1, \infty} \cap \mathbb{R}=\{A_5\}$, and $\sigma_{1, A_2 A_3}$ is a simple curve connecting $A_2$ and $A_3$, and is symmetric with respect to the real axis with $\sigma_{1, A_2 A_3} \cap \mathbb{R}=\{A_6\}$. In particular, both $A_5$ and $A_6$ are branch points of $\sigma_1$. Moreover, $A_5=A_6$ if and only if $p\in \{p_c, \frac12-p_c\}$.
	\end{itemize}
\end{lemma}

\begin{proof}
Recall \eqref{eq7-16}
that
\begin{equation}\label{eq07-16}
	-\sigma_j(p)=\sigma_j\Big(\frac12-p\Big),\qquad j=1,3,\quad p\in \Big(0, \frac12\Big).
\end{equation}
This actually implies that
\begin{equation}\label{eq7-22}
\Delta(p)>0\quad\Longleftrightarrow \quad\Delta\Big(\frac12-p\Big)>0,
\end{equation}
$$\Delta(p)=0\quad\Longleftrightarrow \quad\Delta\Big(\frac12-p\Big)=0,$$
\begin{equation}\label{eq7-22-2}\Delta(p)<0\quad\Longleftrightarrow \quad\Delta\Big(\frac12-p\Big)<0.\end{equation}
Let us take $\Delta(p)>0$ for example. Then Lemma \ref{lem0-s5-13} says that the statement of Lemma \ref{lem0-s5-13}-(2) holds for $\sigma_1(p)$ with the correponding $A_5\neq A_6$. 
Consequently, it follows from \eqref{eq07-16} that the statement of Lemma \ref{lem0-s5-13}-(2) holds for $\sigma_1(\frac12-p)$ with the correponding $A_5\neq A_6$, so we conclude from  Lemma \ref{lem0-s5-13} that $\Delta(\frac12-p)>0$.

Recalling \eqref{deltap}, we define
$$g(x):=16x^4-128\eta_1 x^3+8g_2x^2+32(\eta_1g_2+g_3)x+g_2^2+32\eta_1g_3.$$
It follows from Lemmas \ref{Lemma53-6} and \ref{lem0-s5-13} that 
$$g\Big(\wp\Big(\frac14\Big)\Big)=\Delta\Big(\frac14\Big)<0.$$
Recalling $p_c\in (0,\frac14)$ defined in \eqref{eq-pc} and \eqref{eq-pc1}, we have
$$g(\wp(p_c))=\Delta(p_c)=0,\quad \wp(p_c)>\wp\Big(\frac14\Big).$$ 
Together with $\lim_{x\to+\infty}g(x)=+\infty$, we see that the number of zeros of $g(x)$ in $(\wp(\frac14),+\infty)$ is an odd number $k\geq 1$ by counting multiplicity. Clearly \eqref{eq7-22}-\eqref{eq7-22-2} imply that the number of zeros of $g(x)$ in $(e_1, \wp(\frac14))$ is also the same $k$ by counting multiplicity. Thus $2k\leq 4$ and so $k=1$. 

Therefore, 
$$g(x)\begin{cases}>0\quad\text{for }x\in (\wp(p_c),+\infty)\cup (e_1, \wp(\frac12-p_c)),\\
=0\quad\text{for }x=\wp(p_c), \wp(\frac12-p_c),\\
<0\quad\text{for }x\in (\wp(\frac12-p_c), \wp(p_c)),\end{cases}$$
or equivalently,
$$
\Delta(p)\begin{cases}>0\quad\text{for }p\in (0, p_c)\cup (\frac12-p_c, \frac12),\\
=0\quad\text{for }p=p_c, \frac12-p_c,\\
<0\quad\text{for }p\in (p_c, \frac12-p_c).\end{cases}$$
The proof is complete by applying Lemma \ref{lem0-s5-13}.
\end{proof}

We can also improve Lemma \ref{Lemma53-5}-(1) to a sharp version.

\begin{lemma}\label{lem-s5-20} Let $\tau=\frac12+ib$ with $b\geq \frac12$ and $p\in(0,\frac12)$. Recall Lemma \ref{Lemma53-5} that $\sigma_3=(-\infty, A_1]\cup [A_0, +\infty)\cup\sigma_{3,A_2A_3}$, where $\sigma_{3,A_2A_3}$ is a simple curve connecting $A_2$ and $A_3$, and is symmetric with respect to the real axis.
\begin{itemize}
\item[(1)] Recalling $p_{1,2}\in (0,\frac14)$ and $p_{0,2}=\frac12-p_{1,2}$ in Lemma \ref{7lemma-s5-9} for $b\in [\frac12, b_1)$, if
\begin{equation}
\label{eq7-25} p\in J:=\begin{cases}
(0,\frac12)\quad\text{for }b\geq b_1,\\
(p_{1,2}, p_{0,2})=(p_{1,2}, \frac12-p_{1,2})\quad\text{for }\frac12\leq b< b_1.
\end{cases}
\end{equation}
then $G_p(z)$ has a unique pair of nontrivial critical points $\pm a_0=\pm a_0(p)$ satisfying $\wp(a_0)\in\mathbb R$. Furthermore, $a_0\in II\cup III$, so the corresponding 
\begin{align}\label{eqfc-s5-a4}A_4:=&\frac{1}{2}\left[  \zeta(p+a_0)+\zeta(p-a_0)-\zeta(2p)\right]\\
=&A_0+\frac{\wp'(p)}{2(\wp(p)-\wp(a_0))}\in (A_1, A_0),\nonumber\end{align}
and $\sigma_{3,A_2A_3}\cap\mathbb{R}=\{A_4\}$. In particular, $\sigma_3$ has neither cusps nor branch points.
 \item[(2)] If $b\in [\frac12, b_1)$ and $p\in (0, p_{1,2}]\cup [p_{0,2},\frac12)$, then $\{\tilde{A}_4\}:=\sigma_{3,A_2A_3}\cap\mathbb{R}$ satisfies 
 \begin{equation}\label{eq7-26} \tilde{A}_4\begin{cases}= A_0\quad\text{if }p=p_{0,2},\\
 \in (A_0, +\infty)\quad\text{if }p\in (0,p_{0,2}),\\
 =A_1\quad\text{if }p=p_{1,2},\\
\in (-\infty, A_1) \quad\text{if }p\in (p_{1,2},\frac12),
 \end{cases}\end{equation}
i.e., $\tilde A_4$ is either a cusp or a branch point of $\sigma_3$.
\end{itemize}
\end{lemma}

\begin{proof}
(1). Under the assumption \eqref{eq7-25},
 it follows from those results in Sections 4-6 that $G_p(z)$ has a unique pair of nontrivial critical points $\pm a_0=\pm a_0(p)$ satisfying $\wp(a_0)\in\mathbb R$, and $a_0\in II\cup III$.
In particular, $\wp(a_0)<e_1<\wp(p)$, so $A_4$ defined by \eqref{eqfc-s5-a4} satisfies $A_4\in (A_1, A_0)$.
Since $A_4\in \sigma_1\cap\sigma_3\cap\mathbb R$ and $\sigma_{3,A_2A_3}\cap\mathbb{R}$ consists of a single point,  we conclude that $\{A_4\}=\sigma_{3,A_2A_3}\cap\mathbb{R}$.

(2). Let $b\in [\frac12, b_1)$ and $p\in (0, p_{1,2}]\cup [p_{0,2},\frac12)$. Then it follows from Theorem \ref{sec6tion-thm} that $G_p(z)$ has no nontrivial critical points satisfying $\wp(a)\in\mathbb{R}$, so
$$\sigma_1\cap\sigma_3\cap\mathbb R\setminus\{A_k\}_{k=0}^3=\emptyset.$$
Together with $[A_1, A_0]\subset\sigma_1$, it follows that $\{\tilde{A}_4\}:=\sigma_{3,A_2A_3}\cap\mathbb{R}$ satisfies $\tilde{A}_4\in (-\infty, A_1]\cup [A_0, +\infty)$, so $\tilde A_4$ is either a cusp or a branch point of $\sigma_3$. Using Corollary \ref{7coro-s5-10} and the continuity of $\sigma_3$ as $p$ deforms, we easily obtain \eqref{eq7-26}.
\end{proof}


Now we turn to prove Corollary \ref{III-coro}.
The next result gives a criterion for the non-degeneracy of the nontrivial critical points $\pm a_0$.

\begin{lemma}\label{Lemma531-05} Let $\tau=\frac12+ib$ with $b\geq \frac12$ and 
$$ p\in J:=\begin{cases}
(0,\frac12)\quad\text{for }b\geq b_1,\\
(p_{1,2}, p_{0,2})=(p_{1,2}, \frac12-p_{1,2})\quad\text{for }\frac12\leq b< b_1.
\end{cases}
$$
 Recall the unique pair of nontrivial critical points $\pm a_0$ of $G_p(z)$ satisfying $\wp(a_0)\in\mathbb R$ in Lemma \ref{lem-s5-20}, and recall $p_c\in (0, \frac14)$ defined in \eqref{eq-pc} and \eqref{eq-pc1}.  Then $\pm a_0$ are non-degenerate critical points of $G_p(z)$ if one of the following holds.
\begin{itemize}
\item[(1)] Either $p\in J\cap (p_c, \frac12-p_c)$;
\item[(2)] or $p\in J\cap( (0, p_c]\cup [\frac12-p_c, \frac12))$ and $A_4\notin\{A_5, A_6\}$, where $A_5, A_6$ are given in
Lemma \ref{lem2-s5-13} (2).\end{itemize} 
\end{lemma}

\begin{proof} Since $\wp(p), \wp(a_0)\in\mathbb{R}$ implies $\bar p=\pm p$ and $\bar a_0=\pm a_0$ in $E_{\tau}$, we see that
$$\alpha:=\frac{\wp(a_0+p)+\wp(a_0-p)+2\eta_1}{2}\in\mathbb R.$$
Recalling Part I \cite{CFL} that
\begin{align}\label{eqf0c-15}\det D^2G_p(a_0)
=\frac1{4\pi^2}\left(\frac{\pi^2}{b^2}-\left|\alpha-\frac{\pi}{b}\right|^2\right),\end{align}
 we see that to prove $\det D^2G_p(a_0)\neq 0$ is equivalent to prove
\begin{align}\label{al0pha}
\alpha\neq 0\quad\text{and}\quad\alpha\neq \frac{2\pi}{b}.
\end{align}

Since $p\in J$,
it follows from Lemma \ref{lem-s5-20}-(1) that $A_4$ is not a branch point of $\sigma_3$, so Lemma \ref{lemma2-10-3} implies $\alpha\neq \frac{4\pi i}{2\tau-1}=\frac{2\pi}{b}$.
 
 If $p\in J\cap (p_c, \frac12-p_c)$, then it follows from Lemma \ref{lem2-s5-13} (1) that $\sigma_1$ has no branch points. Thus $A_4$ is not a branch point of $\sigma_1$, so Lemma \ref{lemma2-10-3} implies $\alpha\neq 0$.
 
 If $p\in J\cap ((0, p_c]\cup [\frac12-p_c, \frac12))$, then it follows from Lemma \ref{lem2-s5-13} (2) that $\sigma_1$ has branch points only at $A_5, A_6$. Since we assume $A_4\notin\{A_5, A_6\}$, again $A_4$ is not a branch point of $\sigma_1$, so $\alpha\neq 0$.
 This completes the proof.
\end{proof}

For $p\in J\cap ((0, p_c]\cup [\frac12-p_c, \frac12))$, we have $A_4, A_5, A_6\in (A_1, A_0)$. Note that if $A_4\in\{A_5, A_6\}$ for some such $p$, then it follows from \eqref{eqfc-s5-43} that $\alpha=0$, i.e., $a_0$ is a degenerate critical point of $G_p(z)$. 

\begin{lemma}\label{lem75} Let $\tau=\frac12+ib$ with $b> b_1$. Then
$$\mathcal{D}:=\Big\{p\in (0, p_c]\cup \Big[\frac12-p_c, \frac12\Big) \,\Big|\, \text{$a_0$ is a degenerate critical point of $G_p(z)$}\Big\}$$
is at most a finite set.
\end{lemma}

\begin{proof} Let $p\in (0, \frac14)$, and 
recall the unique pair of nontrivial critical points $\pm a_0$ of $G_p(z)$ satisfying $\wp(a_0)\in\mathbb R$ in Lemma \ref{lem-s5-20}. Since $b>b_1$, we see from Lemma \ref{lemma-aup-2} that $a_0\in (z_0(\tau), \frac14+\frac{\tau}{2})\subset III\subset \frac12+i\mathbb R$, so $\overline {a_0}=1-a_0$. Then we see from \eqref{513-1-0} that
$$\wp(p)=\wp (a_0)+\frac{\wp ^{\prime }(a_0)}{%
2(\zeta(a_0)+(\frac{2\pi}{b}-\eta_1)a_0-\frac{\pi}{b})}.$$
Since $\zeta(z)$ and $\wp(z)$ are meromorphic functions with
$$\zeta(z)=\frac1z+c_1z^3+c_2z^5+\cdots,$$
$$\wp(z)=-\zeta'(z)=\frac1{z^2}-3c_1z^2-5c_2z^4+\cdots,$$
it follows that $a_0$ is an analytic function of $p\in (0, \frac14)$.

Note that $p_c\in (0, \frac14)$. For $p\in (0, p_c)$, we have $\Delta(p)>0$ and it follows from the proof of Lemma \ref{lem0-s5-13} that $\wp(a_5)\neq \wp(a_6)<e_1$ are two real distinct roots of
\begin{align*}
0=&4(\eta_1+\wp(p))\wp(a)^2+(4\wp(p)^2-8\eta_1\wp(p)-g_2)\wp(a)\\
&+4\eta_1\wp(p)^2-g_2\wp(p)-2g_3,
\end{align*}
so $a_5\neq a_6$ are both analytic functions of $p\in (0, p_c)$. Note from $\Delta(p_c)=0$ that $a_5=a_6$ at $p=p_c$.

Now since $b>b_1$, it follows from Theorem \ref{thm-5c} that $a_0$ is a non-degenerate critical point of $G_p(z)$ for $p>0$ small, so $A_4\notin\{A_5, A_6\}$, or equivalently $a_0\neq a_5, a_6$ for $p>0$ small. 
It follows from the property of analytic functions that
\begin{align*}&\Big\{p\in (0, p_c]\,\Big|\, \text{$a_0$ is a degenerate critical point of $G_p(z)$}\Big\}
\\=&\Big\{p\in (0, p_c]\,\Big|\, \text{$a_0=a_5$ or $a_0=a_6$}\Big\}
\end{align*}
is at most finite. Together with \eqref{eq07-16}, we obtain that 
$$\Big\{p\in \Big[\frac12-p_c, \frac12\Big) \,\Big|\, \text{$a_0$ is a degenerate critical point of $G_p(z)$}\Big\}$$
is also at most finite, so
$\mathcal{D}$ is at most a finite set.
\end{proof}

\begin{proof}[Proof of Corollary \ref{III-coro}]
Corollary \ref{III-coro} follows directly from Lemmas \ref{Lemma531-05} and \ref{lem75}.
\end{proof}

Finally, when $p\in II\cup III$, it seems challenging to improve Lemma \ref{Lemma53-7} to a sharp version (the precise characterization of $\sigma_1$ is simple by applying those results in Sections \ref{sec-4}-\ref{sec-6}, but the precise characterization of $\sigma_3$ seems very difficult). We should study this problem elsewhere.

\subsection*{Acknowledgements} Z. Chen was supported by National Key R\&D Program of China (No. 2023YFA1010002) and NSFC (No. 12222109). E. Fu was supported by NSFC (No. 12401188) and BIMSA Start-up Research Fund.

\subsection*{AI assistance statement} The authors used AI to find Lewy's non-vanishing Theorem \ref{thm-Lewy}.


\begin{thebibliography}{99}


\bibitem {BKLY}D. Bartolucci, A. Jevnikar, Y. Lee and W. Yang; \textit{Uniqueness of bubbling solutions of mean field equations}. J. Math. Pure Appl. \textbf{123} (2019), 78-126.

\bibitem{BT} D. Bartolucci and G. Tarantello; \textit{Liouville type equations with singular data and their applications to periodic multivortices for the electroweak theory}. Comm. Math. Phys. \textbf{229} (2002), 3-47.

\bibitem{BYZ} D. Bartolucci, W. Yang and L. Zhang; \textit{Asymptotic analysis and uniqueness of blowup solutions of non-quantized singular mean field equations}.
Math. Ann. \textbf{395} (2026), Paper No. 103.

\bibitem{BE} W. Bergweiler and A. Eremenko; \textit{Green's function and anti-holomorphic dynamics on a torus}.
Proc. Amer. Math. Soc. \textbf{144} (2016), 2911–2922.

\bibitem{BE-2010}  W. Bergweiler and A. Eremenko; \textit{On the number of solutions of a transcendental equation arising in the theory of gravitational lensing}.
Comput. Methods Funct. Theory \textbf{10} (2010), 303–324.

\bibitem {CLW}{ C.L. Chai, C.S. Lin and C.L. Wang; \textit{Mean field equations, hyperelliptic curves, and modular forms: I}. Camb. J. Math. \textbf{3} (2015), 127-274.}



\bibitem {CL-2} C.C. Chen and C.S. Lin; \textit{Topological degree for a mean field equation on Riemann surfaces}. Comm. Pure Appl. Math. \textbf{56} (2003), 1667-1727. 



\bibitem{CW-TAMS} H. Chen and M. Weber; \textit{
 An orthorhombic deformation family of Schwarz' H surfaces}.
Trans. Amer. Math. Soc. \textbf{374} (2021), 2057–2078.

\bibitem{CT-SIAM} H. Chen and M. Traizet; \textit{
Stacking disorder in periodic minimal surfaces}.
SIAM J. Math. Anal. \textbf{53} (2021), 855–887.

\bibitem{CFL} Z. Chen, E. Fu and C.S. Lin; \textit{Green functions, Hitchin's formula and curvature equations on tori}. arXiv:2508.17604v2. Submitted.

\bibitem{CFL-II} Z. Chen, E. Fu and C.S. Lin; \textit{Green functions, Hitchin's formula and curvature equations on tori II: Rectangular torus}. Proc. Lond. Math. Soc. \textbf{133} (2026), e70181.

\bibitem {CKL-JMPA2016}Z. Chen, T.J. Kuo and C.S. Lin; \textit{Hamiltonian
system for the elliptic form of Painlev\'{e} VI equation}. J. Math. Pures
Appl. \textbf{106} (2016), 546-581.


\bibitem {CKL1} Z. Chen, T.J. Kuo and C.S. Lin; \textit{The geometry of generalized Lam\'{e} equation, I}. J. Math. Pures
Appl.  \textbf{127} (2019), 89-120.


\bibitem{CL-JDG19} Z. Chen and C.S. Lin; \textit{Critical points of the classical Eisenstein series of weight two.} J. Differ. Geom. \textbf{113} (2019), 189–226.



\bibitem{FSX} Y. Feng, J. Song and B. Xu;
\textit{Existence and non-uniqueness of cone spherical metrics with prescribed singularities on a compact Riemann surface with positive genus.}
J. Geom. Anal. {\bf 36} (2026), Paper No. 58, 38 pp.



\bibitem {GW}F. Gesztesy and R. Weikard; \textit{Picard potentials and Hill's
equation on a torus}. Acta Math. \textbf{176} (1996), 73-107.

\bibitem{GGLY} F. Gladiali, M. Grossi, P. Luo and S. Yan;
\textit{Qualitative analysis on the critical points of the Robin function}.
J. Eur. Math. Soc. \textbf{27} (2025), 4713–4763.


\bibitem {Hit1} N.J. Hitchin; \textit{Twistor spaces, Einstein metrics and
isomonodromic deformations}. J. Differ. Geom. \textbf{42} (1995),
30-112.


\bibitem{Lewy}
H. Lewy;
\textit{On the non-vanishing of the Jacobian in certain one-to-one mappings},
Bull. Amer. Math. Soc. \textbf{42} (1936), 689--692.




\bibitem {LW} C.S. Lin and C.L. Wang; \textit{Elliptic functions, Green
functions and the mean field equations on tori}. Ann. Math. \textbf{172}
(2010), no.2, 911-954. 


\bibitem {LW4}C.S. Lin and C.L. Wang; \textit{On the minimality of extra
critical points of Green functions on flat tori}, Int. Math. Res. Not. \textbf{2017} (2017), 5591-5608.

\bibitem {LY}C.S. Lin and S. Yan; \textit{Existence of bubbling solutions for Chern-Simons model on a torus.} Arch. Ration. Mech. Anal. \textbf{207} (2013), 353-392.

\bibitem{MR} A. Malchiodi and D. Ruiz; \textit{New improved Moser-Trudinger inequalities and singular Liouville equations on compact surfaces}.
Geom. Funct. Anal. \textbf{21} (2011), 1196-1217.

\bibitem{WZ} J. Wei and L. Zhang; 
\textit{Laplacian vanishing theorem for a quantized singular Liouville equation}.
J. Eur. Math. Soc. \textbf{28} (2026), 237–267.


\bibitem{WWX} Z. Wei, Y. Wu and B. Xu; 
\textit{Geometric structure and existence of reducible spherical conical metrics.}
Math. Ann. \textbf{395} (2026), Paper No. 2, 48 pp.



\end{thebibliography}
\end{document}